\documentclass[11pt]{amsart} 
\usepackage{amscd}
\usepackage{amsfonts}
\usepackage[all]{xy}
\usepackage{amsmath,amsthm}
\usepackage{amsmath,amssymb,amsthm,latexsym}
\usepackage{amscd}
\usepackage{color}
\usepackage{hyperref}
\usepackage{color}
\usepackage{multicol}
\usepackage[all]{xy}
\usepackage{amsmath,amssymb,amsthm,latexsym}
\usepackage{amscd}
\usepackage{graphicx}
\usepackage{enumerate}
\newtheorem{theorem}{Theorem}[section]
\newtheorem{proposition}[theorem]{Proposition}
\newtheorem{definition}[theorem]{Definition}
\newtheorem{corollary}[theorem]{Corollary}
\newtheorem{lemma}[theorem]{Lemma}

\numberwithin{equation}{section}
\theoremstyle{remark}
\newtheorem{remark}[theorem]{Remark}

\newcommand{\R}{\mathbb{R}}
\newcommand{\la}{ \langle}
\newcommand{\ra}{ \rangle}
\newcommand{\C}{\mathbb{C}}
\newcommand{\D}{\mathbb{D}}

\newcommand{\RRR}{\mathbb{R}^{n+2}_1}

\newcommand{\VV}{\mathbb{V}}
\newcommand{\WW}{\mathbb{W}}
\newcommand{\DD}{\mathcal{D}}

 \newcommand{\dd}{\mathrm{d}}
 \newcommand{\db}{D_{\bar{z}}}
 \newcommand{\dz}{D_z}
  \newcommand{\pf}{\partial f}
\newcommand{\blue}{\textcolor{blue}}

\newcommand{\abs}[1]{|#1|}

\newcommand{\hp}{h^{\perp}}

\newcommand\Spc{\hbox{Span}_{\mathbb{C}}}
\newcommand{\wx}{\widehat{\xi}}

\def\dd{\mathrm{d}}
\def\LL{\mathbb{L}}
\newcommand{\TS}{\mathcal{T}S^{2m}}
\newcommand{\FS}{\mathcal{F}_1(S^{2m})}

\def\vec{\mathbf}
\def\<{\langle}
\def\>{\rangle}
\def\L{\mathcal{L}}
\def\I{\mathcal{I}}

\def\D{\mathcal{D}}
\def\X{\mathcal{X}}

\def\R{\mathbb{R}}
\def\C{\mathbb{C}}

\def\Q{\mathcal Q^{n+1}_1}

\def\OO{\omega}
\def\Re{{\rm Re}}
\def\Im{{\rm Im}}
\def\zb{\bar{z}}
\def\hm{\hat{m}}
\def\pj{{\perp_j}}
\def\ff{\mathbf{f}}
\def\FF{\mathcal{F}}

\def\ZZ{\mathcal{Z}}
\def\JJ{\mathcal{J}}

\def\nz{\nabla}
\def\nzz{\nabla_{z}}
\def\nzb{\nabla_{\zb}}
\def\pz{\phi}

\def\ee{\mathbf{e}}
\def\tf{\tilde{F}}
\def\tr{\tilde{r}}
\def\nt{\not\equiv}
\def\FH{\hbox{normal-horizontal}}
\begin{document}
\title[Classification of Willmore $2$-spheres in $S^n$]{\bf{ Classification of Willmore $2$-spheres in $S^n$}}
\author{Xiang Ma, Franz Pedit, Peng Wang}

\address{LMAM, School of Mathematical Sciences, Peking University, Beijing 100871, P.R. China. E-mail address: maxiang@math.pku.edu.cn}

\address{Department of Mathematics and Statistics, University of Massachusetts Amherst, Amherst, MA 01030, USA,  E-mail address: pedit@math.umass.edu}
\address{School of Mathematics and Statistics, Key Laboratory of Analytical Mathematics and Applications (Ministry of Education), FJKLAMA, Fujian Normal University, 350117 Fuzhou, P.R. China,  E-mail address: pengwang@fjnu.edu.cn, netwangpeng@163.com}


\date{\today}

\begin{abstract}
This paper resolves a long-standing open problem by providing a classification of Willmore $2$-spheres in $S^n$. We show that any such $2$-sphere is either totally isotropic--originating from the projection of a special ($\FH$) twistor curve in the twistor bundle over an even-dimensional sphere--or strictly $k$-isotropic, obtained via $(m-k)$ steps of adjoint transforms of a strictly $m$-isotropic minimal surface in $\R^n$, where $0\leq k\leq m\leq[n/2]-2$. Our approach hinges on the construction of a harmonic sequence in the real Grassmannian over the Lorentz space, derived from the harmonic conformal Gauss map of the original Willmore sphere. This sequence terminates finitely and generalizes, in part, the classical theory of harmonic sequences for harmonic $2$-spheres in complex Grassmannians.
\end{abstract}
\maketitle


\indent{{\bf MSC(2020)}:\hspace{2mm} 53A31, 53A10, 53C43,  53C28}

\tableofcontents

\section{Introduction}
It is a central theme in geometry to find the best shape of a given surface $M$ under certain variational principles and global conditions. For example, soap bubbles solve the isoperimetric problem in the sense that a closed surface with constant mean curvature  in $\R^3$ must be the round sphere if it is either a $2$-sphere by the work of Hopf, or embedded by the work of Alexandrov \cite{Hopf}. If we replace the condition constant mean curvature by $H=0$ and seek closed minimal surfaces in $S^n$, we have the famous classification of minimal $2$-spheres by Almgren in $S^3$ \cite{Al} and Calabi \cite{Calabi} in $S^n$ via horizontal twistor curves.

As a natural generalization in the M\"obius geometry of surfaces in $S^n$, we consider Willmore surfaces, the critical surfaces of the M\"obius invariant Willmore functional \[\int_M (H^2-K+1) \dd M,\] where $H$ is the mean curvature and $K$ its Gauss curvature. Now the classification problem should be considered with respect to the M\"obius transformation group.

In Bryant's fundamental work in 1984 \cite{Bryant1984}, a classification of Willmore $2$-spheres in $S^3$ is provided. This was generalized by Ejiri \cite{Ejiri}, Musso \cite{Musso} and later in 2000s by Montiel \cite{Montiel} and Burstall etc. \cite{BFLPP} for the case $n=4$ using different methods. In particular all Willmore 2-spheres in $S^4$ are S-Willmore, i.e. with a dual Willmore surface. In 2017,  Ma-Wang-Wang \cite{MWW} proved that any Willmore $2$-sphere in $S^5$ which is not S-Willmore, must be the adjoint surface of some special minimal surface in $\R^5$, and a concrete example is constructed. This gives a classification of Willmore $2$-spheres in $S^5$.  We refer to \cite{Ejiri, MPW, Mi-Ri}, for related results under various conditions.

In this paper, we present a complete classification of all Willmore $S^2$ in $S^n$.

\begin{theorem}\label{thm-uniqueness}
Let $y:S^2\to S^n$ be a Willmore immersion.
Then $y$ belongs to one of the following cases.
\begin{enumerate}
\item It is totally isotropic and comes from the projection of a  $\FH$, twistor curve of the twistor bundle $\mathcal{T}S^{2\tilde n}$ into $S^{2\tilde n}\subset S^n$. See Section \ref{sec-twistor}.
\item It is M\"obius congruent to some minimal surface with embedded planar ends in $\R^n$.
\item It is strictly $k$-isotropic and comes from $(m-k)$-step adjoint transforms of some strictly $m$-isotropic (branched) minimal surface in $\R^n$.

\end{enumerate}
\end{theorem}

This result directly generalizes the classification from $S^5$ \cite{MWW} to $S^n$, which involves much more  geometric structures, such as the treatment of various isotropy orders within the light cone model of M\"{o}bius geometry.
For a conformal immersion $y:M\rightarrow S^n$ with a local complex coordinate $z$ on $U\subset M$, it has a canonical lift \cite{BPP} (see also Section 2):
\[
Y: U\rightarrow \mathcal L^+\subset \RRR
\]
Here $\mathcal L^+$ is the future light-like cone of the Lorentz-Minkowski space $\RRR$. Then $y$ is called $j$-isotropic if
\[
\<Y_{z},Y_{z}\>=\cdots=\<Y^{(j+1)}_{z},Y^{(j+1)}_{z}\>\equiv0, ~~~~\text{for}~j\geq0
\]
Strictly $j$-isotropic means $j$-isotropic and not $(j+1)-$isotropic; $y$ is called totally isotropic if $\<Y^{(j+1)}_{z},Y^{(j+1)}_{z}\>\equiv0$ for all $j\geq0$. Note that $0$-isotropic amounts to the usual conformality.

Case (1) can be viewed as the conformal version of Calabi's famous works on minimal $2$-spheres \cite{Calabi}. Note also that the projection of a $\FH$ twistor curve is always a (possibly branched) totally isotropic Willmore surface with a special normal bundle structure.
It seems that $\FH$ twistor curves have nice geometric properties and may hold independent interests. We refer Section \ref{sec-twistor} for more details.
\vspace{2mm}



Case (2) corresponds to Ejiri's theorem \cite{Ejiri}, asserting that a strictly $k$-isotropic, S-Willmore $2$-sphere $y^k=[Y^k]$ is M\"obius congruent to a minimal surface with embedded planar ends in $\R^n$.\vspace{2mm}

Case (3) represents the most substantial challenge. This case involves new transforms and considerable constructive efforts, the background for which has been summarized in the Ma-Wang-Wang's 2017 paper \cite{MWW}. In the subsequent discussion, we review the basic concepts related to these transforms, highlight the central difficulties, and we therefore focus on the novel elements introduced to overcome these obstacles.

\subsection{The Willmore sequence and the real harmonic sequence}

The previous study of Willmore $2$-spheres in $S^3$ and $S^4$ relied on the existence of a dual Willmore surface--conformal and sharing the same conformal Gauss map \cite{Bryant1984, Ejiri}--a property exclusive to S-Willmore surfaces. The so-called \emph{2-step forward and backward B\"acklund transforms} in \cite{BFLPP} gives a nice generalization of the duality theorem, which is however subject to two key limitations: it is defined only in $S^4$, and it relies essentially on the quaternionic structure specific to codimension 2.

When the codimension is arbitrary, the first author extended the notion of 2-step B\"{a}cklund transformations by introducing the so-called \emph{adjoint transform} for any Willmore surface \cite{ma0,ma2} (see the definition therein, or Section \ref{sec-adj}). These adjoint transforms remain conformal and Willmore. They always exist locally, though in general without uniqueness. A remarkable observation is that when the original Willmore 2-sphere $y^k$ is strictly $k$-isotropic, certain forms must be holomorphic and vanishes on this 2-sphere, which helps to define a canonical adjoint transform $y^{k+1}$ being strictly $(k+1)$-isotropic. If we can apply this construction inductively, the increasing isotropy order, together with the finite dimension of the ambient space, force this sequence to stop at some euclidean minimal surface $y^m, m>k$ and then a single point:
\[
y^k \to y^{k+1} \to \cdots\to y^m \to y^{m+1}=\{pt\}.
\]
Such a reduction recipe appeared in several beautiful works before \cite{Calabi, Chern, Chern-W, Leschke}.

Despite strong evidence for the general truth of the main theorem, we encountered the unavoidable challenge of singularities. Such adjoint transforms might produce terms approaching infinity at some points (possible poles). In other words, the new adjoint surfaces might only be branched immersions. Yet our moving frame method (see Section~2 or \cite{BPP}) treats only the regular immersion case and hard to cope with the singular terms. Without control of the branched points/singularities, we can NOT apply the vanishing theorems for holomorphic forms on 2-spheres. Thus the classification was unfinished in the first author's 2005 PhD thesis \cite{ma0}.

To overcome such obstacles, one needs new ideas to analyze the conformal Gauss map as well as other sub-bundle structures arising as harmonic maps into certain Grassmannians over the Lorentz-Minkowski space $\RRR$. In the follow up work of 2017 by Ma-Wang-Wang \cite{MWW}, for a Willmore 2-sphere in $S^5$, we constructed a generalized Gauss map  $\xi$ into the de-sitter space $S^6_1$ which is again conformal harmonic. Using $\xi$ and its derivatives $\xi_z, \xi_{zz}$ we re-constructed the Willmore sequence $y^0\to y^1 \to y^2=\{pt\}$ as desired. This accomplished our goal in $S^5$.\vspace{3mm}

From a theoretic viewpoint, underlying such a Willmore sequence, we are working with a sequence of harmonic maps into certain real Grassmannian manifolds $Gr_r(\RRR )$ consisting of rank-$r$ real space-like subspaces (or equivalently, their orthogonal complements as Lorentz subspaces). The construction needs to take complex derivative $\partial/\partial z$ , then take an orthogonal projection and choose isotropic subspaces. This mimics the well-known harmonic sequence in Chern-Wolfson's work \cite{Chern-W}. There are two main differences: Firstly, the ambient space is endowed with a pseudo-Euclidean metric and the Grassmannian is a non-compact symmetric space. Secondly, we need to come back to real subspaces/subbundles to obtain the harmonic maps in each step.

This motivated our definition of \emph{$\D f$ transform} (see Section~\ref{subsect-harmonic}), which produces a harmonic sequence from the initial conformal Gauss map $f^k_0$ of a strictly $k$-isotropic Willmore 2-sphere $y^k$. The strictly isotropic conditions force this sequence to end with a final $f^k_{k+k'}=(h_0)^{\perp}$. This $h_0$ is quite similar to the generalized Gauss map $\xi:S^2 \to S^6_1$ in our 2017 paper \cite{MWW}. To be concrete, it is a conformal minimal branched immersion of 2-sphere into a real Grassmannian which is proven to be totally isotropic and reduces to a Euclidean minimal surface. This inductive and reductive procedure solves the first difficulty.

In the process above we need to construct new isotropic frames repeatedly like the Gram-Schmidt orthogonalization. Previously, certain terms in these computations might diverge to infinity and cause singularity disaster. But upon close examination, we discovered (Lemma~\ref{lemma-linebundle}) that the singular terms will vanish automatically after taking wedge product with old frames (i.e. the new bundle mapping is still regular). This helps us to overcome the second crucial problem of singularities.

After these two key steps, a series of new holomorphic forms spring up. Using the global assumption of 2-spheres, we know these forms vanish and the involved frames are isotropic and orthogonal. This guarantees the construction of the harmonic sequence $f^k_0\to f^k_{k+k'}=(h_0)^{\perp}$.  But how to relate the original Willmore $f^k$ and the final Euclidean minimal surface $y^m$ coming from the differentiation of $h_0$? This is the third difficulty.

Recall that the Willmore sequence conjectured as before comes from the adjoint transform in a recursive manner. Each desired adjoint surface pair $\{y^j, y^{j+1}\}$ for any $j~ (k\le j\le m)$ must lies on the mean curvature sphere  (exactly the conformal Gauss map) of $y^j$ . In the light-cone model, that means their light-line representation
$\{Y^j, Y^{j+1}\}$ are contained in the 4-dim Lorentz subspace $f^j_0$. Our key observation is that
\[
\hbox{Span}_{\R}\{Y^{j},Y^{j+1}\}=f^j_0\cap h_{j+1}^{\perp}
\]
where $h_{j+1}$ comes from the isotropic $(j+1)$'s jet bundle of $h_0$.

The discovery of this intersection relationship solves the third difficulty. But proving it in a unified way is a huge technical challenge. We need to analyze the bundle structure of $(f^j_0)^\perp$ as well as $h_0\subset h_1 \subset h_2 \subset \cdots $ in an inductive manner.  We also want to treat various isotropic situation in this bundle $(f^j_0)^\perp$ without discussing infinitely many cases. At the same time, we have to show that $h_0$ is a full harmonic map into the Grassmannian using an argument of contradiction. For this purpose, the proof in Section~9.3 uses a double induction method. We also introduce the so-called \emph{Frenet-bundle structure} in Section~\ref{subsect-harmonic} which simplifies the statement and the proof in the induction argument. (Note that the Frenet-bundle structure also appears in the definition of the $\D$-transform). This overcomes the fourth and the final barrier and accomplishes the whole proof.\vspace{3mm}

As a summary of the results in Case (3), when the Willmore 2-sphere $y^k=[Y^k]$ is strictly $k$-isotropic and NOT S-Willmore, we show that:
\begin{enumerate}
    \item It admits a {\em global } sequence of harmonic maps into various real Lorentzian Grassmannians $\{f^k_{l},1\leq l\leq k+k'\}$
\[
\begin{array}{cccccccccccccc}
     f^k_0 & \Rightarrow    &  f^k_1 & \Rightarrow& \cdots &\Rightarrow  & f^k_k &\Rightarrow & f^k_{k+1} &\Rightarrow & \cdots & \Rightarrow & f^k_{k+k'}=(h_0)^{\perp}
     \end{array}
    \]
     They are maps into various real Grassmannians, with $f^k_0$ being the conformal Gauss map of $[Y^k]$ and $Rank_{\R}(f^k_{k+k'})=2m+4$. Here either $k'=0$ and $m\leq 2k$, or $m=2k+k'\geq 2k$. 

\item The harmonic map $h_0:=(f^k_{k+k'})^{\perp}$ from $S^2$ has $Rank_{\R}(h_0)=r$ and $Rank_{\C}(\partial h_0)=1$. Moreover, $h_0$ is totally isotropic, which defines a Frenet structure on $h_0^{\perp}$, in the same way that appears in \cite{Bar,Bryant1984,Calabi,Chern, Ejiri, MWW}. This induces a second  sequence of harmonic maps $\{h_j,1\leq j\leq m+1\}$ defined on open dense subsets of $S^2$ with $Rank_{\R}(h_j)=r+2j$:
\[
\begin{array}{ccccccccccccc}
     h_0&\Rightarrow    &  h_1 &\Rightarrow  &\cdots &\Rightarrow & h_{k} &\Rightarrow &\cdots &\Rightarrow  & h_m &\Rightarrow & h_{m+1}.
     \end{array}
    \]

\item We derive the adjoint Willmore sequence $\{Y^j,k\leq j\leq m\}$ inductively by
\[\hbox{Span}_{\R}\{Y^{j},Y^{j+1}\}=f^j_0\cap h_{j+1}^{\perp}, \hbox{ for } k\leq j\leq m.\]
In other words, we have the following diagram:
\[\xymatrix@C=12ex{
        Y^k \ar@{.>}[r]^{AT} \ar[d]_{CGM} & \ar@{.>}[l] Y^{k+1} \ar[d]_{CGM}   \ar@{.>}[r]^{AT}  & \ar@{.>}[l] \cdots \ar@{.>}[r]^{AT} \ar[d]_{CGM} & \ar@{.>}[l]  Y^m \ar@{.>}[r]^{AT} \ar[d]_{CGM} & \ar@{.>}[l] Y^{m+1} \\
        f^k_0  \ar@<-.5ex>[u]_{~~\cap h_{k+1}^{\perp}}  \ar[ur]& f^{k+1}_0  \ar@<-.5ex>[u]_{~~\cap h_{k+2}^{\perp}}  \ar[ur] &      \cdots \ar@<-.5ex>[u]_{~~\cap h_{\cdots}^{\perp}}  \ar[ur]& f^{m}_0  \ar@<-.5ex>[u]_{~~\cap h_{m+1}^{\perp}}  \ar[ur]
    }\]
Here $f^j_0$ is the conformal Gauss map ($CGM$ for short in the diagram) of $Y^j$, and $Y^{j+1}$ is an adjoint transform ($AT$ for short) of $Y^{j}$.

\item   $[Y^m]$ is a strictly $m$-isotropic minimal surface in $\R^n$ and $[Y^{m+1}]$ is a constant point. In fact, the sequence $\{h_j,1\leq j\leq m+1\}$ can also be recovered from the frame of the strictly $m$-isotropic minimal surface $Y^m$, which is therefore very simple. Together with the inductive using of the conformal Gauss map $f^j_0$ of $Y^j$ each time, it leads to the construction of the Willmore pair $Span_{\R}\{Y^j,Y^{j+1}\}=f^j_0\cap h_{j+1}^{\perp}$ in a natural way.

\item The global assumption of $2$-sphere is used to show certain holomorphic forms vanish, which yields isotropic sub-bundles and defines two real harmonic sequences:
 the global one $\{f^k_l,0\leq l\leq k+k'\}$ defined on $S^2$, and the one $\{h_j,1\leq j\leq m+1\}$ defined on open dense subsets of $S^2$.
\end{enumerate}

\begin{remark}
 As to the geometric meaning of $\{f^k_l\}$,  $f^k_0$ denotes the mean curvature $2$-sphere of $y^k$ and $f^k_l$ is the unique global sphere congruence containing the mean curvature $2$-spheres of the surfaces which come from the $j$-step adjoint transforms of $y^k$, $j\leq l$. It is amazing that, although adjoint surfaces might have singularities, the map $f^k_l$ is globally defined on $S^2$, which is the key step to overcome the possible singularities. In this sense, we can view $f^k_l$ the $l$-th order conformal Gauss map of $y^k$. And the final map $f^k_{k+k'}$ denotes the unique global sphere congruence which contains all the possible surfaces, as well as their mean curvature $2$-spheres, coming from multiple steps of adjoint transforms of $y^k$.
\end{remark}

\subsection{Other related works}
\begin{remark}
The work of Uhlenbeck \cite{Uh} and later Burstall and Guest's generalization \cite{Bu-Gu}, provided a classification of all harmonic $2$-spheres in compact Lie groups in terms of integrable system methods. Willmore surfaces have  harmonic conformal Gauss maps, with the target manifold being a non-compact symmetric space with a pseudo-Riemannian metric. So one cannot simply apply the ideas of \cite{Uh,Bu-Gu} for Willmore $2$-spheres. Moreover, there is another key difficulty in figuring out geometric properties of surfaces or harmonic maps from the loop group data.  In \cite{DoWa1,DoWa12}, Dorfmeister and the third author used the DPW method \cite{DPW} to study a Willmore surface via its harmonic conformal Gauss map and obtained the first examples of totally isotropic Willmore $2$-spheres in $S^6$ which are not S-Willmore. A loop group description of totally isotropic Willmore $2$-spheres in $S^6$ is also derived in \cite{Wang-3}. Moreover,  with the help of the duality theorem for harmonic maps \cite{DoWa13,DoWa-fu}, one can adapt the works of Uhlenbeck \cite{Uh}, Burstall-Guest \cite{Bu-Gu,Gu2002} so that they can be applied for Willmore $2$-spheres. As applications, a loop group description of the potentials of all Willmore $2$-spheres has been derived in \cite{Wang-1}, while a geometric description along this way stays still unclear. Therefore, we choose the present geometric treatment for Willmore $2$-spheres.
\end{remark}

\begin{remark} For the case of a torus $T^2$ in $S^3$, we have the famous Willmore conjecture that  the Clifford torus attains the minimum of the Willmore functional uniquely among all $2$-tori in $S^3$ up to a  M\"obius transformation, which is solved by the celebrated work of Marques and Neves\cite{Marques}. And Willmore tori are so abundant that we can only expect classification results under stronger conditions \cite{BB,FP,Pinkall}.

We also refer to \cite{Be-Ri, Ku-Li, Ku-Sch, Ri} for the compactness theorems concerning the moduli space of embedded Willmore surfaces with small energy. Moreover, for the case of higher genus, the embedded Lawson minimal surface $\xi_{g,1}$ \cite{Lawson} is conjectured to be the unique candidate of Willmore minimizers with genus $g$ \cite{Kusner1989}, which remains an open problem with some progress \cite{BK,KW,KLW,KW2,Ku-Li-Sc,Simon93}.
\end{remark}

\subsection{Organization of the paper}
In Section 2, we briefly review the Willmore surfaces and adjoint transforms. In Section 3, we recall the $\partial$-transform defined by Chern and Wolfson, and introduce $\DD$-transforms for special harmonic $2$-spheres in a real Grassmannian. These transforms play crucial roles in this paper and hold independent interests. In Section 4, we construct natural isotropic holomorphic bundles associated with an isotropic Willmore $2$-sphere.  Section 5 characterizes totally isotropic Willmore $2$-spheres in terms of $\FH$ twistor curves. Sections 6 through 8 are devoted to the classification of strictly $k$-isotropic Willmore $2$-spheres that are not S-Willmore. Section 6 presents the main theorem and outlines the proof strategy. Section 7 is devoted to the construction of isotropic line bundles within the conformal normal bundles of a Willmore $2$-sphere, leading to a universal harmonic map $h_0$ and a simple framing of it. Then in Section 8, a new harmonic sequence $\{h_j\}$ is derived from the universal harmonic map $h_0$, which also provides a simple framing of $h_0$. In Section 9, we derive the adjoint Willmore sequence $\{Y^j, k \leq j \leq m\}$, by adapting the framing of the Willmore $2$-sphere and the ascending adjoint transforms according to the framing of $h_0$ inductively, which gives the proof of the main theorem. In Section 10, we briefly outline the construction of examples in the classification theorem. Examples of $k$-isotropic minimal surfaces with embedded planar ends are constructed explicitly.
Section 11 concludes the paper with a proof of Lemma \ref{lemma-linebundle}.

\section{Willmore surfaces and adjoint transforms}

In this section, we will briefly review the M\"obius geometry of Willmore surfaces and adjoint transforms. For more details, we refer to \cite{BPP, ma2, MWW, WangCP}.

\subsection{Surface theory in M\"obius geometry}

Let $\R^{n+2}_1$ be the $(n+2)$-dimensional Lorentz-Minkowski space with the Lorentzian metric
\[\<Y,Y\>=-Y_0^2+\sum_{i=1}^{n+1}Y_i^2.\]
Throughout this paper, we use $\<\ ,\ \>$ to denote the above inner product and its complex linear expansion.
Let $\L\subset \R^{n+2}_1$ be the lightcone and let $\L^+$ be the future lightcone with $Y_{0}>0$.
We identify the unit sphere $S^n\subset \R^{n+1}$ with the projectivized light cone via
\[
S^n\cong \mathbb{P}(\L):~~y\leftrightarrow [(1,y)]=[Y].\]
The projective action of the Lorentz group $O(1,n+1)$ on $\mathbb{P}(\L)$ provides all conformal diffeomorphisms of $S^n$. The following correspondences are well known \cite{BPP}:
\begin{itemize}
\item A point $y\in S^n$ $~~~~~~~~\leftrightarrow~~~~~~~$ a lightlike line $[Y]\in \mathbb{P}(\L)$;
\item A $k$-dim sphere $\sigma\subset S^n$ $~~~~~~\leftrightarrow~~~~~~~$ a space-like $(n-k)$-dim subspace $\Sigma\subset \R^{n+2}_1$;
\item The point $y$ locates on the sphere $\sigma$ $~~~\leftrightarrow~~~$ $Y\bot\Sigma$.
\end{itemize}

Let $y:M\to S^n$ be a conformal immersion from a Riemann surface $M$.
A local lift is just a map $Y$ from $M$ to the light cone $\L$ such that $[(1,y)]=[Y]$.
Taking derivatives w.r.t a local complex coordinate $z$,
we have $\<Y_z,Y_z\>=0$ and $\<Y_z,Y_{\zb}\> >0$, since $y$ is a conformal immersion. Moreover, there exists a 4-dimensional Lorentz subspace of $\R^{n+2}_1$ defined at every point $p\in M$:
\[
V|_p:={\rm Span} \{ Y,\Re (Y_z),\Im (Y_z), Y_{z\zb}\}_{p},
\]
which is independent of the choice of  $Y$ and  $z$. By the correspondence given above, $V$ describes a M{\"o}bius invariant geometric object, called the \emph{mean curvature 2-sphere} of $y$. This name comes from the property that it is the unique $2$-sphere tangent to the surface and having the same mean curvature vector as the surface at the tangent point when the ambient space is endowed with a metric of some Euclidean space (or any other space form). The corresponding map from $M$ into the Grassmannian $Gr_{3,1}(\RRR)$ (which consists of 4-dimensional Lorentz subspaces) $f:M\rightarrow Gr_{3,1}(\RRR)$ is the so-called \emph{conformal Gauss map} \cite{Bryant1984, Ejiri}.

Given a local coordinate $z$, there is a canonical
lift into $\mathcal L^+$, determined by $\abs{\dd Y}^2=\abs{\dd z}^2$.  Assume $Y$ is a canonical lift. Then a canonical frame of $V\otimes\C$ is given as
\begin{equation}
\label{frame}
\{Y,Y_z,Y_{\zb},N\},
\end{equation}
where we choose the unique $N\in V$ with
$\< N,N\>=0,~\< N,Y\>=-1,~\< N,Y_z\> =0.$
These frame vectors are orthogonal to each other except that
$\< Y_{z},Y_{\zb} \> = \frac{1}{2},~\< Y,N \> = -1.$ Let $\xi\in\Gamma(V^{\bot})$ be an arbitrary section of
the normal bundle $V^{\bot}$; $D$ is the normal connection.
The structure equations are as follows:
\begin{equation}
\label{mov-eq}
\left\{
\begin{array}{llll}
Y_{zz} &=& -\frac{s}{2} Y + \kappa, \\
Y_{z\zb} &=&
-\<\kappa,\bar{\kappa}\> Y + \frac{1}{2}N, \\
N_{z} &=& -2 \<\kappa,\bar{\kappa}\> Y_z
- s Y_{\zb}+ 2D_{\zb}\kappa, \\
\xi_z &=& D_z\xi + 2\<\xi,D_{\zb}\kappa\> Y
- 2\<\xi,\kappa\> Y_{\zb}.
\end{array}
\right.
\end{equation}
The first equation above defines
two basic M{\"o}bius variants: the \emph{Schwarzian} $s$,  and the section $\kappa\in \Gamma(V^{\bot}\otimes\C)$ which may be identified with the normal-valued Hopf differential up to scaling.
Later we will need the fact that $\kappa$ vanishes exactly at the umbilic points, and that $\kappa\frac{\dd z^2}{\abs{\dd z}}$
is a globally defined differential form. See \cite{BPP} for more details.

The conformal Gauss, Codazzi and Ricci equations
as integrability conditions are
\begin{eqnarray}
& \frac{1}{2}s_{\zb} = 3\< D_z\bar{\kappa},\kappa\>
+ \<\bar{\kappa},D_z\kappa\>, \label{gauss}\\[.1cm]
& \Im ( D_{\zb}D_{\zb}\kappa
+ \frac{\bar{s}}{2}\kappa ) = 0, \label{codazzi}\\[.1cm]
& R_{\zb z}^D\xi
:= D_{\zb}D_z\xi - D_z D_{\zb}\xi
= 2\<\xi,\kappa\>\bar{\kappa}
- 2\<\xi,\bar{\kappa}\>\kappa. \label{ricci}
\end{eqnarray}

\subsection{Review of Willmore surfaces}
There is a well-defined metric over $M$ invariant under M\"obius transformations, which is also conformal to the original metric induced from $y:M\to S^n$, called the \emph{M\"obius metric} \cite{Bryant1984, BPP, WangCP}:
\[
e^{2\omega}\abs{\dd z}^2=4\<\kappa,\bar\kappa\>\abs{\dd z}^2.
\]
It is well known that this metric coincides with the induced metric of the conformal Gauss map. The area of $M$ with respect to the M\"obius metric
\[
W(y):=2i\cdot\int_M \<\kappa,\bar\kappa\> \dd z\wedge\dd\zb
\]
is exactly the famous \emph{Willmore functional}.
It coincides with the usual definition
\[\widetilde{W}(x):=\int_M (|\vec{H}|^2-K)\dd M\]
for an immersed surface $x$ in $\R^n$ with mean curvature $\vec{H}$ and Gauss curvature $K$. The surface $y$ is called a \emph{Willmore surface} if it is the critical surface with respect to $W(y)$. Willmore surfaces are characterized as follows:
\begin{theorem}\label{th-Willmore}\cite{blaschke,Bryant1984,Ejiri} $y$ is Willmore if and only if one of the following two conditions holds
\begin{enumerate}
\item $y$ satisfies the \emph{Willmore equation} \cite{BPP}:
\begin{equation}
\label{willmore}
D_{\zb}D_{\zb}\kappa+\frac{\bar{s}}{2}\kappa = 0.
\end{equation}
\item The conformal Gauss map $f=2i(Y\wedge Y_z\wedge Y_{\bar z}\wedge N)$ of $y$ is harmonic.
\end{enumerate}
\end{theorem}
Since Willmore surfaces satisfy an elliptic equation, Morrey's result (see \cite{Bryant1984} and Lemma~1.4 of \cite{Ejiri}) guarantees that the related geometric quantities are real analytic.

By the Willmore equation, we have the following Calabi-Hopf type normal bundle valued global holomorphic differential, which is already known by Ejiri \cite{Ejiri}.
\begin{proposition} \cite{Ejiri} For a Willmore surface $y$, the differential
\begin{equation}\label{eq-chi}
\chi_{0}\dd z^3:=D_{\zb}\kappa\wedge\kappa \dd z^3
\end{equation}
is a global holomorphic differential.
\end{proposition}
A special type of Willmore surfaces are those with $\chi_{0}\equiv0$, that is, S-Willmore surfaces  \cite{Ejiri, ma0}.
\subsubsection{Dual surfaces and S-Willmore surfaces}

In the codim-1 case, Bryant \cite{Bryant1984} noticed that when $y:M\to S^3$ is Willmore, there is always a dual conformal Willmore surface $\hat{y}=[\widehat{Y}]:M\to S^3$ enveloping the same mean curvature spheres when $\hat{Y}$ is immersed. Note that in this case, except the umbilic points,
$D_{\zb}\kappa$ depends linearly on $\kappa$, thus locally there is a function $\mu$ such that
\begin{equation}
\label{swillmore}
D_{\zb}\kappa +\frac{\bar\mu}{2}\kappa =0.
\end{equation}
When the codimension increases, Ejiri first noticed that the duality theorem does not hold for all Willmore surfaces and \eqref{swillmore} does not hold for all Willmore surfaces. Moreover, he proved that a Willmore surface admits a dual surface if and only if \eqref{swillmore} holds for some $\mu$. We will call such surfaces \emph{S-Willmore surfaces}\footnote{
In this definition, S-Willmore surfaces include all codimension one Willmore surfaces, which are not included in Ejiri's original definition \cite{Ejiri}. Here we slightly revise the definition.}.
Examples of S-Willmore surfaces include Willmore surfaces in $S^3$, superconformal surfaces in $S^4$,
and minimal surfaces in space forms $\R^n, S^n, H^n$.

The dual surface $\hat{y}=[\widehat{Y}]:M\to S^n$ can be written explicitly as \cite{ma1,ma2}
\begin{equation}
\label{yhat}
\widehat{Y}=\frac{1}{2}\abs{\mu}^2 Y
+\bar\mu Y_z +\mu Y_{\zb} +N
\end{equation}
with respect to the frame $\{Y,Y_z,Y_{\zb},N\}$ depending on some undetermined factors $\mu$.
Calculation using \eqref{mov-eq}, \eqref{willmore} and
\eqref{swillmore} yields
\begin{equation}
\label{yhat-z}
\widehat{Y}_z=\frac{\mu}{2}\widehat{Y}
   + \rho\left(Y_z + \frac{\mu}{2}Y\right),~~
   \text{where}~\rho:=\bar\mu_z-2\<\kappa,\bar\kappa\>.
\end{equation}
Here $\rho\abs{\dd z}^2$ is an invariant
associated with $\{y,\hat{y}\}$ \cite{ma2} independent of the choice of complex coordinate $z$.
It follows that
\begin{equation}
\label{yhat-metric}
\<\widehat{Y}_z,\widehat{Y}_z\>=0, \quad
\<\widehat{Y}_z,\widehat{Y}_{\zb}\>
=\frac{1}{2}\abs{\rho}^2.
\end{equation}
It is direct to verify that $\hat{y}$ shares the same mean curvature sphere as $y$ at the points where $\hat{y}$ is immersed. By Theorem \ref{th-Willmore}, $\hat{y}$ is also a conformal Willmore immersion into $S^n$ when $\rho\ne 0$.

\begin{remark}\label{rem-minimal}
When $\rho\equiv 0$, by \eqref{yhat-z}, $[\widehat{Y}]$ is a fixed point in $S^n$.
The stereographic projection maps with respect to $[\widehat{Y}]$ maps the $2$-spheres containing $[\widehat{Y}]$ to two-planes in $\R^n$. Since $[\widehat{Y}]$ is contained in the mean curvature $2$-sphere of $y$, this stereographic projection maps $y$ to a surface in $\R^n$ with mean curvature vector $\vec{H}\equiv0$, that is, $y$ is M\"obius equivalent to some minimal surface in $\R^n$. Moreover, the ends of the minimal surface must be planar ends. We refer to \cite{Bryant1984} for more details and \cite{Bryant1988,kusner,PX} for the construction of such surfaces.
\end{remark}

\subsubsection{The adjoint transform}\label{sec-adj}

 A Willmore surface in $S^n$, $n\geq4$, does not necessarily have a dual surface. To generalize the notion of dual Willmore surfaces, in \cite{ma0,ma2} the first named author introduced the so-called \emph{adjoint transforms} for arbitrary Willmore surfaces. Given a Willmore surface $y:M\to S^n$, its adjoint transform is characterized as a conformal (branched) immersion $\hat{y}:M\to S^n$ such that the corresponding point $\hat{y}(z)$ is contained in the same mean curvature sphere of $y$ at $z$, at the same time $\hat{y}$ \emph{co-touches} this sphere \cite{ma0, ma2}.

Since the corresponding point $\hat{y}=[\hat{Y}]$ is still contained in the same mean curvature sphere as $y$, we can express $\hat{Y}$ in the same way as \eqref{yhat}, depending on a local function $\mu$:
\begin{equation}
\label{yhat2}
\hat{Y}=\frac{1}{2}\abs{\mu}^2 Y
+\bar\mu Y_z +\mu Y_{\zb} +N,
\end{equation}
And now we have
\begin{equation}
\label{yhat-z2}
\hat{Y}_z=\frac{\mu}{2}\hat{Y}
   + \rho\left(Y_z + \frac{\mu}{2}Y\right)+\theta\left(Y_{\bar{z}} + \frac{\bar\mu}{2}Y\right)+2\eta,
\end{equation}
where
\begin{equation}\label{eq-eta}
\rho:=\bar\mu_z-2\<\kappa,\bar\kappa\>,~~
\eta:=D_{\zb}\kappa +\frac{\bar\mu}{2}\kappa, ~~\theta: =\mu_z-\frac{1}{2}\mu^2-s.
\end{equation}

\begin{definition}\label{def-adj}\cite{ma2}
    We call $\hat{y}=[\hat{Y}]$ in \eqref{yhat2} an adjoint tranform/surface of $y=[Y]$, if $\mu$ satisfies
     the \emph{co-touch} condition and the conformal condition:
\begin{align}
\theta&=\mu_z-\frac{1}{2}\mu^2-s=0, \label{eq-theta}&~~~~\hbox{ co-touch condition}\\
 \<\eta,\eta\>&=\frac{\bar\mu^2}{4}\<\kappa, \kappa\>
+\bar\mu \<\kappa,D_{\bar{z}}\kappa\>
 +\<D_{\bar{z}}\kappa,D_{\bar{z}}\kappa\>=0. \label{eq-theta-2}&~~~~\hbox{ conformal condition}
\end{align}
\end{definition}

A basic fact about the adjoint transform is the following theorem.

\begin{theorem}
\label{thm-adjoint-dual}~\cite{ma2}~~The adjoint transform $[\hat{Y}]:M\to S^n$ is a Willmore surface when it is immersed. Moreover, the original Willmore surface $[Y]$ is an adjoint transform of $[\hat{Y}]$.
\end{theorem}
From this we have the following result.
\begin{proposition}
\label{prop-full}~ Let $[Y]$ be a Willmore surface full in $S^n$, that is, it is not contained in any $S^{n-1}\subset S^n$. Let $[\hat Y]$ be an adjoint transform of $[Y]$. Then $[\hat Y]$ is either a constant map or full in $S^n$ on an open dense subset.
\end{proposition}
We also have the following well-known characterization of minimal surfaces in $\R^n$ (see e.g. \cite{Bryant1984,Ejiri,Helein,ma0}).
\begin{proposition}
\label{prop-minimal}~ Let $Y$ be a Willmore surface in $S^n$ with some adjoint transform $[\hat Y]$ being a constant map. Then $[Y]$ is M\"obius congruent to a minimal surface in $\R^n$ and hence S-Wilmore.
\end{proposition}

Seeking solutions to the quadratic equation about $\bar\mu$ \eqref{eq-theta-2} leads us to its discriminant, which defines a $6$-form below \cite{ma0}:
\begin{equation}\label{eq-Theta2}
\Theta_0\dd z^6=\left(\<D_{\zb}\kappa,\kappa\>^2-\<D_{\zb}\kappa,
D_{\zb}\kappa\>\<\kappa,\kappa\>\right)\dd z^6.
\end{equation}
It is straightforward to verify that $\Theta_0\dd z^6$ is a global holomorphic differential. In fact, it also follows for the fact that $\chi_{0}\dd z^3=D_{\zb}\kappa\wedge\kappa \dd z^3$ is global and holomorphic (recall \eqref{eq-chi}), with
\[\Theta_0\dd z^6=\<\chi_0,\chi_0\>\dd z^6.\]
In terms of $\langle\kappa,\kappa\rangle$ and $\Theta_0$, the existence and uniqueness results of adjoint transforms can be stated as follows.
\begin{theorem}\label{thm-adjoint-no}~\cite{ma2} Let $y$ be a Willmore surface.
\begin{enumerate}
\item
  If $\langle\kappa,\kappa\rangle\equiv0$, then any solution to \eqref{eq-theta} defines an adjoint surface of $y$ via \eqref{yhat2}.
\item
 If $\langle\kappa,\kappa\rangle\not\equiv0$ and $\Theta_0\not\equiv0$, then there exists exactly two solutions to \eqref{eq-theta-2}. They provide two adjoint surfaces of $y$ via \eqref{yhat2}. Moreover, $y$ is not S-Willmore.
\item
 If $\langle\kappa,\kappa\rangle\not\equiv0$ and $\Theta_0\equiv0$, then there exists exactly one solution to \eqref{eq-theta-2}. There is a unique adjoint surface of $y$ via \eqref{yhat2}. Especially, if $y$ is S-Willmore, then the unique adjoint surface is its dual surface.
\end{enumerate}
\end{theorem}

Note that adjoint transforms of a Willmore surface might have singularities or branched points, which makes it more difficult to discuss global properties of adjoint transforms. We refer to \cite{Helein-2004,Ku-Sch,Mi-Ri,Ri} for the analysis on singularities of Willmore surfaces.

\subsection{Isotropic order of surfaces in $S^n$}\label{sec-iso}
The isotropic order of a surface can be viewed as the higher order conformality of surfaces, which is M\"obius invariant and plays an important role in the study of surfaces in spheres.

\begin{definition}\label{def-iso} Let $y:M\rightarrow S^n$ be a conformal immersion with a local complex coordinate $z$.
\begin{enumerate}
\item $y$ is called $m$-isotropic if \[\<y_{z},y_{z}\>=\cdots=\left\<y^{(m+1)}_{z},y^{(m+1)}_{z}\right\>\equiv0\] for $m\geq0$; Moreover, $y$ is called strictly $m$-isotropic if it is $m$-isotropic and not $(m+1)-$isotropic.
\item $y$ is called totally isotropic if for all $j\in \mathbb Z^+$, $y^{(j)}_{z}$ is isotropic, i.e., $\left\<y^{(j)}_{z},y^{(j)}_{z}\right\>\equiv0$.

\end{enumerate}
\end{definition}

\begin{remark}\label{rm-iso}
\
\begin{enumerate}

\item Totally isotropic surfaces play an important role in Calabi's famous classification of minimal $2$-spheres in $S^n$ \cite{Calabi}, where the twistor theory was first introduced to describe surfaces. To be concrete, Calabi introduced the twistor bundle of $S^{2m}$ and showed that all (branched) minimal $2$-spheres in $S^{2m}$ correspond exactly all horizontal twistor curves in the twistor bundle of $S^{2m}$.   The important idea of describing minimal surfaces in terms of auxiliary complex manifolds, as well as the introduction of twistor theory in geometry, has a deep influence in many branches of geometry \cite{BR,ES}, which also has a deep connection with the famous Grothendieck theorem \cite{Gr} for holomorphic bunldes over $2$-spheres \cite{BR}. It is not surprising that the twistor theory also plays an important role in the classification of Willmore $2$-spheres.

\item A Willmore $2$-sphere may not be totally isotropic. We need to consider the structure of the normal bundle of a Willmore $2$-sphere in more details, which leads to the notion of $m$-isotropic surfaces (see \cite{Ejiri-i,ma0}). For an $m$-isotropic surface, it is direct to see
\begin{equation}\label{eq-isotropic}
\left\<y^{(j)}_{z},y^{(l)}_{z}\right\>\equiv0\ \ \hbox{ for all $j+l\leq 2m+3, j,l\in \mathbb Z^+\cup\{0\}$}.
\end{equation}
In term of $\kappa$, we have
\begin{equation}\label{eq-isotropic2}
\left\<\dz^{(j)}\kappa,\dz^{(l)}\kappa\right\>\equiv0\ \ \hbox{ for all $j+l\leq 2m-1, j,l\in \mathbb Z^+$}.
\end{equation}
 Note that the notion of $m$-isotropic surfaces is M\"obius invariant, that is, if $y$ is $m$-isotropic and $T:S^n\rightarrow S^n$ is a M\"obius transformation, then $T\circ y$ is also $m$-isotropic.
\end{enumerate}
\end{remark}

\section{Harmonic sequence associated to a harmonic $2$-sphere}\label{sec-harm}


In this section, we will revisit the harmonic sequence introduced by Chern-Wolfson \cite{Chern-W} and then introduce the $\DD$-tranforms for special harmonic maps with Frenet-bundle structures. Since the target manifold of the conformal Gauss map $f_0$ is $Gr_{3,1}\R^{n+2}_1$, we will consider $Gr_{m,1}\R^{n+2}_1$, $Gr_{m}\R^{n+2}_1$, and $Gr_{m}(\R^{n+2}_1\otimes\C)$ in this section. Note that the proofs also hold for the cases $Gr_{m}\R^{n}$ and $Gr_{m}\C^{n}$ as well. We refer to \cite{BW,Chern-W,DPW,Uh} for more details on harmonic maps.

\subsection{Harmonic sequence in view of Chern-Wolfson}
Let $f:M\rightarrow Gr_{m,1}\R^{n+2}_1$ be a harmonic map, with $z$ a local complex coordinate on the Riemann surface $M$. For each point $p\in M$,  $f(p)$ is the $(m+1)$-subspace $f(p)\subset\R^{n+2}_1$. In this sense, we also identify $f$ with the tautological vector bundle over $M$ defined by $f(p)$.
\begin{definition}\label{def-cw}
 The map $\pf:M\rightarrow Gr_{m,1}\R^{n+2}_1$ is defined to be
\begin{equation}
    \label{eq-CW-1}
    \partial f|_p:=Span_{\C}\left.\left\{\psi_{jz}-\sum_{l=1}^{m}\< \psi_{jz}, \psi_{l}\> \psi_l+\< \psi_{jz}, \psi_{0}\> \psi_0,~ 0\leq j\leq m\right\}\right|_p
\end{equation}
on an open dense subset $M_0$ where  $\dim\partial f|_p=r$ on $M_0$ for some $r\leq m+1$. Here $\{\psi_{0},\cdots,\psi_{m}\}\subset\Gamma(f(U_p))$ is an orthonormal frame on an open neighborhood $U_p$ of $p$, with $\<\psi_0,\psi_0\>=-1$. Similarly we define  $\bar\partial f:M\rightarrow Gr_{\tilde r}\C^{n+2}$:
\begin{equation}
    \label{eq-CW-2}
    \bar\partial f|_p:=Span_{\C}\left.\left\{\psi_{j\bar{z}}-\sum_{l=1}^{m}\< \psi_{j\bar{z}},  {\psi}_{l}\>\psi_l+\< \psi_{j\bar{z}},  {\psi}_{0}\>\psi_0,~ 0\leq j\leq m\right\}\right|_p.
\end{equation}
 \end{definition}

\begin{remark}
  Both of $\pf$ and $\bar\partial{f}$ can be extended smoothly to $M$ by Lemma \ref{lemma-holo} as below. Note that when the target manifold is $Gr_{m+1}\C^{n+2}$, it is proven by \cite{Chern-W} and \cite{BW} independently.
\end{remark}

\begin{lemma}\label{lemma-holo} (Lemma 23 of \cite{BFLPP}, see also page 307 of \cite{Chern-W})
Let $M$ be a Riemann surface and let $\mathbb{V}\rightarrow M$, $\mathbb{W}\rightarrow M$ be two holomorphic bundles over $M$. {Assume $\mathbb W$ is an Hermitian bundle.} Let $T:\mathbb{V}\rightarrow \mathbb{W}$ be a holomorphic bundle map. Then there exists a unique holomorphic subbundle $\mathbb{W}_0$ of $\mathbb{W}$ on $M$such that
\[ Im T(\VV)= \mathbb{W}_0 \hbox{ on } M_0, ~ Im T(\VV)\varsubsetneqq \mathbb{W}_0\hbox{ on $M\setminus M_0$},\]
where  $M_0$ is  an open dense subset of $M$.
\end{lemma}

The proof of Lemma \ref{lemma-holo} is the same as the proof of \cite[Lemma 23]{BFLPP}.



\begin{theorem}\label{thm-CW1}\cite{Chern-W,BW}
Let $f:M\rightarrow Gr_{m,1}\R^{n+2}_1$ be a harmonic map. Then each of the maps $\partial f$ and $\bar\partial f$ is a harmonic map globally defined on $M$ and each of the maps $\partial f$ and $\bar\partial f$ is a harmonic map.
\end{theorem}
\begin{proof}
   Since $f^{\perp}\otimes\C$ has a Hermitian metric for the map $f:M\rightarrow Gr_{m,1}\R^{n+2}_1$,  the first statement is confirmed by Lemma \ref{lemma-holo}. The harmonicity of $\partial f$ and $\bar\partial f$ follows in the same way as \cite{Chern-W,BW}.
\end{proof}

\subsection{Real harmonic sequence for  harmonic maps with a Frenet-bundle structure}\label{subsect-harmonic}
\begin{definition}\label{def-Frenet}
 Let $f:M\rightarrow Gr_m\R^n$ or $f:M\rightarrow Gr_{m,1}\R^{n+2}_{1}$ be a harmonic map.
 \begin{enumerate}
\item We call $f$ has a Frenet-bundle structure $\{f,\ZZ,h; \tr\}$ if there exists an orthogonal decomposition of the trivial bundle
\[M\times\R^{n+2}_1=f\oplus Re(\ZZ)\oplus h\]
such that $\ZZ$ is an isotropic holomorphic subbundle and
 \begin{enumerate}
 \item $\partial f\subset \ZZ\oplus (h\otimes \C)$;
 \item For any $\psi\in\Gamma(\ZZ)$, $\psi_z\in\Gamma(\ZZ\oplus(h\otimes\C))$;
 \item For any $\psi\in\Gamma(\ZZ)$, $\psi_{\zb}\in\Gamma(\ZZ\oplus(f\otimes\C))$;
 \item $Rank_{\C}(\partial h)=\tilde{r}$.
 \end{enumerate}
 When $\tr=1$, we denote the  Frenet-bundle structure simply by $\{f,\ZZ,h\}$.
\item
Moreover, we define the map $\DD_{\ZZ} f:M\rightarrow Gr_{m+2q}\R^n$ or $\DD_{\ZZ} f:M\rightarrow Gr_{m+2q,1}\R^{n+2}_{1}$, with $q=Rank_{\C}(\ZZ)$, as follows.
\begin{equation}\label{eq-real-harmonic-def}
\DD_{\ZZ} f(z):=f\oplus Re(\ZZ).
\end{equation}
 \end{enumerate}
 \end{definition}

 \begin{theorem}\label{thm-harmonicity}  Let $f:M\rightarrow Gr_m\R^n$ or $f:M\rightarrow Gr_{m,1}\R^{n+2}_{1}$ be a  map, with a Frenet-bundle structure $\{f,\ZZ,h;\tr\}$. Then the map $ \DD_{\ZZ} f=f\oplus Re(\ZZ)$ is harmonic if and only if $f$ is harmonic. And it is conformal if and only if $f$ is conformal.
 \end{theorem}
\begin{proof} We need only to prove  Theorem \ref{thm-harmonicity} for the case of $Gr_{m,1}\R^{n+2}_{1}$ since $Gr_m\R^n$ can be viewed as a totally geodesic submanifold of $Gr_{m,1}\R^{n+2}_{1}$. Let $\{\xi_j, 0\leq j\leq m\}$ be an orthonormal basis of $f$ with $\<\xi_j,\xi_j\>=\epsilon_j$ for $0\leq j\leq m$, where $\epsilon_j=-1, \hbox{ for $j=0$, and }  \epsilon_j=1 \hbox{ for $1\leq j\leq m$}.$ Let $\{E_{\alpha}=\frac{1}{\sqrt{2}}(e_{\alpha}+ie_{\alpha+q}), m+1\leq \alpha\leq m+q\}$ be a unitary basis of $\ZZ$. Let $\{\psi_{A}, m+2q+1\leq A\leq n+2\}$ be an orthonormal basis of $(\DD_{\ZZ} f)^{\perp}$.  Then by (1a), (1b) and (1c) of Definition \ref{def-Frenet}, we have
\begin{equation}
    \label{eq-Gij}
    \left\{
    \begin{split}
      \xi_{jz}&=\sum_{l} \OO_{jl}\xi_l+\sum_{  \alpha } {\OO}_{j\alpha}E_{\alpha}+\sum_{A} \OO_{jA}\psi_{A}, ~0\leq j\leq m,\\
      E_{\alpha z}&=\sum_{ \beta } {\OO}_{\alpha\beta}E_{\beta}+\sum_{A} {\OO}_{\alpha A}\psi_{A}, ~m+1\leq \alpha\leq m+q,\\
      \bar{E}_{\alpha z}&=-\sum_j\epsilon_j{\OO}_{j\alpha}\xi_j-\sum_{ \beta }  {\OO}_{\beta\alpha}\bar{E}_{\beta}, ~m+1\leq \alpha\leq m+q,\\
      \psi_{Az}&=-\sum_j\epsilon_j\OO_{j A}\xi_j-\sum_{ \alpha } {\OO}_{\alpha A}\bar{E}_{\alpha}+\sum_{B} \OO_{AB}\psi_{B}, ~m+2q+1\leq A\leq n+2.\\
    \end{split}\right.
\end{equation}
Since both $\xi_{jz\zb}$ and  $\psi_{A z\bar{z}}$ are real, we obtain
\begin{equation}\label{eq-reality}
\left\{\begin{split}
&  \OO_{j\alpha\zb}-\sum_{\beta}\OO_{j\beta}\bar{\OO}_{\alpha\beta}-\sum_{A}\OO_{jA}\bar{\OO}_{\alpha A}-\sum_{l}\OO_{l\alpha}\bar{\omega}_{jl}=0,\\
&
\OO_{\alpha A\bar{z}}+\sum_j\epsilon_j\OO_{j A}\bar{\OO}_{j\alpha}+\sum_{\beta}\OO_{\beta A}\bar{\OO}_{\beta\alpha}+\sum_B\OO_{\alpha B}\bar{\OO}_{BA}=0,
\end{split}\right.
\end{equation}
for $0\leq j\leq m$, $m+1\leq \alpha\leq m+q$ and $m+2q+1\leq a\leq n+1$.\vspace{1mm}

For simplicity, we introduce the following notation
\begin{equation}
    \label{eq-grass}
    [\xi]:=\xi_0\wedge\cdots\wedge \xi_m,~ [\xi_{\widehat{i}}]:=\xi_0\wedge\cdots\wedge\widehat{ {\xi}_{i}}\wedge\cdots\wedge\xi_m,~ [\xi_{\widehat{i}, \widehat{j}}]:=\xi_0\wedge\cdots\wedge\widehat{ {\xi}_{i}}\wedge\cdots\wedge\widehat{ {\xi}_{j}}\wedge\cdots\wedge \xi_m.
\end{equation}
We have
\[[\xi]_z=\sum_j (-1)^{j+1}(\sum_{\alpha} \OO_{j\alpha}E_{\alpha}+\sum_A\OO_{jA}\psi_{A})\wedge[\xi_{\widehat{j}}].\]
So
\[\begin{split}[\xi]_{z\zb}=&\sum_j (-1)^{j+1}(\sum_{\alpha} \OO_{j\alpha\zb}E_{\alpha}+\OO_{j\alpha}E_{\alpha\zb}+\sum_A\OO_{jA\zb}\psi_{A}+\OO_{jA}\psi_{A\zb})\wedge[\xi_{\widehat{j}}]\\
&+   \sum_j (-1)^{j+1}(\sum_{\alpha} \OO_{j\alpha}E_{\alpha}+\sum_A\OO_{jA}\psi_{A})\wedge[\xi_{\widehat{j}}]_{\zb}.\\
\end{split}\]
We also have
\[\begin{split}
    [\xi_{\widehat{j}}]_{\zb}=&\sum_{l}(-1)^{j-l-1}\bar{\omega}_{lj}[\xi_{\widehat{l}}]+\sum_{l<j,\beta}(-1)^{l+1}\bar{\omega}_{l\beta}\bar{E}_{\beta}\wedge[\xi_{\widehat{l},\widetilde{j}}]+\sum_{l>j,\beta}(-1)^{l}\bar{\omega}_{l\beta}\bar{E}_{\beta}\wedge[\xi_{\widehat{j},\widetilde{l}}]\\
    &+\sum_{l<j}(-1)^{l+1}\bar{\omega}_{lA}\psi_A\wedge[\xi_{\widehat{l},\widehat{j}}]+\sum_{l>j}(-1)^{l}\bar{\omega}_{lA}\psi_A\wedge[\xi_{\widehat{j},\widehat{l}}].\\
\end{split}\]
Altogether, we obtain
\[\begin{split}[\xi]_{z\zb}=&\sum_j (-1)^{j+1}(\sum_{\alpha} \OO_{j\alpha\zb}E_{\alpha}+\OO_{j\alpha}E_{\alpha\zb}+\sum_A\OO_{jA\zb}\psi_{A}+\OO_{jA}\psi_{A\zb})\wedge[\xi_{\widehat{j}}]\\
&+   \sum_j (-1)^{j+1}(\sum_{\alpha} \OO_{j\alpha}E_{\alpha}+\sum_A\OO_{jA}\psi_{A})\wedge[\xi_{\widehat{j}}]_{\zb}.\\
=&\sum_j (-1)^{j+1}(\sum_{\alpha} \OO_{j\alpha\zb}-\sum_{\beta}\OO_{j\beta}\bar{\OO}_{\alpha\beta}-\sum_{A}\OO_{jA}\bar{\OO}_{\alpha A}-\sum_{l}\OO_{l\alpha}\bar{\omega}_{jl})E_{\alpha}\wedge[\xi_{\widehat{j}}]\\
&+\sum_j (-1)^{j+1}(\sum_A\OO_{jA\zb}+\sum_{B}\OO_{jB}\bar{\OO}_{BA}-\sum_{l}\OO_{lA}\bar{\OO}_{jl})\psi_{A}\wedge[\xi_{\widehat{j}}]\\
&-\sum_j \epsilon_j(\sum_{\alpha}|\OO_{j\alpha}|^2+\sum_A|\OO_{jA}|^2)[\xi]\\
&-\sum_j (-1)^{j}(\sum_{\alpha} \OO_{j\alpha}E_{\alpha}+\sum_A\OO_{jA}\psi_{A})\wedge\sum_{l<j}(-1)^{l+1}(\sum_{\beta}\bar{\omega}_{l\beta}\bar{E}_{\beta}+\sum_B\bar{\omega}_{lB}\psi_B)\wedge[\xi_{\widehat{l},\widetilde{j}}]\\
&-\sum_j (-1)^{j}(\sum_{\alpha} \OO_{j\alpha}E_{\alpha}+\sum_A\OO_{jA}\psi_{A})\wedge\sum_{l>j}(-1)^{l}((\sum_{\beta}\bar{\omega}_{l\beta}\bar{E}_{\beta}+\sum_B\bar{\omega}_{lB}\psi_B))\wedge[\xi_{\widehat{j},\widetilde{l}}].\\
\end{split}\]
So $[\xi]$ is harmonic if and only if
\begin{equation}\label{eq-xi-H}
    \left\{
    \begin{split}
        &
        \OO_{j\alpha\zb}-\sum_{\beta}\OO_{j\beta}\bar{\OO}_{\alpha\beta}-\sum_{A}\OO_{jA}\bar{\OO}_{\alpha A}-\sum_{l}\OO_{l\alpha}\bar{\omega}_{jl}=0,\\
        &
       \OO_{jA\zb}+\sum_{B}\OO_{jB}\bar{\OO}_{BA}-\sum_{l}\OO_{lA}\bar{\OO}_{jl}=0,\\
    \end{split}\right.
\end{equation}
for $0\leq j\leq m$, $m+1\leq \alpha\leq m+q$ and $m+2q+1\leq A\leq n+2$.

Similarly, we have $[\psi]=(\DD_{\ZZ}f)^{\perp}$ is harmonic if and only if
\begin{equation}\label{eq-psi-H}
    \left\{
    \begin{split}
        &       \OO_{jA\zb}+\sum_{B}\OO_{jB}\bar{\OO}_{BA}-\sum_{l}\OO_{lA}\bar{\OO}_{jl}=0,\\
& \OO_{\alpha A\bar{z}}+\sum_j\epsilon_j\OO_{j A}\bar{\OO}_{j\alpha}+\sum_{\beta}\OO_{\beta A}\bar{\OO}_{\beta\alpha}+\sum_B\OO_{\alpha B}\bar{\OO}_{BA}=0,
    \end{split}\right.
\end{equation}
for $0\leq j\leq m$, $m+1\leq \alpha\leq m+q$ and $m+2q+1\leq A\leq n+2$.

Since $[\psi]=(\DD_{\ZZ}f)^{\perp}$ is harmonic if and only if $\DD_{\ZZ}f$ is harmonic,
by \eqref{eq-reality}, \eqref{eq-xi-H} and \eqref{eq-psi-H}, $ \DD_{\ZZ} f=f\oplus Re(\ZZ)$ is harmonic if and only if $f$ is harmonic.

Finally, by \eqref{eq-Gij},  $ \DD_{\ZZ} f$ is conformal if and only if $\sum_{j,A}(\OO_{j A})^2=0$, if and only if $f$ is conformal.

\end{proof}

The quadratic $\mathcal Q^{n+1}_1\subset \C^{n+2}_1=\R^{n+2}_1\otimes\C$ is defined by
\begin{equation}
\mathcal Q^{n+1}_1:=\{Z\in\C^{n+2}_1|\<Z,Z\>=0\}.
\end{equation}
\begin{definition} A map $f:M\rightarrow Gr_{m,1}\R^{n+2}_1$ is called $\partial-$tangent to $\Q$ if it satisfies:
\begin{enumerate}
    \item
For every $p\in M$, $(\partial f)|_{p}\cap \Q$ is a linear subspace;
    \item There exists an isotropic holomorphic subbundle $\ZZ_Q$ of $\partial f$ such that $\ZZ_Q=\partial f\cap \Q$ on an open dense subset of $M$.
\end{enumerate}

 Moreover, for a $\partial-$tangent  map $f$, we define $\D f$ as follows:
 \begin{equation}\label{def-D}
 \mathcal{D}f:=\mathcal{D}_{\ZZ_Q}(f)=f\oplus Re(\ZZ_Q).
 \end{equation}
\end{definition}
\begin{theorem}
    \label{thm-db-1} Let $f:M\rightarrow Gr_{m,1}\R^{n+2}_1$ be a $\partial-$tangent harmonic map with $\ZZ_Q$ satisfying $\bar\partial \ZZ_Q\subset f\otimes\C$, then $\D f$ is a harmonic map.
\end{theorem}
\begin{proof}
Set $h=(f\oplus Re(\ZZ_Q))^{\perp}$. Then by definition and assumptions, $\{f,\ZZ_Q,h;\tilde{r}\}$ forms a Frenet-bundle structure and hence $\DD f$ is a harmonic map by Theorem \ref{thm-harmonicity}.
\end{proof}
For  $\partial-$tangent harmonic maps, we have the following useful lemma.
\begin{lemma}
If $f$ is a $\partial-$tangent harmonic map, then the complex linear inner product restricting on $\partial f$ has rank $\leq 1$ on every point of $M$.
\end{lemma}
\begin{proof} Assume by contradiction that for some $p\in M$, there exist two linearly independent vectors $\{X_1, X_2\}\subset\partial f|_p$ with $\det(\<X_j, X_l\>)\neq0$. If
$\<X_1, X_1\>=\<X_2, X_2\>=0$, then $\<X_1, X_2\>\neq0$ and $X_1+X_2\not\in \Q$,  contradicting to $\hbox{Span}_{\C}\{X_1, X_2\}\subset\Q$. If
$\<X_j, X_j\>\neq0$ for some $j=1,2$, there are exactly two linear independent isotropic vectors such that $Y_j=X_1+c_j X_2$, $j=1,2$, which means $\hbox{Span}_{\C}\{X_1, X_2\}=\hbox{Span}_{\C}\{Y_1,Y_2\}\subset\Q$, a contradiction since $X_j\not\in \Q$.
\end{proof}
\vspace{2mm}
  From the above proofs, we can see that all of the above definitions and results can be generalized straightforwardly to the Lorentzian case or other pseudo-Riemannian cases \emph{locally} when $\partial f$ or $\ZZ$ is an hermitian subbundle on some open subset of $M$. Here we just state one of them which will be used in later sections.
\begin{theorem}
    \label{thm-db-2}
  Assume $f:M\rightarrow Gr_{r}\R^{n+2}_1$ is a $\partial-$tangent harmonic map on an open dense subset $M_0\subset M$, with $\ZZ_Q$ being an hermitian bundle over $M_0$. Define \[\D f:=\D_{\ZZ_Q}f=f\oplus Re(\ZZ_Q)\] on $M_0$ in the same way as \eqref{def-D}. If $\bar\partial \ZZ_Q\subset f\otimes\C$, then $\D f$ is a harmonic map on $M_0$.
\end{theorem}
\begin{remark} If $\mathcal \ZZ$ ($\ZZ_Q$) is not an hermitian subbundle, the statement of Theorem \ref{thm-harmonicity} (Theorem \ref{thm-db-2}) will be more technical, which will be discussed separately.\end{remark}

\vspace{2mm}

\section{Geometric structure of $k$-isotropic Willmore $2$-spheres}

In this section, we show that a $k$-isotropic Willmore $2$-sphere has a natural sequence of isotropic holomorphic subbundles, which are useful for characterizations of Willmore $2$-spheres.

\begin{definition} Let $y:S^2\rightarrow S^n$ be a $k$-isotropic Willmore $2$-sphere  with $k\geq1$. We define a sequence of global subbundles of its conformal normal bundle as follows:
\begin{equation}\label{eq-pi-j}
  \Pi_j:=\hbox{Span}_{\C}\{\kappa,\db \kappa,\cdots,\dz^{(j)}\kappa,\db\dz^{(j)}\kappa\},  0\leq j\leq k-1.
\end{equation}
\end{definition}
One can check easily that $\Pi_j$ is globally defined by the Willmore equation \eqref{willmore}, Ricci equation \eqref{ricci} and Chern's Lemma \cite{Chern}.  Moreover, we have the following theorem.
\begin{theorem} \label{thm-m-iso}
For each $j$, $0\leq j\leq k-1$, $\Pi_j$ is an isotropic, holomorphic bundle.
\end{theorem}
To prove Theorem \ref{thm-m-iso}, we need the following useful technical lemma.
\begin{lemma}\label{lemma-ind2}
Assume that the bundle $\Pi_l$ defined in \eqref{eq-pi-j} is isotropic for some $0\leq l\leq k-1$. Then the following equations hold (we set $\Pi_{-1}:=\{0\}$ for simplicity):
\begin{equation}\label{eq-dzdb}
\db\dz\dz^{(j)}\kappa-\dz\db\dz^{(j)}\kappa=0\mod\{\kappa\},\ \hbox{ for all } 0\leq j\leq 2k-1;
\end{equation}
\begin{equation}\label{eq-dzpi}
 \dz \phi\in\Gamma(\Pi_{l}), \ \hbox{ for all\ }\ \phi\in\Gamma(\Pi_{l-1});
\end{equation}
\begin{equation}\label{eq-dbdb}
\db\db\dz^{(j)}\kappa=-\frac{\bar s}{2}\dz^{(j)}\kappa\mod \{\Pi_{j-1} \}\hbox{ for all } 0\leq j\leq l+1;
\end{equation}
\begin{equation}\label{eq-dbpi}
 \db \phi\in\Gamma(\Pi_{l}), \ \hbox{ for all\ }\ \phi\in\Gamma(\Pi_l).
\end{equation}

\end{lemma}
\begin{proof}
By (3) of Remark \ref{rm-iso},  for a $k$-isotropic surface,  we have \eqref{eq-isotropic2}
\begin{equation}
    \label{eq-isotropic-k}
\la\dz^{(j)}\kappa, \kappa\ra=0, \hbox{ for all $0\leq j\leq 2k-1$}. \end{equation}
Then by the Ricci equation \eqref{ricci}, we have inductively that,  for all $0\leq j\leq 2k-1$,
\[\db\dz\dz^{(j)}\kappa-\dz\db\dz^{(j)}\kappa=2\la\dz^{(j)}\kappa, \bar\kappa\ra\kappa-2\la\dz^{(j)}\kappa, \kappa\ra\bar\kappa=0\mod\{\kappa\}.\]
This proves \eqref{eq-dzdb}.

Equation \eqref{eq-dzpi} follows directly from the definition of $\Pi_l$ and \eqref{eq-dzdb}.

We prove \eqref{eq-dbdb} by induction. When $j=0$, \eqref{eq-dbdb} is the Willmore equation \eqref{willmore}. Assume now  \eqref{eq-dbdb}  holds for $0\leq j\leq l$, that is,
$\db\db\dz^{(j)}\kappa=-\frac{\bar s}{2}\dz^{(j)}\kappa\mod \{\Pi_{j-1}\}$. Then by \eqref{eq-isotropic-k} and the Ricci equation \eqref{ricci}, we have
\[\begin{split}
\db\db\dz^{(l+1)}\kappa&=\db\left(\dz\db\dz^{(l)}\kappa+(\cdots)\kappa\right) \\
&=\dz\left(\db\db\dz^{(l)}\kappa\right)-2\<\db\dz^{(l)}\kappa,\kappa\>\bar\kappa+(\cdots)\kappa+(\cdots)\db\kappa.
\end{split}\]
By  assumption of $\Pi_l$ being isotropic, we have $\<\db\dz^{(l)}\kappa,\kappa\>=0$.
So
\[\begin{split}
\db\db\dz^{(l+1)}\kappa&=\dz\left(\db\db\dz^{(l)}\kappa\right) +(\cdots)\kappa+(\cdots)\db\kappa 
=-\frac{\bar s}{2}\dz^{(l+1)}\kappa\mod \{\Pi_{l}\} \\
\end{split}\]
by \eqref{eq-dzpi} and \eqref{eq-dbdb} for $j=l$. This finishes the proof of \eqref{eq-dbdb} for $j=l+1$.

Finally,  by \eqref{eq-dzdb} and \eqref{eq-dbdb} we obtain $\db\db\dz^{(j)}\kappa=(\cdots)\dz^{(j)}\kappa\mod \{\Pi_{j-1}\}$  for all $j\leq l$. Together with $\db\dz^{(j)}\kappa\in\Gamma(\Pi_{j})$, we obtain \eqref{eq-dbpi}.
\end{proof}

\begin{proof}[Proof of Theorem \ref{thm-m-iso}] We prove it by induction.

Since $\<\kappa,\kappa\>\equiv0$, we have  $\<\db\kappa,\kappa\>\equiv0$ and $\<\db\kappa,\db\kappa\>=-\<\kappa,\db\db\kappa\>\equiv0$ by the Willmore equation. This shows that $\Pi_{0}$ is an isotropic holomorphic bundle.

Now assume Theorem \ref{thm-m-iso} holds for $\Pi_l$, $0\leq l\leq \max\{0,k-2\}$. In particular, $\Pi_{l}$ is an isotropic holomorphic bundle. So we need only to consider $\dz^{(l+1)}\kappa$ and $\db\dz^{(l+1)}\kappa$.
The isotropicity of $y$ \eqref{eq-isotropic2} indicates $\<\dz^{(l+1)}\kappa,\dz^{(j)}\kappa\>=0$ for all $0\leq j\leq l+1$.
For $j\leq l-1$, we get
\[\begin{split}\<\dz^{(l+1)}\kappa,\db\dz^{(j)}\kappa\>&=\<\dz^{(l+1)}\kappa,\dz^{(j)}\kappa\>_{\bar z}-\<\db\dz^{(l+1)}\kappa,\dz^{(j)}\kappa\>
\\
&=-\<\dz\db\dz^{(l)}\kappa,\dz^{(j)}\kappa\>\\
&=\<\db\dz^{(l)}\kappa,\dz^{(j)}\kappa\>_z+\<\db\dz^{(l)}\kappa,\dz^{(j+1)}\kappa\>\\
&=0.
\end{split}
\]
So we have $\dz^{(l+1)}\kappa\perp \Pi_{l-1}$ and $\dz^{(l+1)}\kappa\perp \dz^{(l)}\kappa$. To show $\<\dz^{(l+1)}\kappa,\db\dz^{(l)}\kappa\>=0$, consider
\[\Omega_l\dd z^{2l+4}:=\<\dz^{(l+1)}\kappa,\db\dz^{(l)}\kappa\>\dd z^{2l+4}\]
Since $\Pi_l$ is totally isotropic, it is easy to check that $\Omega_l\dd z^{2l+4}$ is globally defined on $S^2$. We have
\[\Omega_{l\bar z}=\<\dz^{(l+1)}\kappa,\db\db\dz^{(l)}\kappa\>+\<\db\dz^{(l+1)}\kappa,\db\dz^{(l)}\kappa\>\]
 Since $\dz^{(l+1)}\kappa\perp \Pi_{l-1}$ and $\dz^{(l+1)}\kappa\perp \dz^{(l)}\kappa$, together with \eqref{eq-dbdb}, we have
\[\<\dz^{(l+1)}\kappa,\db\db\dz^{(l)}\kappa\>=\<\dz^{(l+1)}\kappa,(\cdots)\dz^{(l)}\kappa\>=0.\] Since $\Pi_l$ is isotropic, we have $\<\db\dz^{(l)}\kappa,\db\dz^{(l)}\kappa\>=0$. So
\[\<\db\dz^{(l+1)}\kappa,\db\dz^{(l)}\kappa\>=\<\dz\db\dz^{(l)}\kappa,\db\dz^{(l)}\kappa\>=\frac{1}{2}\<\db\dz^{(l)}\kappa,\db\dz^{(l)}\kappa\>_{z}=0.\]
Hence $\Omega_{l\bar z}=0$, which indicates that $\Omega_l\dd z^{2l+4}$ vanishes on $S^2$ and $\<\dz^{(l+1)}\kappa,\dz^{(l+1)}\kappa\>\equiv0$ holds. So we have $\<\dz^{(l+1)}\kappa,\db\dz^{(l+1)}\kappa\>=0$.
So $\dz^{(l+1)}\kappa \perp \Pi_{l+1}$.

Finally, we have $\<\db\dz^{(l+1)}\kappa,\dz^{(j)}\kappa\>=-\<\dz^{(l+1)}\kappa,\db\dz^{(j)}\kappa\>=0$ for all $0\leq j\leq l+1$. By \eqref{eq-dbdb}, we have for $0\leq j\leq l+1$,
\[\begin{split}\<\db\dz^{(l+1)}\kappa,\db\dz^{(j)}\kappa\>&=\<\dz^{(l+1)}\kappa,\db\dz^{(j)}\kappa\>_{\bar{z}}-\<\dz^{(l+1)}\kappa,\db\db\dz^{(j)}\kappa\>\\
&=-\<\dz^{(l+1)}\kappa,\db\db\dz^{(j)}\kappa\>\\
&=0,\end{split}
\]
since $\dz^{(l+1)}\kappa\perp   \Pi_{l+1}$. This indicates that $\Pi_{l+1}$ is  isotropic. The holomorphicity of $\Pi_{l+1}$ comes from \eqref{eq-dbdb}.  This finishes the proof.
\end{proof}

\vspace{2mm}

\section{Totally isotropic Willmore $2$-spheres \& $\FH$ twistor curves}\label{sec-twistor}

In this section, we  first recall the twistor bundle $\TS$ of $S^{2m}$ briefly and then give a characterization of totally isotropic Willmore $2$-spheres in terms of $\FH$ twistor curves. We refer to \cite{Bar, BR, Calabi, Helein-B, PT} for more details on the twistor theory.
\subsection{The twistor bundle  $\TS$ and twistor curves}

\subsubsection{The twistor bundle  $\TS$}
To begin with, we fix the natural orientation of $S^{2m}$. We define the twistor bundle $\TS$ as
\begin{equation}\label{eq-twistor}
\TS:=\coprod_{p\in S^{2m}}\{J_p\hbox{ is an oriented orthonormal basis of } T_pS^{2m}\},
\end{equation}
with the twistor projection
\[\pi_t:\TS\rightarrow S^{2m}, ~~\pi_t(J_p):=p.\]
To describe $\TS$, we recall the bundle $\pi_s:SO(2m+1)\rightarrow S^{2m}=SO(2m+1)/SO(2m)$.
Let $F\in\pi_s^{-1}(p)\subset SO(2m+1)$ with $p\in S^{2m}$. So $F$ is of the form
\begin{equation}\label{eq-F-lift}
    F=\left(
      \begin{array}{cccc}
        p & A & B\\
      \end{array}
    \right), \ \hbox{ with } A=\left(
      \begin{array}{cccc}
       e_1 & \cdots & e_{m} \\
      \end{array}
    \right),~ B=\left(
      \begin{array}{cccc}
       e_{m+1} & \cdots & e_{2m} \\
      \end{array}
    \right),
\end{equation}
with $\{e_1,\cdots,e_{2m}\}$ forming an oriented, orthonormal basis of $T_pS^{2m}$. So we can define an oriented orthonorml complex structure $J_p^F$ of $T_pS^{2m}$ as follows:
\[J_p^F(e_j)=e_{m+j},~ J_p^F(e_{m+j})=-e_{j}, \hbox{ for } 1\leq j\leq m.\]
In this way, we obtain the projection
 \[\pi_{st}:SO(2m+1)\rightarrow \TS,~ \pi_{st}(F):=J_p^F.\]
Moreover, set
 \begin{equation}\label{eq-J0}
   J_{2m}=\left(
        \begin{array}{cc}
          0 & -I_m \\
          I_m & 0 \\
        \end{array}
      \right).
 \end{equation}
We see that for $F_1,F_2\in \pi_s^{-1}(p)$, $J_p^{F_1}=J_p^{F_2}$ if and only if $F_1=F_2 K$ with
\[K\in U(m):=\left\{\left.\left(
           \begin{array}{ccc}
             1 & 0 \\
             0 & A \\
           \end{array}
         \right)\in SO(2m+1)\right|
AJ_{2m}=J_{2m}A
\right\}.\]
As a consequence we have
$\TS\cong SO(2m+1)/U(m)$ and the following diagram holds:
\[\xymatrix{
  SO(2m+1) \ar[dr]_{\pi_s} \ar[r]^{\pi_{st}}
                & \TS \ar[d]^{\pi_t}  \\
                & S^{2m} }
     \]
Moreover, the tangent bundle of $\TS$ can be written as follows:
\begin{equation}\label{eq-tts}
T_{J^F}\TS=\left\{(\pi_{st})_{* F}F\left(
                    \begin{array}{ccc}
                                                     0 & -\alpha^T   \\
                                                     \alpha  &X \\
                                                   \end{array}
                                                 \right)|X^T+X=0, J_{2m}X+XJ_{2m}=0
\right\}.\end{equation}
We define the almost complex structure  $\JJ$ of $\TS$ as follows:
\begin{equation}\label{eq-J1}
  \JJ\left((\pi_{st})_{* F}\left(F\left(
                                                   \begin{array}{ccc}
                                                     0 & -\alpha^T   \\
                                                     \alpha  &X \\
                                                   \end{array}
                                                 \right)\right)\right):= (\pi_{st})_{* F}\left(F\left(
                                                   \begin{array}{ccc}
                                                     0 & -\alpha^TJ_{2m}^T   \\
                                                     J_{2m}\alpha  &J_{2m}X \\
                                                   \end{array}
                                                 \right)\right),
\end{equation}
 It is well known that $\JJ$ is well defined and $(\TS,\JJ)$ is a complex manifold (see \cite{BR,Bar}).

Finally, we define the associated isotropic bundle $ \I$ as follows:
\begin{equation}\label{eq-zz}
   \I|_{p}:=\hbox{Span}_{\C}\left\{E_{j}:=e_j-ie_{m+j},1\leq j\leq m\right\}.
\end{equation}
Note that $\I|_p$ is the eigenspace of $J_p^F$ with respect to $\sqrt{-1}$, which uniquely determines $J_p^F$.

\subsubsection{Twistor curves}

Let $x:M\rightarrow S^{2m}$ be a conformal immersion. Let $U\subset M$ be an open subset with a complex coordinate $z$. Moreover, assume
\[x=\pi_t(\gamma), \hbox{ for some }\gamma: M\rightarrow \TS.\]
Let
$F=(x\ A \ B):U\rightarrow SO(2m+1)$
    be a local lift of both $x$ and $\gamma$, that is, $\pi_{st}(F)=\gamma$ and  $\pi_s(F)=x$. We decompose $\alpha$ and $X$ in \eqref{eq-tts} with respect to \eqref{eq-F-lift}. Then we have:
    \[
   \alpha= \left(
                                                   \begin{array}{cc}
                                                     \alpha_1  \\
                                                     \alpha_2  \\
                                                   \end{array}
                                                 \right), ~~ X= \left(
                                                   \begin{array}{cc}
                                                      a&b \\
                                                      b&-a \\
                                                   \end{array}
                                                 \right), \hbox{ with }a+a^T=b+b^T=0.  \]
                                                 Moreover, we have
    \begin{equation}
        \label{eq-Fz-1}
        \frac{\partial F}{\partial z}=F\left(
                                                   \begin{array}{ccc}
                                                     0 & -\alpha_1^T &-\alpha_2^T  \\
                                                     \alpha_1 & a+\hat a&b-\hat b \\
                                                     \alpha_2 & b+\hat b&-a+\hat a \\
                                                   \end{array}
                                                 \right),
    \end{equation}
    with
$a+a^T=b+b^T=\hat a+\hat a^T=\hat{b}-\hat{b}^T=0.$ As a consequence, we get
 \begin{equation}\label{eq-fz} (\pi_{st})_{* F}\left(\frac{\partial F}{\partial z}\right)=(\pi_{st})_{* F}F\left(
                                                   \begin{array}{ccc}
                                                     0 & -\alpha_1^T &-\alpha_2^T  \\
                                                     \alpha_1 & a &b \\
                                                     \alpha_2 & b&-a  \\
                                                   \end{array}
                                                 \right),\end{equation}
                                                 and
                                    \begin{equation}
                                    \label{eq-jfz} \JJ(\pi_{st})_{* F}\left(\frac{\partial F}{\partial z}\right)=(\pi_{st})_{* F}F\left(
                                                   \begin{array}{ccc}
                                                     0 & \alpha_2^T &-\alpha_1^T  \\
                                                     -\alpha_2 & -b & a \\
                                                    \alpha_1 & a&b  \\
                                                   \end{array}
                                                 \right).\end{equation}
\begin{proposition}\label{prop-J1}Let $\I$ be the isotropic bundle of $\gamma=\pi_{st}(F)$.
\begin{enumerate}
    \item
The following statements are equivalent:
\begin{enumerate}
\item $\pi_{st}(F)$ is $\JJ$ holomorphic;
\item $\alpha_2=-i\alpha_1$ and  $b=-ia$;
\item   $x_z\in \Gamma (\I)$, and $\phi_{z}\in\Gamma (\I)$ for any $\phi\in\Gamma(\I)$.
\end{enumerate}
\item The following statements are equivalent:
\begin{enumerate}
\item $\pi_{st}(F)$ is $\JJ$ anti-holomorphic;
\item $\alpha_2=i\alpha_1$ and  $b=ia$,
\item $x_{\zb}\in \Gamma (\I)$, and $\phi_{\zb}\in\Gamma (\I)$ for any $\phi\in\Gamma(\I)$.
\end{enumerate}
\end{enumerate}
\end{proposition}
\begin{proof}We just prove the holomorphic case; the anti-holomorphic case is the same.
By definition, $\gamma=\pi_{st}(F)$ is $\JJ$ holomorphic if and only if
\[\JJ(\pi_{st})_{* F}\left(\frac{\partial F}{\partial z}\right)=i\left(\frac{\partial F}{\partial z}\right).\]
By \eqref{eq-fz} and \eqref{eq-jfz},  (a) and (b) are equivalent.

By \eqref{eq-F-lift} and \eqref{eq-Fz-1},  we obtain
\[\begin{split}
    x_z&=A\alpha_1+B\alpha_2,\\
A_z&=-x\alpha_1^T+Aa+A\hat a+B\hat b+Bb,\\
B_z&=-x\alpha_2^T+Ab-A\hat b-Ba+B\hat a,
\end{split}\]
frow which we have
\[(A-iB)_z=-x(\alpha_1^T-i\alpha_2^T)+(A+iB)({a}-ib)+(A-iB)(\hat{a}+i\hat b).\]
As a consequence, (b) holds if and only if
\[\begin{split}
    x_z&=(A-iB)\alpha_1, \\
    (A-iB)_z&=(A-iB)(\hat{a}+i\hat b).
\end{split} \]
Hence (b) holds if and only if  (c) holds, given the fact $\I=Span_{\C}\{e_j-ie_{m+j},1\leq j\leq m\}$.
\end{proof}
We  call $\pi_{st}(F)$ {\em a $\JJ$ twistor curve } if it is either $\JJ$ holomorphic or anti-holomorphic. We have the following well known results \cite{BR,Calabi, Ejiri,Helein-B}.
For later applications, we include a proof here.
 \begin{theorem} \label{thm-tw} If $\pi_{st}(F)$ is a   $\JJ$ twistor curve, then $x$ is   totally isotropic.
\end{theorem}
\begin{proof}
Assume $\pi_{st}(F)$ is  $\JJ$ holomorphic, and set $E_j=e_j-ie_{m+j}$, $1\leq j\leq m$, as before. By (1) of Proposition \ref{prop-J1}, we obtain \[
\begin{split}
    x_z&=(A-iB)\alpha_1=\sum (\cdots)E_{j},\\
{E}_{lz}&=\sum(\cdots){E}_{j},~ 1\leq l\leq m.
\end{split}\]
 Consequently, $x$ is totally isotropic, since for any $k \geq 1$ we have
 \begin{equation}
     \label{eq-xzz}
     x_{z}^{(k)}=\sum(\cdots){E}_{j},~\hbox{ for any } k\geq 1.
 \end{equation}
A similar argument applies to the case where $\pi_{st}(F)$ is $\JJ$-anti-holomorphic.
\end{proof}

\subsection{Totally isotropic Willmore $2$-spheres and  $\FH$ twistor curves}
\begin{definition}  Let $\ff$ be a $\JJ$ twistor curve with $x=\pi_t(\ff)$ being a conformal immersion in $S^{2m}$. We say that $\ff$ is  a $\FH$ lift of $x$ if the following statements hold.
\begin{enumerate}
    \item When $\ff$ is $\JJ$ holomorphic, let $F=(x,e_1,\cdots,e_m, e_{m+1},\cdots,e_{2m})$ be a local lift of $\ff$ satisfying
    $x_z=(\cdots)E_1, \hbox{ and } ~E_j:=e_j-ie_{m+j}, 1\leq j\leq m.$ Then we have
    \begin{equation}\label{eq-H-horizontal}
 \phi_{\bar{z}}=0\mod{ \{\I_2,x_z,x_{\bar z}\}},\hbox{ for any }\phi\in\Gamma(\I_2), \hbox{ where }  (\I_2)_p:=Span_{\C}\{E_2,\cdots,E_m\}|_p.
    \end{equation}
\item If $\ff$ is $\JJ$ anti-holomorphic, let $F=(x,e_1,\cdots,e_m, e_{m+1},\cdots,e_{2m})$ be a local lift of $\ff$ with
  \[x_{\zb}=(\cdots)E_1, \hbox{ and } ~E_j:=e_j-ie_{m+j}, 1\leq j\leq m.\]
  Define the bundle $\I_2$ to be the bundle spanned by $\{E_2,\cdots,E_m\}$ at $x$. Then we have
    \begin{equation}\label{eq-H-horizontal-2}
 \phi_{z}=0\mod{ \{\I_2,x_z,x_{\bar z}\}},\hbox{ for any }\phi\in\Gamma(\I_2).
    \end{equation}
\end{enumerate}
\end{definition}
It is easy to check that the above definition of $\FH$ twistor curves is well defined.
\begin{remark} \
\begin{enumerate}
\item Recall that $\ff$ is called horizontal if
 \begin{equation}\label{eq-horizontal}
 \phi_{\zb}=0\mod{ \{\I,x\}},\hbox{ for any }\phi\in\Gamma(\I)
    \end{equation}
    holds \cite{Calabi, Helein-B}. So a horizontal twistor curve is also a $\FH$ twistor curve. Note that both the properties of existence of horizontal twistor lift and $\FH$ twistor lift are M\"obius invariant.
    \item
 Calabi Theorem \cite{Calabi} states that the twistor projection $\pi_{st}(\ff)=x$ of a horizontal twistor curve $\ff$ is a totally isotropic minimal surface in $S^{2m}$ and
all minimal $2$-spheres in $S^{2m}$ come from  twistor projections of horizontal twistor curves. In the spirit of Calabi, we have the following theorem for totally isotropic Willmore $2$-spheres.
\end{enumerate}
\end{remark}
 \begin{theorem} \label{thm-total-iso}\
 \begin{enumerate}
\item Let $x$ be a totally isotropic Willmore $2$-spheres  in $S^{2m}$. Then there exists a twistor lift $\ff$ which is  a $\FH$ twistor curve.

\item
The twistor projection $\pi_{st}(\ff)=x$ of  a $\FH$ twistor curve $\ff$ is a totally isotropic Willmore surface in $S^{2m}$ when $x$ is immersed.
 \end{enumerate}\end{theorem}
\begin{proof}
(1) Set $Y:=e^{-\omega}(1,x)$ with $e^{2\omega}=2|x_z|^2$. Then under the global assumption of $2$-sphere, by Theorem \ref{thm-m-iso}, there exists some  $\Pi_{k}$ such that $\Pi_{k+1}=\Pi_{k}$. Set $m=\dim \Pi_k+1$. Then there exists a basis $\{\hat{E}_2,\cdots, \hat{E}_m\}$ of $\Pi_{k}$ such that, for $~ \ 2\leq j\leq m$,
\[\begin{split}
    \dz \hat{E}_j&=0\mod \{\hat{E}_2,\cdots,\hat{E}_m\},\\
    \db \hat{E}_j&=0\mod \{\hat{E}_2,\cdots, \hat{E}_m\}.
\end{split}\]
Hence we have for $2\leq j\leq m$
\[\begin{split}
    \hat{E}_{jz}&=0\mod \{\hat{E}_2,\cdots, \hat{E}_m\},\\
\hat{E}_{j\bar z}&=0\mod \{Y_{\bar{z}}, N,  \hat{E}_2,\cdots, \hat{E}_m\}.
\end{split}\]
Assume \[\hat{E}_j=(0,E_j)+(\cdots)Y\]
such that
\[\begin{split}
   & E_1=0\mod{x_z}\\
&E_j\perp\{ x,x_z,x_{\zb}\}, 2\leq j\leq m.
\end{split}\]
Then we have
\[\begin{split}
    E_{jz}&=0\mod \{x_z,x_{\zb}, {E}_2,\cdots, E_m\},\\
    E_{j\zb}&=0\mod \{x_z,x_{\zb}, {E}_2,\cdots, E_m\},\ 2\leq j\leq m.
\end{split}\]
Assume ${E}_j=e_j-ie_{m+j}$ for $1\leq j\leq m$, and set \[\tilde{F}=(x,e_1,\cdots, e_{2m})\in O(2m+1).\]

If $\det \tilde{F}=1$, we have  $\ff=\pi_{st}(F)$ is  a $\FH$ $\JJ$ holomorphic curve.

If $\det \tilde{F}=-1$, we have  $\ff=\pi_{st}(F)$ is  a $\FH$ $\JJ$ anti-holomorphic curve.

(2) Assume without of generality $\ff$ is $\JJ$ holomorphic.  By definition, we have a framing $\{x, E_1, \cdots E_m,  \bar{E}_1, \cdots \bar{E}_m\}$ by use of  $\ff$, with $x_z=0\mod\{E_1\}$ and $x_{z}^{(j)}=0\mod\{E_1,\cdots, E_m\}$ for all $j\geq2$ \eqref{eq-xzz}.

Now set  $Y:=e^{-\omega}(1,x)$ and $\hat{E}_1=(0, E_1)$, $\hat{E}_j=(0, E_j)+(\cdots)Y$ such that $\hat{E}_j\perp (0,x_{z\bar{z}})$, where $e^{2\omega}=2|x_z|^2$. Then by  $x_{z}^{(j)}=0\mod\{E_1,\cdots, E_m\}$, $1\leq j\leq m$, we have that
\[Y_{zz}=0\mod \{Y,\hat{E}_2,\cdots,\hat{E}_m\}.\]
So $\kappa=0\mod \{\hat{E}_2,\cdots,\hat{E}_m\}$.
Hence  $D_{\zb}\kappa=\sum_{2\leq j\leq m}(\cdots)\hat{E}_{jz}$.
By \eqref{eq-H-horizontal}, we have
$\db\kappa=0\mod  \{\hat{E}_2,\cdots,\hat{E}_m\}$. Similarly, we have $
\db\db\kappa=0\mod  \{\hat{E}_2,\cdots,\hat{E}_m\}.$ Therefore
\[\db\db\kappa+\frac{\bar{s}}{2}\kappa=0\mod  \{\hat{E}_2,\cdots,\hat{E}_m\}.\]
So $\db\db\kappa+\frac{\bar{s}}{2}\kappa$ is totally isotropic. Since the Codazzi equation of $Y$ shows $Im(\db\db\kappa+\frac{\bar{s}}{2}\kappa)=0$, we have $\db\db\kappa+\frac{\bar{s}}{2}\kappa=0$ and hence $x$ is a totally isotropic Willmore surface.
\end{proof}
\begin{remark}
  In \cite{Fer-03,Fer-12}, a beautiful characterization of horizontal twistor curves is introduced and used successfully in the discussion of the moduli spaces of harmonic $2$-spheres in $S^{2m}$. One can adapt the treatment of \cite{Fer-03,Fer-12} to discuss $\FH$ twistor curves as well as the corresponding Willmore $2$-spheres in $S^{2m}$, which will be the topic of another publication.
\end{remark}
\begin{remark}
    From the definition of $\FH$ twistor curves, we see that one can also use the $1$-flag bundle $\FS$ to describe totally isotropic Willmore $2$-spheres. To be concrete, for $p\in S^{2m}$, set
\[IS_{p}:=\{\hbox{oriented $m$-dimensional isotropic subspace of } T_pS^{2m}\otimes\C\}.\]
 We define the flag bundle $\FF_1(S^{2m})$ as follows:
\begin{equation}\label{eq-flag}
\FF_1(S^{2m}):=\coprod_{p\in S^{2m}}\{(\ZZ_{p,1},\ZZ_p)| \ZZ_p\in IS_{p},\ZZ_{p,1}\subset\ZZ_p \hbox{ is an $1$-dim subspace}\}.\end{equation}
Then we can also define the complex structure on $\FS$ and use lifts into $\FS$ to describe totally isotropic Willmore $2$-spheres. Note that in \cite{Burst}, Burstall introduced flag manifolds for harmonic tori in spheres and complex projective spaces, which generalized the classical work of Eells and Salamon \cite{ES}.
\end{remark}
\vspace{2mm}
\section{Strictly $k$-isotropic Willmore $2$-spheres: the classification theorem}

In this section, we will present the classification theorem for non-S-Willmore, strictly $k$-isotropic Willmore $2$-spheres, together with an outline of proofs. The technical proofs will be given in the next three sections.

\begin{theorem}\label{thm-1}
Let $y^k:S^2\to S^n$ be a full, strictly $k-$isotropic, Willmore immersion which is not S-Willmore.
Then there exists a sequence of adjoint Willmore surfaces
\begin{equation}\label{eq-W-s1}
y^j:M_j\rightarrow S^n, ~ k\leq j\leq m,~ M_j\subseteq S^2,
\end{equation}
such that
\begin{enumerate}
    \item $y^{j+1}$ is strictly $(j+1)-$isotropic Willmore with $M_{j+1}$ being an open dense subset of $M_{j}$, and $y^{j+1}$ is an adjoint Willmore surface of $y^j$,   $k\leq j\leq m-1$;
    \item  $y^m$ is M\"obius congruent to a (branched) strictly $m$-isotropic minimal surface in $\R^n$.
\end{enumerate}
\end{theorem}


We outline the idea of proofs as follows.
\begin{enumerate}
\item The conformal Gauss map $f^{k}_{0}$ of $y^k$ has a Frenet-bundle structure and induces a \emph{global, real}, harmonic sequence $\{f^{k}_{j}\}_{0\leq j\leq k+k'}$ with $f^{k}_{j}\subset f^{k}_{j+1}$ and \[Rank_{\C}\partial f^{k}_{k+k'}=1.\] Here $f^{k}_{j}$ can be viewed as higher order conformal Gauss maps (Theorem \ref{thm-line} ).
    \item Set $h_0=(f^{k}_{k+k'})^{\perp}$. Then $Rank_{\C}\partial h_0=1$ and $h_0$ is a totally isotropic harmonic map which induces a harmonic sequence $\{h_{j}\}_{0\leq j\leq m+1}$ on some open dense subset of $S^2$, with \[h_{m+1}^{\perp}=f^{m}_{0}\cap h_{m+1}^{\perp}=\hbox{Span}_{\R}\{Y^m,Y^{m+1}\}\] and $y^m=[Y^m]$ M\"obius congruent to a  (branched) strictly $m$-isotropic  minimal surface in $\R^n$ and $[Y^{m+1}]=const$
    (Theorem \ref{thm-h0} and Proof of Theorem \ref{thm-1}). Here $f^{m}_{0}$ is the conformal Gauss map of $y^m$.
        \item
        The sequence  $\{h_j\}_{0\leq j\leq m+1}$ defines a flag structure which determines $y^j=[Y^j]$ inductively: for each $k\leq j\leq m$,  we have
        \begin{equation}
            \label{eq-inter}
        \hbox{Span}_{\R}\{Y^j,Y^{j+1}\}=f^{j}_{0}\cap h^{\perp}_{j+1},
        \end{equation}with  $f^{j}_{0}$ being the conformal Gauss map of $y^j=[Y^j]$ (Theorem \ref{thm-initial} and Theorem \ref{thm-inductive}), and with $Y^{j+1}$ adjoint to $Y^{j}$.
\end{enumerate}

\begin{remark}Note that some of the possible branched/singular points of $Y^{j+1}$ come from the singular points of $h_{j+1}^{\perp}$ by \eqref{eq-inter}. This means that probably $Y^{j+1}$ is not globally well-defined, which turns out to be the key difficulty of this work. Here we successfully build \eqref{eq-inter}  to avoid discussing the singular points of $[Y^{j+1}]$ for $k\leq j\leq m-1$.
\end{remark}
\vspace{2mm}

 To derive $\{f^{k}_{j}\}$, we need to show the vanishing of some global holomorphic differentials.
The following lemma is a key technical tool to control the frame involved and to remove the possible singularities.
\begin{lemma}\label{lemma-linebundle} Let $\VV$ be a Riemannian sub-bundle of $M\times\R^{n+2}_1$ over a Riemann surface $M$. Let $\WW$ be a non-totally-isotropic, rank-$2$ holomorphic subbundle of $\VV\otimes{\C}$, such that:
\begin{enumerate}
\item The differential $\mathcal{X}\dd z^{k}\not\equiv0$ is a global holomorphic differential over $M$, taking values in $\wedge^2\WW\subset \wedge^2\C^{n+2}_1$;
\item The holomorphic differential $\<\mathcal{X},\mathcal{X}\>\dd z^{2k}$ vanishes on $M$.
\end{enumerate}
Then
\begin{enumerate}
\item There exists a unique isotropic, holomorphic line bundle $\LL\subset \WW$ over $M$;
\item For any $p\in M$, there exists a neighborhood $U_p'$ of $p$ and a basis $\{\tilde X_1,\tilde X_2\}$ of $\WW$ on $U_p$ such that $\LL=\hbox{Span}_{\C}\{\tilde X_1\}$ and $\tilde X_2\perp \tilde X_1,\bar{\tilde{ X}}_1.$\end{enumerate}
\end{lemma}
\begin{proof}
Let $\nabla^{V}$ be the Levi-Civita connection on $\VV\otimes{\C}$. Let $U_p\subset M$ be a neighborhood of a point $p\in M$, where $\mathcal{X}=X_1\wedge X_2$ with $X_j$ sections of $\WW$ such that
$\nabla^{V}_{\bar {z}}X_j=0\mod\{X_1,X_2\}$. By Chern's Lemma \cite{Chern}, there exists $\hat{X}_1$ and $\hat{X}_2$ such that
$X_j=\sum_{1\leq l\leq 2}\hat{a}_{jl}\hat{X}_l$ with $\hat{X}_1\wedge\hat{X}_2\neq0$. Since $\WW$ is not totally isotropic on an open dense subset, without loss of generality we can assume that $\<\hat{X}_2,\hat{X}_2\>\neq0$ on an open dense subset of $U_p$.

On the points $\<\hat{X}_2,\hat{X}_2\>\neq0$, we can simply set
\[\tilde{X}_1=\hat{X}_1-\frac{\<\hat{X}_1,\hat{X}_2\>}{\<\hat{X}_2,\hat{X}_2\>}\hat{X}_2,\ \tilde{X}_2=\hat{X}_2-\frac{\<\hat{X}_1,\bar{\tilde{X}}_1\>}{\<\tilde{X}_1,\bar{\tilde{X}}_1\>}\tilde{X}_1\]
so that $\tilde{X}_1\wedge\tilde{X}_2=\hat{X}_1\wedge\hat{X}_2\neq0$ and $\<\tilde{X}_1,\hat{X}_2\>=0$. By assumption (2), we obtain $\<\tilde{X}_1,\tilde{X}_1\>=0$ and hence $\<\tilde{X}_1,\tilde{X}_2\>=0$. Note that $\<\tilde{X}_2,\tilde{X}_2\>=\<\hat{X}_2,\hat{X}_2\>\neq0$. So $\LL:=\hbox{Span}_{\C}\{\tilde{X}_1\}$ is the unique isotropic line bundle on the subset $\<\hat{X}_2,\hat{X}_2\>\neq0$.


For the case $\<\hat{X}_2,\hat{X}_2\>|_p=0$, we have that the inner products on its neighborhood $U_p$,
\[\{\<\hat{X}_1,\hat{X}_1\>,\<\hat{X}_2,\hat{X}_2\>, \<\hat{X}_1,\hat{X}_2\>\}\] satisfy some psuedo-holomorphic equations  since  $\nabla^{V}_{\bar {z}}\hat{X}_j=0\mod\{\hat{X}_1,\hat{X}_2\}$. Since $\<\hat{X}_1,\hat{X}_2\>^2=\<\hat{X}_1,\hat{X}_1\>\<\hat{X}_2,\hat{X}_2\>$ by assumption $(2)$, together with Chern's Lemma \cite{Chern}, we have that either $\<\hat{X}_1,\hat{X}_2\>/\<\hat{X}_1,\hat{X}_1\>$ or $\<\hat{X}_1,\hat{X}_2\>/\<\hat{X}_2,\hat{X}_2\>$ has a finite limit at $p$. So on an open neighborhood $U_p'\subset U_p$ of $p$, either $\<\hat{X}_1,\hat{X}_2\>/\<\hat{X}_1,\hat{X}_1\>$ or $\<\hat{X}_1,\hat{X}_2\>/\<\hat{X}_2,\hat{X}_2\>$ is finite.
For the first case, set
\[\tilde{X}_1=\hat{X}_2-\frac{\<\hat{X}_1,\hat{X}_2\>}{\<\hat{X}_1,\hat{X}_1\>}\hat{X}_1,\ \tilde{X}_2=\hat{X}_1-\frac{\<\hat{X}_1,\bar{\tilde{X}}_1\>}{\<\tilde{X}_1,\bar{\tilde{X}}_1\>}\tilde{X}_1.\]
For the second case, set
\[\tilde{X}_1=\hat{X}_1-\frac{\<\hat{X}_1,\hat{X}_2\>}{\<\hat{X}_2,\hat{X}_2\>}\hat{X}_2,\ \tilde{X}_2=\hat{X}_2-\frac{\<\hat{X}_2,\bar{\tilde{X}}_1\>}{\<\tilde{X}_1,\bar{\tilde{X}}_1\>}\tilde{X}_1.\]
Set the isotropic line bundle $\LL:=\hbox{Span}_{\C}\{\tilde{X}_1\}$ on $U'_p$. So $\LL$ extends to a global isotropic holomorphic line bundle over $M$. This finishes the proof.
\end{proof}

\vspace{4mm}
\section{The harmonic sequence $\{f^{k}_{j}\}$ and the universal Gauss map $h_0$}\label{s-fkj}

 Let $y^k$ be a full, strictly $k$-isotropic Willmore $2$-sphere in $S^n$ which is not S-Willmore, that is
 \[\<\kappa,\kappa\>=\cdots=\<\dz^{(k-1)}\kappa,\dz^{(k-1)}\kappa\>\equiv0\not\equiv\<\dz^{(k)}\kappa,\dz^{(k)}\kappa\>, \hbox{ and } \db\kappa\wedge\kappa\nt0.\]
 When $k=0$, we have simply $\<\kappa,\kappa\>\not\equiv0$.

 Let $f^{k}_{0}$ be the conformal Gauss map of $y^k$. We shall derive a global harmonic sequence $\{f^{k}_{j}\}$, which defines the universal Gauss map $h_0=(f^{k}_{k+k'})^{\perp}$ on $S^2$, and also induces a simple framing of $y^k$. In this process we shall use the local section induced by  $\dz^{(k)}\kappa$, and we always have to avoid the singular set
 \begin{equation}\label{eq-Sing}
     Sing(y^k):=\{p\in S^2|\<\dz^{(k)}\kappa,\dz^{(k)}\kappa\>(p)=0 \}
 \end{equation}
of $y^k$ on $S^2$. Fortunately, our construction extends across this singular set to the whole $S^2$.

When $k=0$, we get $f^k_k=f^0_0$ automatically. When $k\geq 1$,  we can set inductively
 \[f^{k}_{j+1}:=\DD f^{k}_{j}=f^{k}_j\oplus Re(\partial f^k_j)=f^{k}_0\oplus Re( \Pi_{j-1}),  \hbox{ for } 0\leq j\leq k-1,\]
 to derive a (conformal) harmonic sequence $\{f^k_j,~0\leq j\leq k\}$ by Theorem \ref{thm-db-1}, since $\partial f^k_j\cap\Q=\partial f^k_j$ and $\bar\partial(\partial f^k_j)\subset f^k_j\otimes\C$ by Theorem \ref{thm-m-iso}.

 Now consider $f^{k}_k$ for $k\geq0$. Since $y^k$ is strictly $k$-isotropic, $\partial f^{k}_k$ is not totally isotropic and has rank $1$ or $2$. We discuss these two cases separately.

 \begin{proposition}\label{prop-rank1} Let $y^k$ be a strictly $k$-isotropic Willmore $2$-sphere in $S^n$ which is not S-Willmore, with $Rank_{\C}\partial f^{k}_{k}=1$. Then $y^k=[Y^k]$ has a Frenet-bundle structure  \[\{f^{k}_{0}, \ZZ^k, h_0\}, \hbox{ with } \mathcal{Z}^k=\Pi_{k-1}, h_0=(f^{k}_{k})^{\perp}:S^2\rightarrow Gr_{r}\RRR.\]
 Here $r=n-2m-4$ with $m:=\dim\ZZ^k$.
\end{proposition}
\begin{proof}
    It is a direct corollary of Lemma \ref{lemma-ind2}.
\end{proof}

When $Rank_{\C}\partial f^{k}_{k}=2$, we have
\begin{theorem}[Thereom A, structure of $\{f^k_{l}\}$]\label{thm-line} Let $y^k$ be a strictly $k$-isotropic Willmore $2$-sphere in $S^n$ which is not S-Willmore, with $Rank_{\C}\partial f^{k}_{k}=2$. Then
\begin{enumerate}
    \item There exists a sequence of isotropic holomorphic line  bundles $\{\LL_j,1\leq j\leq k'\}$ over $S^2$ such that the following statements hold true inductively.
    \begin{enumerate}
        \item $\LL_{j}\subset\partial f^{k}_{k+j-1}$ for $1\leq j \leq k'$;
        \item  The map $f^{k}_{k+j}:=\DD f^{k}_{k+j-1}=f^{k}_{k+j-1}\oplus Re(\LL_{j})$ is globally harmonic on $S^2$ with $\hbox{Rank}_{\C}\partial f^{k}_{k+j}=2$ for $1\leq j<k'$, and   $\hbox{Rank}_{\C}\partial f^{k}_{k+j}=1$ for $j=k'$.
    \end{enumerate}
\item  $y^k=[Y^k]$ has a Frenet-bundle structure $\{f^{k}_{0}, \ZZ^k, h_0\}$ with
\[\mathcal{Z}^k=\Pi_{k-1}\oplus \LL_{1}\oplus \cdots\oplus \LL_{k'},\hbox{ and } h_0:=(f^{k}_{k+k'})^{\perp}:S^2\rightarrow Gr_{r}\RRR.\]
Here $r:=n-2m-4$ and $m:=2k+k'$.
\end{enumerate}
\end{theorem}

\begin{proof} We first make some notations.
\begin{itemize}
  \item Let $\nabla$ denote the induced connection on $(f^{k}_{k})^{\perp}\otimes\C$ by the normal connection $D$ on $V^{\perp}\otimes\C=(f^{k}_{0})^{\perp}\otimes\C$. Let $\nabla_j$ denote the induced connection on $(f^{k}_{k+j})^{\perp}\otimes\C$ by the normal connection $D$ on $V^{\perp}\otimes\C$.
  \item
For any $X\in\Gamma((V^{\perp})\otimes{\C})$,  define $X^{\perp_j}\in \Gamma((f^{k}_{k+j})^{\perp}\otimes{\C})$ by
\[X^{\perp_j}:=X\mod f^{k}_{k+j}\otimes{\C}.\]
If $j=0$, $X^{\perp_0}=X\mod f^{k}_{k}\otimes{\C}$ and if $k=0$ furthermore, we have $X^{\perp_0}=X$.
    \item  Set 
\begin{equation}
    \pz:=(D^{(k)}_{z}\kappa)^{\perp_0},
\end{equation}
\begin{equation}\label{eq-def-phi1}
\phi_1:=\nzb\pz+\frac{\bar\mu}{2}\pz=(D_{\bar z}\pz+\frac{\bar\mu}{2}\pz)^{\perp_0}, \hbox{ with } \bar\mu=-\frac{2\<\db\pz,\pz\>}{\<\pz,\pz\>}.
\end{equation}
We define $\phi_{j+1}$ inductively for $j\geq1$:
\begin{equation}\phi_{j+1}:=(\nzz\phi_j)^{\pj}-\frac{\<(\nzz\phi_j)^{\pj},\pz^{\pj}\>}{\<\pz^{\pj},\pz^{\pj}\>}\pz^\pj=(\nabla_{j,z}\phi_j)^{\pj}-\frac{\<(\nabla_{j,z}\phi_j)^{\pj},\pz^{\pj}\>}{\<\pz^{\pj},\pz^{\pj}\>}\pz^\pj.
\end{equation}
\end{itemize}

The theorem follows directly from the following claims for $1\leq j\leq k'$.
\begin{enumerate}
    \item[{\bf{Claim (1)}}] The isotropic section $\phi_j$ is defined on an open dense subset of $S^2$, with
    \begin{equation}
        \label{eq-nzb-phij}
        \nzb\phi_j=\frac{\bar{\mu}}{2}\phi_j+\sum_{l\leq j-1}(\cdots)\phi_{l}.
    \end{equation}
It defines an isotropic line bundle $\LL_j=Span_{\C}\{\phi_j\}$, which extends holomophically and globally to $S^2$.
    \item[{\bf{Claim (2)}}] The bi-vector valued differential \[\X_{j}\dd z^{2k+j+3}:=(\nz_{j,z}\phi_j)\wedge\pz^{\pj}\dd z^{2k+j+3}=(\nzz\phi_j)^{\pj}\wedge\pz^{\pj}\dd z^{2k+j+3}\] extends to a global holomorphic differential on $S^2$ and \[\Theta_j\dd z^{4k+2j+6}:=\<(\nzz\phi_j)^{\pj}\wedge\pz^{\pj},(\nzz\phi_j)^{\pj}\wedge\pz^{\pj}\>\dd z^{4k+2j+6}\] extends to a global holomorphic differential on $S^2$.
\end{enumerate}

Let us prove the claims by induction.\vspace{1mm}

{\bf Step 1: the initial part.} By Theorem \ref{thm-CW1},
$\partial f^{k}_{k}$ is a global holomorphic subbundle. To derive $\partial f^{k}_{k}$, by definition of $f^{k}_k=f^{k}_{k-1}\oplus Re(\partial f^{k}_{k-1})$, we need only to consider $(D_z^{(k-1)}\kappa)_z$ and  $(\db D_z^{(k-1)}\kappa)_z$. It is direct to check that
\[((D_z^{(k-1)}\kappa)_z)^{\perp_0}=(D_z^{(k)}\kappa)^{\perp_0}=\pz,\]
\[((\db D^{(k-1)}_{z}\kappa)_z)^{\perp_0}=(\dz(\db D^{(k-1)}_{z}\kappa))^{\perp_0}=(\db(D^{(k)}_{z}\kappa))^{\perp_0}=(\db\pz)^{\perp_0}=\nz_{\zb}\pz,\]
by the Ricci equation \eqref{ricci} and Lemma \ref{lemma-ind2}. So $\partial f^{k}_{k}=\hbox{Span}_{\C}\{\pz,\nz_{\zb}\pz\}$ on an open dense subset of $S^2$ where $\nzb\pz\wedge\pz\neq0$ since $Rank_{\C}\partial f^{k}_{k}=2$.\vspace{1mm}

Let us first prove the generalized Ricci equation: for $\psi\in\Gamma((f^k_k)^{\perp}\otimes\C)$ with $\psi\perp\pz,\nzb\pz$,
\begin{equation}
    \label{eq-Ricci-nabla}
    \nzb\nzz\psi- \nzz\nzb\psi=(\cdots)\pz+(\cdots)\nzb\pz.
\end{equation}
To this end, if $k=0$, we have $\pz=\kappa$ and it is the Ricci equation \eqref{ricci}. If $k\geq 1$,
 for any $\tilde\phi\in\Gamma(\Pi_{k-1})$, we have  $ \<\psi,\tilde\phi\>=0$ and
$-\<D_z\psi,\tilde\phi\>=\<\psi,D_z\tilde\phi\>=\<\psi,(\cdots)\pz+(\cdots)\nzb\pz)\>=0$.
So,
\begin{equation}\label{eq-psidz}
    D_z\psi=\nzz\psi \mod\{\Pi_{k-1}\}.
\end{equation}
Since $\Pi_{k-1}$ is a holomorphic bundle, we have
\[\begin{split}
\nzb\nzz\psi&=\nzb\left( D_z\psi \mod\{\Pi_{k-1}\}\right)=\db D_z\psi \mod\{\Pi_{k-1}\}.
\end{split}\]
Since $\partial f^{k}_{k}=\hbox{Span}_{\C}\{\pz,\nz_{\zb}\pz\}$, we obtain
$\nzz \nzb\psi= D_z\db\psi +(\cdots)\pz+(\cdots)\nzb\pz \mod\{\Pi_{k-1}\}.$
Together with the Ricci equation \eqref{ricci}, we obtain \eqref{eq-Ricci-nabla}.\vspace{2mm}

Now consider
\begin{equation}
    \label{eq-chi-0-0}
\mathcal{X}_0\dd z^{2k+3}=\nzb\pz\wedge\pz \dd z^{2k+3}.
\end{equation}
 Since $\kappa\frac{\dd z^2}{\abs{\dd z}}$ is a global invariant,
 $\pz\frac{\dd z^{k+2}}{\abs{\dd z}}$ is a global invariant by Theorem \ref{thm-m-iso}. So $\mathcal{X}_0\dd z^{2k+3}$ is a global invariant. It is holomorphic by the generalized Willmore equation \eqref{eq-dbdb}: $\nzb\nzb\pz+\frac{\bar s}{2}\pz=0$. Hence, its determinant, the $(4k+6)$-form \[\Theta_0\dd z^{4k+6}=\left(\<\nzb\pz,\pz\>^2-\<\nzb\pz,\nzb\pz\>\<\pz,\pz\>\right)\dd z^{4k+6}\]
is a holomorphic form on $S^2$ and vanishes. Since $\partial f^{k}_{k}$ is not totally isotropic, by Lemma \ref{lemma-linebundle}, there exists a unique isotropic holomorphic line bundle $\LL_1\subset \partial f^{k}_{k}$ over $S^2$. Moreover, since $\Theta_0\dd z^{4k+6}\equiv0$, the quadratic equation $\<\nzb\pz+\frac{\bar\mu}{2}\pz,\nzb\pz+\frac{\bar\mu}{2}\pz\>\equiv0$ has a unique solution on $S^2\setminus Sing(y^k)$ (recall \eqref{eq-Sing}),  denoted by $\mu$. 
So  $0=2\<\nzb\nzb\pz+\frac{\bar\mu_{\zb}}{2}\pz+\frac{\bar\mu}{2}\nzb\pz,\nzb\pz+\frac{\bar\mu}{2}\pz\>=(\bar\mu_{\bar z}-\frac{\bar\mu^2}{2}-\bar s)\<\pz,\nzb\pz+\frac{\bar\mu}{2}\pz\>$, by
 the generalized Willmore equation $\nzb\nzb\pz+\frac{\bar s}{2}\pz=0$. So $\mu$ satisfies the Ricatti equation \eqref{eq-theta}: $\bar\mu_{\bar z}-\frac{\bar\mu^2}{2}-\bar s=0$.
So,
\[\phi_1:=\nzb\pz+\frac{\bar\mu}{2}\pz\] is a local basis of $\LL_1$ when $\phi_1\neq0$ and $\<\pz,\pz\>\neq0$, with
\[\<\phi_1,\phi_1\>=\<\phi_1,\pz\>\equiv0.\] We have
 $\nzb\phi_1=\nzb\nzb\pz+\frac{\bar\mu_{\bar z}}{2}\pz+\frac{\bar\mu}{2}\nzb\pz=\frac{\bar\mu}{2}\phi_1,$
by the above generalized Willmore equation and the Ricatti equation, which proves \eqref{eq-nzb-phij}.
Note that $\phi_1$ might have singularities on $Sing(y^k)$.

Next, consider Claim (2). Since $\nzb\phi_1=\frac{\bar\mu}{2}\phi_1$, we have $f^{k}_{k+1}$ is harmonic  by Theorem \ref{thm-db-1}, since $\partial f^{k}_{k} \cap\Q=\LL_1$ and \eqref{eq-nzb-phij} holds.  Furthermore, by Theorem \ref{thm-CW1}, $\partial f^{k}_{k+1}$ is a global holomorphic sub-bundle of $(f^{k}_{k+1})^{\perp}\otimes{\C}$ on $S^2$.

If $Rank_{\C} \partial f^{k}_{k+1}=1$, the proof is finished by setting $k'=1$.  If $Rank_{\C} \partial f^{k}_{k+1}=2$,  on an open dense subset of $S^2$, we represent \[\partial f^{k}_{k+1}=\hbox{Span}_{\C}\{(\nzz\phi_1)^{\perp_1}, \pz^{\perp_1}\}\] by similar computations.
Consider
\begin{equation}
    \label{eq-xi-1}\mathcal{X}_1\dd z^{2k+4}=(\nzz\phi_1)^{\perp_1}\wedge\pz^{\perp_1}\dd z^{2k+4}.
\end{equation}
It is independent of the choice of $z$ as before. To show it is holomorphic, we first have
\[\nabla_{1,\Bar{z}}\pz^{\perp_1}=\nzb(\pz-(\cdots)\phi_1)-(\cdots)\phi_1=\phi_1-\frac{\bar\mu}{2}\pz+(\cdots)\phi_1=-\frac{\bar\mu}{2}\pz^{\perp_1}.\]
To compute $\nabla_{1,\Bar{z}}\left((\nzz\phi_1)^{\perp_1}\right)$, we notice $\<\nzz\phi_1,\phi_1\>=0$ and $\<\nzb\nzz\phi_1,\phi_1\>=-\<\nzz\phi_1,\nzb\phi_1\>=0$.
By the generalized Ricci equation \eqref{eq-Ricci-nabla}, we obtain (since $\nzb\pz=\phi_1+(\cdots)\pz$)
\[\begin{split}
    \nabla_{\bar{z}}\left((\nzz\phi_1)^{\perp_1}\right)&=\nzb\nzz\phi_1+(\cdots)\phi_1=\nzz\nzb\phi_1+(\cdots)\phi_1+(\cdots)\pz\\
    &=\nzz(\frac{\bar\mu}{2}\phi_1)+(\cdots)\phi_1+(\cdots)\pz=\frac{\bar\mu}{2}(\nzz\phi_1)+(\cdots)\pz+(\cdots)\phi_1.
\end{split}\]
After projection we obtain
\[\nabla_{1,\Bar{z}}\left((\nzz\phi_1)^{\perp_1}\right)=\frac{\bar\mu}{2}(\nzz\phi_1)^{\perp_1}+(\cdots)\pz^{\perp_1}.\]
So we obtain $\nabla_{1,\Bar{z}}((\nzz\phi_1)^{\perp_1}\wedge\pz^{\perp_1})=0$.

Now consider the definition of $\mathcal{X}_1\dd z^{2k+4}$ in \eqref{eq-xi-1} at the possible singular points of $\phi_1$. For the bundle $\partial f^k_k$, by Lemma \ref{lemma-linebundle}, for any $p\in S^2$, there exist $X_{01},X_{02}$ with $X_{01} \perp \{X_{01}, X_{02},\bar{X}_{02}\}$ and  $X_{01}\wedge X_{02}\neq0$ on a neighborhood $U_p$ of $p$ such that
\begin{equation}\label{eq-pz1}
\nzb\pz=\sum_{1\leq l\leq 2} a_{0l}X_{0l},\ \pz=\sum_{1\leq l\leq 2} b_{0l}X_{0l}.
\end{equation}
Since $\pz_1$ is isotropic, $\pz_1\parallel X_{01}$. So \[\pz_1=\nzb\pz-(\cdots)\pz=\frac{a_{01}b_{02}-a_{02}b_{01}}{b_{02}}X_{01}\] by \eqref{eq-pz1}. Note that the possible higher singular term $(\frac{a_{01}b_{02}-a_{02}b_{01}}{b_{02}})_zX_{01}$  of $\nzz\pz_1$ is eliminated in $(\nzz\pz_1)^{\perp_1}$ after projection. Now we have
\[(\nzz\pz_1)^{\perp_1}\wedge\pz^{\perp_1}=(a_{01}b_{02}-a_{02}b_{01})(\nzz X_{01})^{\perp_1}\wedge X_{02}.\]
Since both
$(a_{01}b_{02}-a_{02}b_{01})$
and $(\nzz X_{01})^{\perp_1}\wedge X_{02}$ are smooth,
we see that  $(\nzz\phi_1)^{\perp_1}\wedge\pz^{\perp_1}$ is smooth on the whole $U_p$. Hence $(\nzz\phi_1)^{\perp_1}\wedge\pz^{\perp_1}\dd z^{2k+4}$ defines a global holomorphic differential on $S^2$.  As a consequence, on the whole $U_p$, the differential
\[\begin{split}
\Theta_1\dd z^{4k+8}&=\<(\nzz\phi_1)^{\perp_1}\wedge\pz^{\perp_1},(\nzz\phi_1)^{\perp_1}\wedge\pz^{\perp_1}\> \dd z^{4k+8}\\
&=
(a_{01}b_{02}-a_{02}b_{01})^2\<(\nzz X_{01})^{\perp_1}\wedge X_{02},(\nzz X_{01})^{\perp_1}\wedge X_{02}\>\dd z^{4k+8}\end{split}
\]
 is  smooth on $U_p$ and holomorphic on an open dense subset of  $U_p$. So it is holomorphic on the whole $U_p$. So $\mathcal{X}_1\dd z^{2k+4}$ defines a global holomorphic differential on $S^2$. This finishes the proof of Claim (2) for $j=1$.\vspace{2mm}

{\bf Step 2: the inductive part.} Now assume the claims hold for $1\leq l\leq j$.
By Theorem \ref{thm-CW1} due to Chern-Wolfson \cite{Chern-W},
$\partial f^{k}_{k+j}$ is a global holomorphic subbundle. By the definition $f^{k}_{k+j}=f^{k}_{k+j-1}\oplus Re(\LL_{j})$ and assumption, we have $\partial f^{k}_{k+j}=Span_{\C}\{\pz^{\perp_j}, (\nzz \phi_{j})^{\perp_j}\}$.

If $Rank_{\C} \partial f^{k}_{k+j}=1$,  the proof is finished by setting $k'=j$.

If $Rank_{\C} \partial f^{k}_{k+j+1}=2$,  by Claim (2) in the inductive assumption, the holomorphic differential
\[\Theta_j\dd z^{4k+2j+6}=\<(\nzz\phi_j)^{\pj}\wedge\pz^{\pj},(\nzz\phi_j)^{\pj}\wedge\pz^{\pj}\>\dd z^{4k+2j+6}\] vanishes on $S^2$. Since $\partial f^{k}_{k+j}$ is not totally isotropic, by Lemma \ref{lemma-linebundle}, there exists a unique isotropic holomorphic line bundle $\LL_{j+1}\subset \partial f^{k}_{k+j}$ over $S^2$.
Set
\[\phi_{j+1}=(\nzz\phi_j)^{\pj}-\frac{\<(\nzz\phi_j)^{\pj},\pz^{\pj}\>}{\<\pz^{\pj},\pz^{\pj}\>}\pz^{\pj}.\]
We have
  \[\<\phi_{j+1},\phi_{j+1}\>=\<\phi_{j+1},\phi_{l}\>=\<\phi_{j+1},\pz\>=\<\phi_{l},\phi_{\tilde{l}}\>=0\not\equiv\<\phi_{j+1},\pz\>,  \hbox{ for all } 1\leq l,\tilde{l}\leq j.\]
 So $\phi_{j+1}$ is a local framing of $\LL_{j+1}$.
One verifies by induction that
 \begin{equation}\label{dbdz-phi}
      \left\{\begin{split}
      (\nzb\phi_l)^{\perp_l}&=\nzb\phi_l\mod\{\phi_1,\cdots,\phi_l\},\\
      (\nzz\phi_l)^{\perp_l}&=\nzz\phi_l \mod\{\phi_1,\cdots,\phi_l \},  1\leq l\leq j+1.
 \end{split}\right.
 \end{equation}
In fact, \eqref{dbdz-phi} holds when $l=1$ by definition and by \eqref{eq-nzb-phij}. Now assume \eqref{dbdz-phi} holds when $1\leq l\leq\Tilde{l}$. Then for any $1\leq \hat{l}\leq l$, we get again by definition and by \eqref{eq-nzb-phij}, $   \<\nzz\phi_{\Tilde{l}+1},\phi_{\hat{l}}\>=-\<\phi_{\Tilde{l}+1},\nzz\phi_{\hat{l}}\>=-\<\phi_{\Tilde{l}+1},\phi_{\hat{l}+1}\>=0,
    $
    and
    $
\<\nzb\phi_{\Tilde{l}+1},\phi_{\hat{l}}\>=-\<\phi_{\Tilde{l}+1},\nzb\phi_{\hat{l}}\>=-\<\phi_{\Tilde{l}+1},\frac{\bar{\mu}}{2}\phi_{\hat{l}}\>=0.$
So, \eqref{dbdz-phi} holds for $l=\Tilde{l}+1$.

Moreover, by \eqref{eq-Ricci-nabla} and inductive assumption of \eqref{eq-nzb-phij}, we have
\[\begin{split}
    \nzb\phi_{j+1}&=\nzb\left(\nzz\phi_j+(\cdots)\pz +(\cdots)\phi_1+\cdots+(\cdots)\phi_j\right)\\
    &=\nzz\nzb\phi_j+(\cdots)\pz +(\cdots)\phi_1+\cdots+(\cdots)\phi_j\\
    &=\frac{\bar\mu}{2}\nzz\phi_j+(\cdots)\pz +(\cdots)\phi_1+\cdots+(\cdots)\phi_j\\
    &=\frac{\bar\mu}{2}\phi_{j+1}+(\cdots)\pz^{\pj}+(\cdots)\phi_1+\cdots+(\cdots)\phi_j.\\
\end{split}
\]
Since $\<\nzb\phi_{j+1},\pz\>=-\<\phi_{j+1},\nzb\pz\>=-\<\phi_{j+1},\phi_1-\frac{\bar{\mu}}{2}\pz\>=0\not\equiv\<\pz,\pz\>$, we obtain $\nzb\phi_{j+1}=\frac{\bar\mu}{2}\phi_{j+1}\mod\{\phi_1,\cdots,\phi_j\}$. This finishes Claim (1).

Next consider Claim (2). Since $\nabla_{j,\bar{z}}\phi_{j+1}=\frac{\bar\mu}{2}\phi_{j+1}$ and $\partial f^{k}_{k+j} \cap\Q=\LL_j$, $f^{k}_{k+j+1}$ is a harmonic map by Theorem \ref{thm-db-1}.  Furthermore, by Theorem \ref{thm-CW1},
$\partial f^{k}_{k+j+1}$ is a global holomorphic subbundle of $(f^{k}_{k+j+1})^{\perp}\otimes{\C}$. By similar computations, on an open dense subset of $S^2$, we have
\[\partial f^{k}_{k+j+1}=\hbox{Span}_{\C}\left\{(\nzz\phi_{j+1})^{\perp_{j+1}}, \pz^{\perp_{j+1}}\right\}.\]

If $Rank_{\C} \partial f^{k}_{k+j+1}=1$, the proof is finished by setting $k'=j+1$.

If $Rank_{\C} \partial f^{k}_{k+j+1}=2$, consider
\[\mathcal{X}_{j+1}\dd z^{2k+j+4}=(\nzz\phi_{j+1})^{\perp_{j+1}}\wedge\pz^{\perp_{j+1}}\dd z^{2k+j+4}.\]
By \eqref{eq-nzb-phij} and \eqref{dbdz-phi}, we have
\[\nabla_{j+1,\Bar{z}}\pz^{\perp_{j+1}}=\left(\nzb\pz\right)^{\perp_{j+1}}=-\frac{\bar\mu}{2}\pz^{\perp_{j+1}}.\]
Again by \eqref{eq-nzb-phij}, \eqref{dbdz-phi} and the generalized Ricci equation \eqref{eq-Ricci-nabla},  we have
\[\begin{split}
   \nabla_{j+1,\Bar{z}}(\nzz\phi_{j+1})^{\perp_{j+1}}&=\nzb\nzz\phi_{j+1}+(\cdots)\pz+(\cdots)\phi_1+\cdots+(\cdots)\phi_{j+1}\\
   &=\nzz\nzb\phi_{j+1}+(\cdots)\pz +(\cdots)\phi_1+\cdots+(\cdots)\phi_{j+1}\\
&=\frac{\bar\mu}{2}(\nzz\phi_{j+1})^{\perp_{j+1}}+(\cdots)\pz^{\perp_{j+1}}.\\
\end{split}
\]
 So we obtain $
   \nabla_{j+1,\Bar{z}}((\nzz\phi_{j+1})^{\perp_{j+1}}\wedge\pz^{\perp_{j+1}})=0.$

Now consider the definition of $\mathcal{X}_{j+1}\dd z^{2k+j+4}$  at the possible singular points of $\phi_{j+1}$. By Lemma \ref{lemma-linebundle}, for any $p\in S^2$,  there exists  a frame $\{X_{j1}, X_{j2}\}$ of $\partial f^{k}_{k+j+1}$ on $U_p$ such that  $\LL_{j+1}=\Spc\{X_{j1}\}$ on $U_p$ and we have
\[(\nzz\phi_j)^{\pj}=\sum_{1\leq l\leq 2} a_{jl}X_{jl},\ \pz^{\pj}=\sum_{1\leq l\leq 2} b_{jl}X_{jl},\]
from which we obtain
\[(\nzz\phi_j)^{\pj}\wedge\pz^{\pj}=(a_{j1}b_{j2}-a_{j2}b_{j1})X_{j1}\wedge X_{j2}.\]
Since both $(\nzz\phi_j)^{\pj}\wedge\pz^{\pj}$ and $X_{j1}\wedge X_{j2}\neq 0$ are smooth on $U_p$, we have that \[(a_{j1}b_{j2}-a_{j2}b_{j1})\] is a smooth function on $U_p$. We also have  $\phi_{j+1}=\frac{(a_{j1}b_{j2}-a_{j2}b_{j1})}{b_{j2}}X_{j1}$ on an open dense subset of $U_p$. So
 \[\nzz\phi_{j+1}=(\frac{(a_{j1}b_{j2}-a_{j2}b_{j1})}{b_{j2}})_zX_{j1}+\frac{(a_{j1}b_{j2}-a_{j2}b_{j1})}{b_{j2}}\nzz X_{j1}\]
from which we have
\[(\nzz\phi_{j+1})^{\perp_{j+1}}\wedge\pz^{\perp_{j+1}}=(a_{j1}b_{j2}-a_{j2}b_{j1})(\nzz X_{j1})^{\perp_{j+1}}\wedge X_{j2}.\]
Since both $(a_{j1}b_{j2}-a_{j2}b_{j1})$ and $(\nzz X_{j1})^{\perp_{j+1}}\wedge X_{j2}$ are smooth on $U_p$, $(\nzz\phi_{j+1})^{\perp_{j+1}}\wedge\pz^{\perp_{j+1}}$
is smooth on the whole $U_p$. Hence $(\nzz\phi_{j+1})^{\perp_{j+1}}\wedge\pz^{\perp_{j+1}}\dd z^{2k+j+4}$ extends to a global holomorphic differential on $S^2$. As a consequence,  the differential $
  \Theta_{j+1}\dd z^{4k+2j+8}$ is also a global holomorphic differential on $S^2$. This finishes the proof of Claim (2).
\end{proof}

We can unify Proposition \ref{prop-rank1} and Theorem \ref{thm-line} as follows.
\begin{corollary}
\label{cor-frenet} Set $h_0=(f^{k}_{k+k'})^{\perp}$ ($k'=0$ when $Rank_{\C}\partial f^k_k=1$),
 \[
  \mathcal{Z}^k=\left\{\begin{array}{lc}
\Pi_{k-1}, & \hbox{  when $Rank_{\C}\partial f^k_k=1$},\\
\Pi_{k-1}\oplus\mathbb L_1\oplus \cdots\oplus\mathbb L_{k'},& \hbox{  when $Rank_{\C}\partial f^k_k=2$}.\\
  \end{array}\right.
\]
 and
 \begin{equation}
     \label{eq-m-def}
     m:= Rank_{\C}\mathcal{Z}^k=
     \left\{\begin{array}{lc}
\dim \Pi_{k-1}, & \hbox{  when $Rank_{\C}\partial f^k_k=1$},\\
2k+k',& \hbox{  when $Rank_{\C}\partial f^k_k=2$}.\\
  \end{array}\right.
 \end{equation}

Then
\begin{enumerate}
    \item $y^k=[Y^k]$ has a Frenet-bundle structure  $\{f^{k}_{0}, \ZZ^k, h_0\}$.
    \item On $S^2\setminus Sing(y^k)$, set
    \begin{equation}\label{eq-wx}
    \wx:=\pz^{\perp_{k'}}\in\Gamma(h_0\otimes\C).
    \end{equation}
    Then for any  $\psi\in\Gamma((f^k_0\otimes\C)\oplus\ZZ^k)$,
\begin{equation}\label{eq-psiz}
     \psi_{z}=(\cdots)\widehat{\xi}\mod\{(f^k_0\otimes\C)\oplus\ZZ^k\}.
\end{equation}
\end{enumerate}
\end{corollary}

\begin{remark}\begin{enumerate}
    \item
In the first case, $m=\dim \Pi_{k-1}$ could take any value between $k+1$ and $2k$, depending on the structure of $Y^k$. Note that if we allow $Y^k$ to be S-Willmore, then we will have $\dim \Pi_{k-1}=k$.
 So, the Willmore surface $y^k=[Y^k]$ with $Rank_{\C}\partial f^k_k=1$ could be viewed as a generalization of S-Willmore surface.
\item
Moreover, we can also clarify non-totally isotropic Willmore $2$-spheres by the pair $(m,k)$ with $k$ being the strictly isotropic order,  and $m:=\dim\ZZ^k\geq k$.
Note that $m-k$ measures the lowest times of adjoint transforms to transform $Y^k$ to a strictly $m$-isotropic minimal surface in $\R^n$.
\item
For the totally isotropic case, we also have $m=\dim\Pi_{k-1}$ with $m\geq k$, where $k$ is defined to be the smallest value with $\Pi_{k-1}=\Pi_{k}$.
\end{enumerate}
\end{remark}
\begin{proof} The bundle decomposition structure follows directly from Theorem \ref{thm-line}. To verify the Frenet-bundle structure (a)-(c) of Definition \ref{def-Frenet}, we first prove it in the case $k\geq1$.

Recall that now $\partial f^{k}_{0}=Span_{\C}\{\kappa,\db\kappa\}$ is a sub-bundle of $\Pi_{k-1}$ and hence of $\partial f^k_0\subset \mathcal{Z}^k$. This proves $\partial f^k_0\subset \mathcal{Z}^k\oplus( h_{0}\otimes\C)$.

Next consider $\partial \ZZ^k$. By definition of $\ZZ^k$, $\ZZ^k\perp\kappa,\db\kappa$. Let $\psi\in\Gamma(\ZZ^k)$.
By the structure equation \eqref{mov-eq}, we have
\[\psi_{z}=\dz\psi + 2\<\psi,D_{\zb}\kappa\> Y
- 2\<\psi,\kappa\> Y_{\zb}=\dz\psi.\]
We decompose \begin{equation}
\label{eq-decom-psi}
\psi=\psi_1+\psi_2,~~ \hbox{  with $\psi_1\in\Gamma(\Pi_{k-1})$ and $\psi_2\in\Gamma(\mathbb L_1\oplus \cdots\oplus\mathbb L_{k'})$.}
\end{equation} For  $\psi_1$, by Lemma \ref{lemma-ind2} and \eqref{eq-dzdb}, as in the proof of Theorem \ref{thm-line}, we have
\[\dz\psi_1=(\cdots)\pz+(\cdots)\nzb\pz\mod\{\Pi_{k-1}\}.\]
For  $\psi_2$, by \eqref{eq-psidz} and the construction of $\LL_j$ in Theorem \ref{thm-line}, we have
\[\begin{split}
    \dz\psi_2&=\nzz\psi_2\mod\{\Pi_{k-1}\}\\
    &=(\cdots)\pz^{\perp_{k'}}\mod\{\Pi_{k-1}\oplus\mathbb L_1\oplus \cdots\oplus\mathbb L_{k'}\}.
\end{split}\]
By \eqref{eq-def-phi1}, $\phi_1=\nzb\pz+(\cdots)\pz\in\Gamma(\LL_1)$. So
\[\nzb\pz=(\cdots)\pz+\phi_1=(\cdots)\pz^{\perp_{k'}}\mod\{\mathbb L_1\oplus \cdots\oplus\mathbb L_{k'}\}.\]
This proves both  $\partial \ZZ^k\subset h_{0}\otimes{\C}$ and \eqref{eq-psiz} with $\widehat{\xi}:=\pz^{\perp_{k'}}\in\Gamma(h_0\otimes\C)$.


Finally consider $\bar\partial \ZZ^k$. Again by the structure equation \eqref{mov-eq}, we have
\[\psi_{\bar{z}}=\db\psi + 2\<\psi,D_{z}\bar{\kappa}\> Y
- 2\<\psi,\bar{\kappa}\> Y_{z}.\]
Recall that for any $\tilde{\psi}\in\Gamma(\Pi_{k-1})$,  $\db\tilde{\psi}\in \Gamma(\Pi_{k-1})$ by Lemma \ref{lemma-ind2}.
So by \eqref{eq-decom-psi}, we obtain $\db\psi=\db\psi_1+\db\psi_2$ with $\db\psi_1\in \Gamma(\Pi_{k-1})$. Since $\<\db\psi_2,\tilde{\psi}\>=-\<\psi_2,\db\tilde{\psi}\>=0$ for any $\tilde{\psi}\in\Gamma(\Pi_{k-1})$, we also have $\db\psi_2-\nzb\psi_2\in \Gamma(\Pi_{k-1})$. Taking into account \eqref{eq-nzb-phij}, we obtain $\nzb\psi_2\in \Gamma(\mathbb L_1\oplus \cdots\oplus\mathbb L_{k'})$. Hence $\psi_{\zb}\in\Gamma(\ZZ^k\oplus(f^k_0\otimes\C))$, that is, $\bar{\partial} \ZZ^k\subset f^k_0\otimes{\C}$.

For the case $k=0$, the proof is almost the same.  The only differences are now $\ZZ^0=\LL_1\oplus\cdots\oplus\LL_{k'}$, $\pz=\kappa$ and $\phi_1=\nzb\pz+(\cdots)\pz=\db\kappa+(\cdots)\kappa$, so that $\partial f^{k}_{0}=Span_{\C}\{\pz,\phi_1\}$. Hence $\partial f^k_0\subset \mathcal{Z}^k\oplus( h_{0}\otimes\C)$ holds again. The proofs of the rest are the same as above.
\end{proof}

For later applications, from the proof of Corollary \ref{cor-frenet}, we reformulate the structure equations of a Willmore surface $y^k$ which admits a Frenet-bundle structure $\{f^k_0,\ZZ^k,h_0\}$ as follows.
\begin{proposition}\label{prop-moving-frame}
    Let  $y^k=[Y^k]:M\rightarrow S^n$ be a strictly $k$-isotropic, non S-Willmore, Willmore surface with a Frenet-bundle structure $\{f^k_0,\ZZ^k,h_0\}$. On $M\setminus Sing(y^k)$,  we define
       \begin{equation}\label{eq-wx-1}
       \pz:=(D_z^{(k)}\kappa)^{\perp_0}~~ \hbox{ and }~~ \widehat{\xi}:=\pz^{\perp_{k'}} , \hbox{ on $M\setminus Sing(y^k)$}.
       \end{equation}
       Then we have
    \begin{equation}\label{eq-wx-2}
D_z\pz=\hat\tau_0\widehat{\xi}\mod \ZZ^k.
\end{equation}
Moreover, on $M\setminus Sing(Y^k)$, for any $\psi\in \Gamma(\ZZ^k)$ and $\xi\in\Gamma(h_0\otimes\C)$, we have
\begin{equation}\label{eq-moving-2}
\left\{\begin{split}
Y^k_{zz}&=(\cdots)Y^k+\kappa,\\
Y^k_{z\zb}&=0\mod \{Y^k,N^k\},\\
N^k_z&=  2\db\kappa\mod \{Y^k_z,Y^k_{\zb}\},\\
\psi_{z}&=(\cdots) \wx\mod \{\ZZ^k\}, \\
 \psi_{\zb}&=(\cdots)Y^k+(\cdots)Y^k_{z}\mod \{\ZZ^k\},\\
 \xi_{z}&=(\cdots)Y^k+(\cdots)Y^k_{\zb}\mod \{\bar{\ZZ}^k\oplus h_0\otimes\C\}.
\end{split}\right.
\end{equation}
\end{proposition}
\begin{proof}
Since $Rank_{\C}\partial h_0=1$ from the definition of the Frenet-bundle structure $\{f^k_0,\ZZ^k,h_0\}$, we obtain $Rank_{\C}\partial (h_0^{\perp})=1$. So we have
$\db\pz=(\cdots)\widehat{\xi}\mod \{\ZZ^k\}.$ This proves \eqref{eq-wx-2}.

For \eqref{eq-wx-2}, the first three equations are the same as the one in \eqref{mov-eq}. The fourth equation also holds due to the conclusion $Rank_{\C}\partial (h_0^{\perp})=1$. The fifth equation comes from the first three equations of \eqref{eq-moving-2} and the definition of the Frenet-bundle structure: $\bar\partial \ZZ^k\subset f^k_0\otimes\C$. The last equation is a corollary of  the first five equations of \eqref{eq-moving-2} and the definition \ref{def-Frenet} of the Frenet-bundle structure: $\partial h_0\subset (f^k_0\otimes\C)\oplus\bar{\ZZ}^k$.
\end{proof}



 Finally, we recall the following simple properties of a strictly $k$-isotropic surface $Y^k$.
\begin{lemma} \label{lemma-k}
For a strictly $k$-isotropic surface $Y^k$, if there exists some section $\xi$ such that
\begin{equation}\label{eq-zeta} \<\xi,Y^k\>=\<\xi,Y^k_z\>=\cdots=\<\xi,(Y^k)^{(k+1)}_z\>=0\neq \<\xi,(Y^k)^{(k+2)}_z\>,
\end{equation}
then
\begin{equation}\label{eq-zeta-2}
\<Y^k,\xi_z\>=\cdots=\<Y^k,\xi^{(k+1)}_z\>=0\neq \<Y^k,\xi^{(k+2)}_z\>.
\end{equation}
In particular, $\{\xi,\xi_z,\cdots,\xi^{(k+1)}_z\}$ has rank $k+2$ on an open dense subset.
\end{lemma}
The proof is straightforward.
\vspace{2mm}
\section{Differentiate $h_0$ to obtain harmonic sequence $\{h_j\}$ and minimal surface in $\R^n$}\label{sec-hj}

This section aims to study the geometry of the universal Gauss map $h_0$ globally defined on $S^2$, by constructing a second harmonic sequence $\{h_j, 1\leq j\leq \hat{m}+1\}$ defined on certain open dense subsets $M_j$ of $S^2$, which induces a specific minimal surface $Y^{\hat{m}}$ in $\R^n$.


 \subsection{Total isotropicity of $h_0$}
Since $h_0^{\perp}$ is a Lorentzian bundle, the rank as well as the induced metric of $Re(\partial h_{0})$ might degenerate when $\partial h_0$ contains lightlike vectors. Therefore, the map $h_1=\DD h_0=h_{0}\oplus Re(\partial h_{0})$ is in general defined only on some open dense subset of $S^2$.
To this end, let $\{\xi_{\alpha},1\leq \alpha\leq r\}$ be an orthonormal basis of $h_0$, with
\begin{equation}
    \label{eq-xi-z}
    \xi_{\alpha z}=\sum_{\beta}B_{\alpha\beta}\xi_{\beta}\mod h_0^{\perp}\otimes\C,~ B_{\alpha\beta}+B_{\beta\alpha}=0,
\end{equation}
for $1\leq \alpha,\beta\leq r$.
We then introduce $\zeta_{\alpha}$ as
 \begin{equation}\label{eq-Bab}\zeta_{\alpha}:=\xi_{\alpha
 z}-\sum_{\beta}B_{\alpha\beta}\xi_{\beta}.\end{equation}
We see by the last equation of \eqref{eq-moving-2} that
$$\zeta_{\alpha}=0 \mod \{Y^k,Y^k_{\zb},\bar{\ZZ}^k\}.
$$
As a consequence, $\<\zeta_{\alpha},\zeta_{\alpha}\>\equiv0$, $1\leq\alpha\leq r$. We will prove that $h_0$ is totally isotropic, that is,
\[\<(\zeta_{\alpha})^{(j)}_{z},(\zeta_{\beta})^{(l)}_{z}\>=0, \hbox{ for any $j,l\in \mathbb{Z}^+$ and $1\leq \alpha,\beta\leq r$.}\]
This property is well-defined, since it is independent of the choice of $z$ and $\{\xi_{\alpha},1\leq \alpha\leq r\}$.
\begin{theorem}[Theorem B, structure of $h_0$]\label{thm-h0}\
\begin{enumerate}
    \item The harmonic map $h_0:S^2=M_0\rightarrow Gr_{r}\RRR$ defined in Theorem \ref{thm-line} is totally isotropic, and it induces a harmonic sequence $\{h_j, 1\leq j\leq \hat{m}+1\}$, where
\[h_{j}:=\DD h_{j-1}=h_{j-1}\oplus Re(\partial h_{j-1}): M_j\rightarrow  Gr_{r+2j}\RRR,  1\leq j\leq \hm+1,\]
and $M_j\subset M_{j-1}$ is an open dense subset of $S^2$ with $\DD h_{j-1}$ being well defined on $M_{j}$.
\item Here $\hm\geq k$. Furthermore, for any  strictly $\hat{k}$-isotropic Willmore surface $Y^{\hat{k}}$ on an open dense subset of $S^2\setminus Sing(y^k)$, with a Frenet-bundle structure $\{f^{\hat{k}}_0,\ZZ^{\hat{k}},h_0\}$ with respect to the same $h_0$, we have
\[\hm\geq \hat{k}.\]
\end{enumerate}
\end{theorem}

It is obvious $\hat{m}\leq m=\dim\ZZ^k$ \eqref{eq-m-def}, due to the co-dimension restriction. We will prove $\hat m=m$ in Section \ref{sec-seq}.  for later convenience, we also define \[h_{\hm+2}:=h_{\hm+1}\oplus Re(\partial h_{\hm+1}).\]

Theorem \ref{thm-h0} is a parallel result of Calabi's Theorem for minimal $2$-spheres in $S^n$ \cite{Calabi}, since it shows that $h_0$ shares similar properties as harmonic $2$-spheres in a sphere, due to the condition $Rank_{\C}(\partial h_{0})=1$. To prove Theorem \ref{thm-h0}, we need the following lemma.

\begin{lemma}\label{lemma-zeta}
For each $j\geq1$, the differential form
\[\Xi_j=\sum_{\alpha}\<(\zeta_{\alpha})^{(j)}_{z},(\zeta_{\alpha})^{(j)}_{z}\>\dd z^{2j+2}\] is a global holomorphic differential and therefore vanishes on $S^2$.
\end{lemma}
\begin{proof}

We will prove the lemma by induction. We recall that the harmonicity of $h_0$ indicates
    \begin{equation}
        \label{eq-zeta-zb}
    \zeta_{\alpha\zb}-\sum_{\beta}\bar{B}_{\alpha\beta}\zeta_{\beta}=0\mod\{\xi_{\beta},1\leq\beta\leq r\},\ 1\leq \alpha\leq r.
    \end{equation}

When $j=1$, it is direct to check that $\Xi_1$ is independent of the choice of $z$ and $\xi_{\alpha}$.
By the harmonicity equation \eqref{eq-zeta-zb}, we obtain
\[(\Xi_1)_{\zb}=2\sum_{\alpha}\<\zeta_{\alpha\zb z},\zeta_{\alpha z}\>=2\sum_{\alpha}\<(\sum_{\beta}\bar{B}_{\alpha\beta}\zeta_{\beta}+(\cdots)\xi_{\beta})_z,\zeta_{\alpha z}\>.\]
 Since $\<\xi_{\alpha},\zeta_{\beta}\>=\<\zeta_{\alpha},\zeta_{\beta}\>=0$, we obtain
$0=\<\xi_{\alpha z},\zeta_{\beta}\>
=-\<\xi_{\alpha},\zeta_{\beta z}\>$. We also have $\<\zeta_{\alpha},\zeta_{\alpha z}\>=0$ and $\<\zeta_{\alpha },\zeta_{\beta z}\>=0$ since $\zeta_{\alpha}\parallel\zeta_{\beta}$. Since $\bar{B}_{\alpha\beta}+\bar{B}_{\beta\alpha}\equiv0$ by \eqref{eq-xi-z}, we have
\[(\Xi_1)_{\zb}=2\sum_{\alpha,\beta}\bar{B}_{\alpha\beta}\<\zeta_{\beta z},\zeta_{\alpha z}\>=0.\] So,  we have $\Xi_1\equiv0$ on $S^2$.

    Now we assume that, for $1\leq j\leq l$, $\Xi_{j}$ is holomorphic on $S^2$ and hence vanishes. Then we have $\<(\zeta_{\alpha})^{(j)}_{z},(\zeta_{\alpha})^{(j)}_{z}\>=0$ for all $\alpha$ and $1\leq j\leq l$. As a consequence,
       $\<(\zeta_{\alpha})^{(j)}_{z},(\zeta_{\alpha})^{(\tilde j)}_{z}\>=0$ for each $j+\tilde{j}\leq 2l+1$, from which we obtain that $\Xi_{l+1}$ is globally defined.

      Since $\<\zeta_{\alpha },\xi_{\beta}\>=\<\zeta_{\alpha z},\xi_{\beta}\>=0$ for all $1\leq\alpha,\beta\leq r$, we obtain inductively
      \[\<(\zeta_{\alpha })^{(j+1)}_{z},\xi_{\beta}\>=-\<(\zeta_{\alpha })^{(j)}_{z},\xi_{\beta z}\>=0 \hbox{ for $1\leq j\leq l.$}\]
Next we compute
    \[ \begin{split}
           ((\zeta_{\alpha})^{(j+1)}_{z})_{\zb}&=\left(\zeta_{\alpha \zb}\right)^{(j+1)}_{z}=\sum_{\beta}\left(\bar{B}_{\alpha\beta}\zeta_{\beta}+(\cdots)\xi_{\beta}\right)^{(j+1)}_{z}\\
           &=\sum_{\beta}(\cdots)\xi_{\beta}+(\cdots)\zeta_{\beta}+\cdots+(\cdots)(\zeta_{\beta})^{(j)}_{z}+\bar{B}_{\alpha\beta}(\zeta_{\beta})^{(j+1)}_{z}.\\
       \end{split}\]
As a consequence, similar to the case $j=1$, we obtain
\[\sum_{\alpha}\<(\zeta_{\alpha})^{(l+1)}_{z},(\zeta_{\alpha})^{(l+1)}_{z}\>_{\zb}=2\sum_{\alpha}\<((\zeta_{\alpha})^{(l+1)}_{z})_{\zb},(\zeta_{\alpha})^{(l+1)}_{z}\>=2\sum_{\alpha,\beta}\bar{B}_{\alpha\beta}\<(\zeta_{\beta})^{(l+1)}_{z},(\zeta_{\alpha})^{(l+1)}_{z}\>=0.\]
\end{proof}

\subsection{The local isotropic framing $Q_j$}
\begin{proof}[Proof of Theorem \ref{thm-h0}]Part
(1) follows from the proof of Lemma \ref{lemma-zeta}.
Consider (2). To show  $\hm\geq\hat{k}$, setting $\xi=\wx$ in Lemma \ref{lemma-k}, we obtain $Span_{\C}\{\wx_{z},\cdots,\wx_{z}^{(\hat{k}+2)}\}$
has rank $\hat{k}+2$ on an open dense subset $M_{\hat{k}+2}$. Set
\begin{equation}
    \label{eq-Q1}
    \bar{Q}_1=\wx_z\mod \{h_0\otimes\C\}, \hbox{ with } \bar{Q}_1\in\Gamma(\hp_0\otimes\C).
\end{equation}
We have $\<Q_1,Q_1\>=0$ and
$\wx_z^{(j)}=(\cdots)\bar{Q}_1+\cdots+(\cdots)(\bar{Q}_1)_z^{(j-1)}\mod \{h_0\otimes\C\},~ 2\leq j\leq \hat{k}+2.$
So on $M_{\hat{k}+2}$, for $2\leq j\leq \hat{k}+2$, $Span_{\C}\{\bar{Q}_1,\cdots, (\bar{Q}_1)_z^{(j-1)}\}$ has   rank $j$.
By Lemma \ref{lemma-zeta}, we obtain
\[\<(\bar{Q}_1)^{(j)}_z,(\bar{Q}_1)^{(j)}_z\>\equiv0, \hbox{ for all } j\geq1.\]
We also have
\[\<\xi_{\alpha},\bar{Q}_{1z}\>=-\<\xi_{\alpha z},\bar{Q}_1\>=0,  \hbox{ for all } 1\leq \alpha\leq r.\]
So $\bar{Q}_{1z}\in \Gamma((h_0^{\perp})\otimes\C)$. Inductively, we obtain
\begin{equation}
    \label{eq-Q1jz-p}(\bar{Q}_{1})^{(j)}_z\in \Gamma((h_0^{\perp})\otimes\C)
~    \hbox{  for all $j\geq 1$.}
\end{equation} Now for $1\leq j\leq \hat{k}+1$, we define inductively that
\begin{equation}\label{eq-Qjz-1}
 Q_{j+1}:={Q}_{j \zb}\mod \{Q_1,\cdots,Q_j\}, \hbox{ with  $Q_{j+1}\perp\{Q_l,\Bar{Q}_l, 1\leq l\leq j\}$.}
 \end{equation}
 On the other hand, consider
 $Q_{1z}$. Since $h_0$ is harmonic, we obtain
\begin{equation}\label{eq-Q1z}
Q_{1z}=(\cdots) Q_{1}+\sum_{1\leq\alpha\leq r}(\cdots)\xi_{\alpha},\\
 \end{equation}
Hence
\[Q_{2z}=(\cdots)(Q_1)_{\zb z}-(\cdots)Q_1-(\cdots)Q_{1z}=(\cdots)Q_1+(\cdots)Q_2 \mod \{h_0\otimes\C\}.\]
Since $\<\xi_{\alpha},{Q}_{2z}\>=-\<\xi_{\alpha z},{Q}_2\>=0,  \hbox{ for all } 1\leq \alpha\leq r,$ we obtain
\[Q_{2z}=(\cdots)Q_1+(\cdots)Q_2.\]
We claim $\<Q_{2},\bar{Q}_{2}\>\not\equiv0$ on $M_{\hat{k}+2}$.
Otherwise, assume $\<Q_{2},\bar{Q}_{2}\>=0$ on $M_{\hat{k}+2}$.  Since $\<Q_{2},{Q}_{2}\>=0$,
$\<Re(Q_2),Re(Q_2)\>=\<Re(Q_2),Im(Q_2)\>=\<Im(Q_2),Im(Q_2)\>=0$. So $Re(Q_2)$ and $Im(Q_2)$ must be parallel light-like vector fields. Therefore, there exists some $Y^*$ that is real and light-like such that $Q_2=(\cdots)Y^*$.
 So $Q_{2z}=(\cdots)Y^*_{z}+(\cdots)Y^*$. Since $Q_{2z}=(\cdots)Q_2+(\cdots)Q_{1}$, we obtain $Y^*_{z}=(\cdots)Q_1+(\cdots)Q_{2}$. Hence $\bar{Q}_{2z}=(\cdots)Y^*_{z}+(\cdots)Y^*=(\cdots)\bar{Q}_1+(\cdots)\bar{Q}_{2}$, contradicting the conclusion $Span_{\C}\{\bar{Q}_1,\cdots, (\bar{Q}_1)_z^{(j-1)}\}$ has rank $j$ for $2\leq j\leq \hat{k}+2$.

In the same way, we obtain inductively that there exists $\hm\geq \hat{k}$ such that
\begin{equation}\label{eq-Qjz}
\left\{\begin{split}
Q_{j z}&=(\cdots) Q_{j}+(\cdots)Q_{j-1},\hbox{ with $\<Q_{j },\bar{Q}_{j}\>\not\equiv0\not\equiv\<Q_{j z},\bar{Q}_{j-1}\>$},~~ \hbox{ for } 2\leq j\leq \hm+1,\\
Q_{ j \zb}&= Q_{j+1}+(\cdots)Q_{j},~~ \hbox{ for } 1\leq j\leq \hm+1.\\
\end{split}\right.
\end{equation}
From this we see $h_j$ is well-defined, and the harmonicity of $h_j$ comes from the fact that $h_0$ is harmonic and Theorem \ref{thm-db-2}, taking into account \eqref{eq-Qjz}.
\end{proof}

\vspace{2mm}

\section{The sequence of adjoint Willmore surfaces}\label{sec-seq}

The local sequence  $\{h_j\}_{0\leq j\leq \hm+2}$ defines a flag structure which determines $y^j=[Y^j]$ inductively as follows:
\begin{theorem}[Theorem C, initial theorem]\label{thm-initial}  The strictly $k$-isotropic Willmore $2$-sphere $y^k=[Y^k]$ satisfies the following statements for $l=k$.
\begin{enumerate}
    \item[(a)] There exists a Frenet-bundle structure $\{f^{l}_{0},\ZZ^l, h_0\}$ for $Y^l$ with respect to $h_0$, with $f^{l}_{0}$ its conformal Gauss map and $Rank_{\C}\partial h_0=1$;
    \item[(b)] $Y^l$ is strictly $l$-isotropic, in particular, $Y^l\in \Gamma(h_{l+1}^{\perp})\setminus \Gamma(h_{l+2}^{\perp})$.
\end{enumerate}
\end{theorem}

\begin{theorem}[Theorem  D, inductive theorem]\label{thm-inductive} Assume a Willmore surface $[Y^j]:M_j\rightarrow S^n$ satisfies the above statements $(a)$ and $(b)$ with $l=j$  in Theorem \ref{thm-initial} for $k\leq j\leq m-1$.

Then there exists a unique $Y^{j+1}$ on $M_{j+1}$, an open dense subset of $M_j$, such that
\begin{enumerate}
    \item There exists $h_{j+3}$ such that the above statements $(a)$ and $(b)$ with $l=j+1$ hold for $Y^{j+1}$;
    \item $Y^{j+1}$ is adjoint to $Y^j$ when $j\leq m-1$.
Moreover,  $[Y^{m}]$ is M\"obius equivalent to a minimal surface in $\R^{n}$.
\end{enumerate}
\end{theorem}

The crucial point in the proof is that the Willmore sequence $\{Y^j\}$ is uniquely determined inductively by the intersection formula $Span_{\R}\{Y^j, Y^{j+1}\}=f^{j}_{0}\cap h^{\perp}_{j+1},\ k\leq j\leq m$. And the main tools we have are two frames of $h_0^{\perp}\otimes\C=(f^k_0\oplus Re(\ZZ^k))\otimes\C$ in Sections \ref{s-fkj} and \ref{sec-hj}, respectively:
\begin{enumerate}
    \item $\{Y^k,Y^k_z,N^k,F^k_j\}$, induced by the Frenet-bundle structure of $Y^k$;
    \item $\{Q_j,1\leq j\leq \hat{m}+2\},$ induced by the totally isotropic structure of $h_0$.
\end{enumerate}
The key ideas of proofs of  Theorem \ref{thm-initial}  and  Theorem \ref{thm-inductive} as follows:
\begin{enumerate}
    \item To obtain the unique $Y^{j+1}$  by comparing the frame of $Y^j$ with the frame of $(h_0)^{\perp}$;
    \item  To derive the Frenet-bundle structure $\{f^{j+1}_0,\ZZ^{j+1},h_0\}$ of $Y^{j+1}$, from that of $Y^{j}$ with respect to the frame of $h_0$;
\item The Willmore sequence $\{Y^j\}$ will continue until $j=m$, due to the fullness of $Y^k$ in $S^n$.
\end{enumerate}
To be concrete,  Theorem \ref{thm-initial}  and  Theorem \ref{thm-inductive} follow from the following two lemmas.

\begin{lemma}
 \label{lemma-yk-initial}
Let $Y^k$ be a strictly $k$-isotropic Willmore $2$-sphere which is not S-Willmore. Then $Y^k$ has a Frenet-bundle structure $\{f^{k}_{0}, \ZZ^k,h_0\}$, with the following statements holding:
\begin{enumerate}

\item  On an open dense subset $M_{k+1}$ of $S^2$,  $Y^k\in\Gamma(h_{k+1}^{\perp})\setminus \Gamma(h_{k+2}^{\perp})$ and there exists a unique light-like vector field  $Y^{k+1}\in \Gamma(f^{k}_{0})$, with
\begin{equation}
    \label{eq-YkYk+1}
Span_{\R}\{Y^k, Y^{k+1}\}=f^{k}_{0}\cap h^{\perp}_{k+1}, ~\<Y^{k+1},Q_{k+1}\>=0, \hbox{ and }\<Y^k,Y^{k+1}\>=-1.
\end{equation}
Moreover, setting $F^{k}_{0}=Y^k_{z}+\<Y^k_z,Y^{k+1}\>Y^k$, we have
\begin{equation}\label{eq-Qk}
Q_{k+1}=(\cdots)F^{k}_{0}\mod \{\ZZ^{k}\},   \hbox{ with }\ \<Q_{k+1},\overline{F^{k}_{0}}\>\not\equiv0,
\end{equation}
\begin{equation}\label{eq-Qk+2}
Q_{k+2}=(\cdots)F^{k}_{0}+(\cdots) Y^{k+1} \mod \{\ZZ^{k}\}, \hbox{ with }\<Q_{k+2},Y^{k}\>\not\equiv0;   \end{equation}
    \item  $\hat{m}\geq k+1$ and $Y^{k+1}$ is an adjoint surface of $Y^k$ with $Y^{k+1}\in\Gamma(h_{k+2}^{\perp})\setminus \Gamma(h_{k+3}^{\perp})$.  And $Y^{k+1}$ is strictly $(k+1)$-isotropic.
\end{enumerate}
\end{lemma}

\begin{lemma}\label{lemma-key-inductive}
For $k+1\leq j\leq m$,
if the following statements hold for $l=j-1$ on $M_{l}$, then they hold for $l=j$ (here $M_j$ is an open dense subset of $S^2$ with $M_{j}\subset M_{j-1}$):
\begin{enumerate}

\item  The surface  and
$Y^{l}$ has a Frenet-bundle structure $\{f^l_0, \ZZ^l,h_0\}$ and $Y^l\in\Gamma(h_{l+1}^{\perp})\setminus \Gamma(h_{l+2}^{\perp})$;
\item There exists a unique light-like vector field $Y^{l+1}\in \Gamma(f^{l}_{0})$  with \[\hbox{$Span_{\R}\{Y^l, Y^{l+1}\}=f^{l}_{0}\cap h^{\perp}_{l+1}$, $\<Y^l,Y^{l+1}\>=-1$ and $F^{l}_{0}:=Y^l_{z}+\<Y^l_z,Y^{l+1}\>Y^l$.}\] Moreover,
\begin{equation}\label{eq-Ql1-Fb}
Q_{l+1}=(\cdots)F^{l}_{0}\mod \{\ZZ^{l}\}, \hbox{ with }\<Q_{l+1},\overline{F^{l}_{0}}\>\not\equiv0,
\end{equation}
\begin{equation}\label{eq-Ql+2-F}
Q_{l+2}=(\cdots)F^{l}_{0}+(\cdots) Y^{l+1} \mod \{\ZZ^{l}\}, \hbox{ with }\<Q_{l+2},Y^{l}\>\not\equiv0;   \end{equation}
    \item The map $Y^{l+1}$ is an adjoint transform of $Y^l$. When $l\leq m-1$, the bundle $\hp_{l+3}$ is well-defined on an open dense subset of $M_{l+1}$. And $Y^{l+1}\in\Gamma(h_{l+2}^{\perp})\setminus \Gamma(h_{l+3}^{\perp})$ is strictly $(l+1)$-isotropic.
\end{enumerate}
\end{lemma}

\begin{proof}[Proof of Theorem \ref{thm-initial}]
 Theorem \ref{thm-initial} follows directly form Corollary \ref{cor-frenet} and Lemma \ref{lemma-yk-initial}.
\end{proof}

\subsection{Two technical lemmas}
To prove Lemmas \ref{lemma-yk-initial} and \ref{lemma-key-inductive}, we need two technical lemmas.
\begin{lemma}\label{lemma-hj-sheaf}
Assume that for $0\leq j\leq \hm+2$,
$h_{j}$ is well defined on an open dense subset of $S^2$.
\begin{enumerate}
    \item
Let $X\in\Gamma(h_l^{\perp}\otimes\C)$, $1\leq l\leq\hm+1$.  Then
\begin{enumerate}
    \item $X_z\in\Gamma(h_l^{\perp}\otimes\C)\setminus\Gamma(\hp_{l+1}\otimes\C)$ if and only if   $\<X,\bar{Q}_{l+1}\>\not\equiv0$.
    \item If $X\in\Gamma(h_{l}^{\perp})\setminus\Gamma(h_{l+1}^{\perp})$, then $X_{z\zb}\in \Gamma(h_{l-1}^{\perp})$.
\end{enumerate}

\item Let $X\in\Gamma(h_0^{\perp}\otimes\C)$.  Then for $1\leq l\leq\hm+1$, $X\in\Gamma(h_{l}^{\perp}\otimes\C)\setminus\Gamma(h_{l+1}^{\perp}\otimes\C)$ if and only if $X_z^{(j)}\in\Gamma(h_{0}^{\perp}\otimes\C)$ for $1\leq j\leq l$ and $X_z^{(l+1)}\not\in\Gamma(h_{0}^{\perp}\otimes\C)$.
\end{enumerate}
\end{lemma}
The proof is a straightforward application of \eqref{eq-Qjz}.
\begin{lemma} \label{lemma-hj-sheaf-2}
Let $X\in\Gamma(h_{l}^{\perp})\setminus\Gamma(h_{l+1}^{\perp})$ with $\<X,X\>\equiv0$.
Assume \[X_z=(\cdots)X+(\cdots)F_{0}+\cdots+(\cdots)F_{m}\]
with
$\{F_0, F_1,\cdots, F_{m}\}$ being sections $h_0^{\perp}\otimes\C$ satisfying $\<F_{i},F_{j}\>=\<X,F_j\>\equiv0$, and
\[F_{jz}=(\cdots)F_{0}+\cdots+(\cdots)F_{m}+(\cdots)\wx,\hbox{ for all }0\leq j\leq m.\]Then $X$ is strictly $(l-1)$-isotropic.
\end{lemma}
\begin{proof}
    Since $X\in\Gamma(h_{l}^{\perp})\setminus\Gamma(h_{l+1}^{\perp})$, by (2) of Lemma \ref{lemma-hj-sheaf} and assumptions of $\{F_0,\cdots, F_m\}$, we obtain
$X_z^{(j)}=(\cdots)X+(\cdots)F_{0}+\cdots+(\cdots)F_{m}$ when $1\leq j\leq l,$ and $X_z^{(l+1)}=(\cdots)X+(\cdots)F_{0}+\cdots+(\cdots)F_{m}+(\cdots)\wx.$ So $\<X_z^{(j)}, X_z^{(j)}\>=0$ for $1\leq j\leq l$. Since $X_z^{(l+1)}\not\equiv0\mod{h_0^{\perp}}\otimes\C$, we obtain $\<X_z^{(l+1)}, X_z^{(l+1)}\>=(\cdots)\<\wx,\wx\>\not\equiv0$.
\end{proof}


\subsection{Proof of The Initial Lemma}

\begin{proof}[Proof of Lemma \ref{lemma-yk-initial}]

(1)   Since $Y^k$ is strictly $k$-isotropic, we have $\hm\geq k$ by Theorem \ref{thm-h0}. Moreover, by construction of $h_0$, we have that \[Y^k,Y^k_{z},\cdots,(Y^k)^{(k+1)}_{z}\in \Gamma(h_0^{\perp}\otimes\C), \hbox{ and }(Y^k)^{(k+2)}_{z}\not\equiv 0\mod \{h_0^{\perp}\otimes\C\}.\]
Then by (2) of Lemma \ref{lemma-hj-sheaf}, we obtain  $Y^k\in\Gamma(h_{k+1}^{\perp})\setminus \Gamma(h_{k+2}^{\perp})$. So by (1) of Lemma \ref{lemma-hj-sheaf}, we have $Y^k_z\in\Gamma(h_{k}^{\perp}\otimes\C)\setminus \Gamma(h_{k+1}^{\perp}\otimes\C)$ and $N^k=Y^k_{z\zb}-(\cdots)Y^k_{z}-(\cdots)Y^k_{\zb}-(\cdots)Y^k\in \Gamma(h_{k}^{\perp})$.

Next, we first consider \eqref{eq-Qk}.  By \eqref{eq-Q1} and \eqref{eq-moving-2} for $Y^k$, we have
\[Q_1= \left\{
\begin{split}
&  Y^0_z+(\cdots)Y^0\mod \{\ZZ^k\},\\
& 0\mod \{\ZZ^k\},
\end{split}~~\begin{split}
 & \hbox{ when $k=0$},\\
 &\hbox{ when $k\geq1$}.
\end{split}
\right.\]
So, when $k=0$, there exists a unique $F^0_0=Y^0_z+(\cdots)Y^0$ such that \eqref{eq-Qk} holds.

When $k\geq1$, by \eqref{eq-Q1},  \eqref{eq-Qjz-1} and \eqref{eq-moving-2}, we first have $Q_1= 0\mod \{\ZZ^k\}$, and then
\[Q_2=(\cdots) Y^k+(\cdots) Y^k_z\mod \{\ZZ^k\}.\]

If $k=1$, since $Y^1_z\in\Gamma(h_{1}^{\perp}\otimes\C)\setminus \Gamma(h_{2}^{\perp}\otimes\C)$,  on an open dense subset $M_{2}$, it holds $\<Q_2,Y^1_{\zb}\>\neq0$ and hence  there exists a unique $F^1_0=Y^1_z+(\cdots)Y^1$ such that \eqref{eq-Qk} holds.

If $k\geq2$, since $Y^k_z\in\Gamma(h_{k}^{\perp}\otimes\C)\setminus \Gamma(h_{k+1}^{\perp}\otimes\C)$ and $N^k\in \Gamma(h_{k}^{\perp})$, $\<Q_2,N^k\>=\<Q_2,Y^k_{\zb}\>=0$, that is, $Q_2=0\mod \{\ZZ^k\}.$ Inductively, we obtain $Q_k=0\mod \{\ZZ^k\}$. So,
\[Q_{k+1}=(\cdots)F^{k}_{1}+\cdots+(\cdots)F^{k}_{m}+(\cdots) Y^k+ \sigma_k Y^k_z, \hbox{ with }  \sigma_k\not\equiv0,\]
since $Y^k_z\in\Gamma(h_{k}^{\perp}\otimes\C)\setminus \Gamma(h_{k+1}^{\perp}\otimes\C)$.
Hence, there exists a unique $F^{k}_{0}=Y^k_z+(\cdots)Y^k$ on an open dense subset $M_{k+1}$ such that \eqref{eq-Qk} holds.\vspace{2mm}

Moreover, there exists a unique light-like section $Y^{k+1}\in\Gamma(f^k_0)$ with
\[\<F^k_0,Y^{k+1}\>=0,~ \<Y^k,Y^{k+1}\>=-1 \hbox{ and }f^{k}_{0}=Span_{\R} \{Y^k,Y^{k+1},Re(F^{k}_{0}),Im(F^{k}_{0})\}.\]
By \eqref{eq-Qk}, we have $Y^{k+1}\perp Q_{k+1}$, that is, $Y^{k+1}\in \Gamma(h_{k+1}^{\perp})$. Since $\< F^{k}_{0},\bar{Q}_{k+1}\>\neq0$, we obtain
$Span\{Y^{k}, Y^{k+1}\}=f^{k}_{0}\cap \hp_{k+1}$, that is, \eqref{eq-YkYk+1} holds. \vspace{2mm}

Consider $Q_{k+2}$.  By \eqref{eq-Qjz-1} and \eqref{eq-moving-2}, since  $\<(Q_{k+1})_{\zb}, F^k_0\> =-\<Q_{k+1}, (F^k_0)_{\zb}\> = 0$, we get
\[Q_{k+2}=Q_{k+1,\zb}-(\cdots)Q_{k+1}=(\cdots)F^{k}_{1}+\cdots+(\cdots)F^{k}_{m}+(\cdots) F^{k}_{0}+(\cdots)Y^{k+1}+(\cdots)Y^k.\]
Since $Y^k\in\Gamma(h_{k+1}^{\perp})\setminus \Gamma(h_{k+2}^{\perp})$, we have $\<Q_{k+2},Y^k\>\nt0$. Since $Q_{k+2}$ is isotropic, we get
\[Q_{k+2}=(\cdots)F^{k}_{1}+\cdots+(\cdots)F^{k}_{m}+(\cdots)F^{k}_{0}+(\cdots)Y^{k+1}, \hbox{ with } \<Q_{k+2},Y^k\>\not\equiv0.\]

(2) By \eqref{eq-Qk+2}, we get $Y^{k+1}\in \Gamma(\hp_{k+2})$ and hence $Y^{k+1}_z\in \Gamma(\hp_{k+1}\otimes\C)$ by (1) of Lemma \ref{lemma-hj-sheaf}. In particular, $Y^{k+1}_z\perp h_0$. So $Y^{k+1}_z=0 \mod\{Y^k,Y^{k+1},F^{k}_0,\bar{F}^k_0,F^k_l,\bar{F}^k_l,1\leq l\leq m\}$.

Since  $\<Y^{k+1},Y^{k+1}\>=\<Y^{k+1},F^{k}_{l}\>=0$ for $1\leq l\leq m$, by \eqref{eq-moving-2} we have
\[\<Y^{k+1}_z,Y^{k+1}\>=0,~ \<Y^{k+1}_z,F^{k}_{l}\>=-\<Y^{k+1},F^{k}_{l z}\>=0, ~\text{for}~ 1\leq l\leq m.\] So $Y^{k+1}_z$ has no components of $\{Y^{k},\bar{F}^{k}_l,1\leq l\leq k\}$, that is,
\begin{equation}
    \label{eq-Y1z}
Y^{k+1}_{z}=(\cdots) Y^{k+1}+(\cdots)F^{k}_{0}+\theta\bar{F}^{k}_{0}+(\cdots) F^{k}_{1}+(\cdots)+(\cdots) F^{k}_{m}.
\end{equation}
Since $\<\bar{F}^{k}_{0},Q_{k+1}\>\neq0$ and $Y^{k+1}_z\in \Gamma(\hp_{k+1}\otimes\C)$, we have $0=\<Y^{k+1}_z,Q_{k+1}\>=\theta\<\bar{F}^{k}_{0},Q_{k+1}\>$ and hence $\theta=0$.
So, by Definition \ref{def-adj}, $Y^{k+1}$ is an adjoint tranform of $Y^k$, since it satisfies both the co-touch condition $\theta=0$, and the conformal condition $\<Y^{k+1}_z,Y^{k+1}_z\>=0$. By Proposition \ref{prop-full}, $[Y^{k+1}]$ is either full in $S^n$ or a constant map.   Since $Y^k$ is not S-Willmore by assumption, by Proposition \ref{prop-minimal}, $[Y^{k+1}]$ is full in $S^n$ on an open dense subset. In particular, $Y^{k+1}_{\zb}\neq(\cdots)Y^{k+1}$.

Since $\<Y^k, Q_{k+2}\>\neq 0$ by $Y^k\in\Gamma(h_{k+1}^{\perp})\setminus \Gamma(h_{k+2}^{\perp})$, we obtain
 \[(Q_{k+2})_{\zb}=(\cdots) F^k_0+\cdots+(\cdots) F^k_m+(\cdots)Y^{k+1}+(\cdots)Y^{k}+(\cdots)Y^{k+1}_{\zb},\]
with $Y^{k+1}_{\zb}=(\cdots)Y^{k+1}+(\cdots) \bar{F}^k_0+\cdots+(\cdots) \bar{F}^k_m$.
 Since $Y^{k+1}_{\zb}\neq(\cdots)Y^{k+1}$,  $(Q_{k+2})_{\zb}\nt 0\mod\{Q_{k+2}\}$, that is, $Q_{k+3}\not\equiv0$ exists.

 We have already shown $Y^{k+1}\in \Gamma(\hp_{k+2})$. Since $Y^{k+1}$ is adjoint to $Y^k$, $Y^k$ is also adjoint to $Y^{k+1}$ with $Y^{k+1}=(\cdots)Y^{k+1}_{z\zb}+(\cdots)Y^{k+1}+(\cdots)Y^{k+1}_z+(\cdots)Y^{k+1}_{\zb}$. Therefore if $Y^{k+1}\in \Gamma(\hp_{k+3})$, we will have $Y^k\in\Gamma(\hp_{k+2})$ by (1) of Lemma \ref{lemma-hj-sheaf}, a contradiction to $Y^k\in\Gamma(\hp_{k+1})\setminus  \Gamma(\hp_{k+2})$. Hence $Y^{k+1}\in \Gamma(\hp_{k+2})\setminus  \Gamma(\hp_{k+3})$ holds.

By \eqref{eq-moving-2} for $Y^k$, \eqref{eq-Y1z} and Lemma \ref{lemma-hj-sheaf-2},  $Y^{k+1}$ is strictly $(k+1)$-isotropic.
\end{proof}

\subsection{The Inductive Lemma}

\begin{proof}[Proof of Lemma \ref{lemma-key-inductive}]
The key idea is to derive a Frenet-bundle structure of $Y^j$ by the Frenet-bundle structure of $Y^{j-1}$ so that we can compare the frame of $Y^j$ with the frame of $h_0^{\perp}$ .

(1)
 By assumption, $Y^{j-1}$ has a Frenet-bundle structure $\{f^{j-1}_0, \ZZ^{j-1},h_0\}$. Let
$\{F^{j-1}_{i},~1\leq i\leq m\}$ be a unitary frame of $\ZZ^{j-1}$. Let $\kappa_{j-1}$ and $D^{Y^{j-1}}$ denote the Hopf differential and the normal connection of $Y^{j-1}$ respectively.
Since $Y^{j}$ is adjoint to $Y^{j-1}$ by the inductive assumption (3), we have
\begin{equation}\label{eq_Yjz1}
\begin{split}
Y^{j}_z&=(\cdots)Y^j+(\cdots)F^{j-1}_{0}+a_{j-1}\kappa_{j-1}+\hat{a}_{j-1}(D^{Y^{j-1}})_{\zb}\kappa_{j-1}, \\
       &=(\cdots)Y^j+(\cdots)F^{j-1}_{0}+(\cdots)F^{j-1}_{1}+\cdots+(\cdots)F^{j-1}_{m}.\\
\end{split}
\end{equation}
Here the first equation comes from Definition \ref{def-adj} for adjoint surface and the second one comes from Theorem \ref{thm-m-iso} when $j-1>0$.
In case that $j-1=0$, $a_{j-1}\kappa_{j-1}+\hat{a}_{j-1}(D^{Y^{j-1}})_{\zb}\kappa_{j-1}$ is isotropic by \eqref{eq-theta-2} of Definition \ref{def-adj} for adjoint transform. So it is contained in $\LL_1$ for $Y^{j-1}$ as in Theorem \ref{thm-line}. This proves \eqref{eq_Yjz1} since $\LL_1\subset\ZZ^{j-1}$.

So on an open dense subset $M_j\subset M_{j-1}$, we can rewrite \eqref{eq_Yjz1} as
\[Y^{j}_z=(\cdots)Y^j+(\cdots)\tilde F^{j}_{0}\]
with $\tilde F^{j}_{0}=(\cdots)F^{j-1}_{0}+\cdots+(\cdots)F^{j-1}_{m}\neq0$.
Since  $Y^{j-1}$ is also adjoint to $Y^j$, we obtain \[f^j_0=Span_{\R}\{Y^j,Y^{j-1},Re(\tilde F^{j}_{0}), Im(\tilde F^{j}_{0})\}.\]  
Let $\ZZ^j$ be the unique sub-bunlde of the totally isotropic bundle $\ZZ^{j-1}\oplus Span_{\C}\{F^{j-1}_{0}\}$ such that
\begin{equation}
    \label{eq-gauge}
\tilde F^{j}_{0}\perp \overline{\ZZ^j}\hbox{ and }\ZZ^j\oplus Span_{\C}\{\tilde F^{j}_{0}\}=\ZZ^{j-1}\oplus Span_{\C}\{F^{j-1}_{0}\}.
\end{equation}
This gives the first gauge of frames respecting to the structure of $Y^{j}$, which adapts the bundle decomposition of $h_0^{\perp}$ from $h_0^{\perp}=f^{j-1}_0\oplus Re(\ZZ^{j-1})$  to  $h_0^{\perp}=f^j_0\oplus Re(\ZZ^{j})$.
\vspace{2mm}

Now, let us prove $\partial f^{j}_{0}\subset  \ZZ^j\oplus(h_0\otimes\C)$ (recall (a) of Definition \ref{def-Frenet}). Since $F^{j-1}_0=(\cdots)Y^{j-1}_z+(\cdots)Y^{j-1}=(\cdots)Y^{j-1}+(\cdots)F^{j-1}_0$, we get \[Y^{j-1}_z\in \Gamma( (f^{j}_{0}\otimes\C)\oplus\ZZ^j)\] by \eqref{eq_Yjz1} and \eqref{eq-gauge}.
By the structure equation \eqref{mov-eq}, we also have
$\tilde {F}^{j}_{0\zb}\in\Gamma(f^{j}_{0}\otimes\C)$. Moreover, by \eqref{eq_Yjz1} and by \eqref{eq-moving-2} for $Y^{j-1}$ in the inductive assumption (1) we obtain
$\tilde F^{j}_{0z}=(\cdots)\wx \mod \{f^{j}_{0}\otimes\C, \ZZ^j\}$. This proves $\partial f^{j}_{0}\subset \ZZ^j\oplus(h_0\otimes\C)$.

Consider any section $\psi\in \Gamma(\ZZ^j)$. Let $\{F^{j}_{i},~1\leq i\leq m\}$ be a unitary frame of $\ZZ^k$.

On the one hand, we have by \eqref{eq-gauge} and \eqref{eq-moving-2} for $Y^{j-1}$ in the inductive assumption (1)
\[\begin{split}
    \psi_z&=((\cdots)F^{j-1}_{0}+(\cdots)F^{j-1}_{1}+\cdots+(\cdots)F^{j-1}_{m})_z\\
&=(\cdots)F^{j-1}_{0}+(\cdots)F^{j-1}_{1}+\cdots+(\cdots)F^{j-1}_{m}+(\cdots)\wx\\
&=(\cdots)\tilde{F}^{j}_{0}+(\cdots)F^{j}_{1}+\cdots+(\cdots)F^{j}_{m}+(\cdots)\wx.\\
\end{split}\]
Since $\<\psi_{z},\overline{\tilde{F}^{j}_{0}}\>=-\<\psi,\overline{\tilde{F}^{j}_{0}}_{z}\>=0$, we have
$\psi_z=(\cdots)F^{j}_{1}+\cdots+(\cdots)F^{j}_{m}+(\cdots)\wx$, that is, $\partial \ZZ^j\subset h_0\otimes\C$ holds.

On the other hand, since $(F^{j-1}_{0})_{\zb}=(\cdots)F^{j-1}_{0}+(\cdots)Y^{j-1}+(\cdots)Y^{j}$ by the structure equation \eqref{mov-eq}, together with \eqref{eq-moving-2} for $Y^{j-1}$ and \eqref{eq-gauge}, we have
\[\begin{split}
    \psi_{\zb}&=((\cdots)F^{j-1}_{0}+(\cdots)F^{j-1}_{1}+\cdots+(\cdots)F^{j-1}_{m})_{\zb}\\
&=(\cdots)Y^{j-1}+(\cdots)Y^j+(\cdots)F^{j-1}_{0}+(\cdots)F^{j-1}_{1}+\cdots+(\cdots)F^{j-1}_{m}\\
&=(\cdots)Y^{j-1}+(\cdots)Y^j+(\cdots)\tilde{F}^{j}_{0}+(\cdots)F^{j}_{1}+\cdots+(\cdots)F^{j}_{m}.\\
\end{split}\]
That is, $\bar\partial \ZZ^j\subset f^j_0\otimes\C$ holds.  Hence $Y^j$ has a Frenet-bundle structure $\{f^j_0, \ZZ^j,h_0\}$.\vspace{2mm}

(2)  The next step is to adapt $\tilde{F}^j_0$ respecting to $Q_{j+1}$. By the inductive assumption \eqref{eq-Ql+2-F} of (2) for $Y^{j-1}$, together with \eqref{eq-gauge}, we get
\begin{equation}\label{eq-Emj-F}
\begin{split}
    Q_{j+1}&=(\cdots)Y^{j}+(\cdots)F^{j-1}_{0}+(\cdots)F^{j-1}_{1}+\cdots+(\cdots)F^{j-1}_{m}\\
    &=(\cdots)Y^{j}+(\cdots)\tilde F^{j}_{0}+(\cdots)F^{j}_{1}+\cdots+(\cdots)F^{j}_{m}.\\
\end{split}
\end{equation}
Since $Y^j\in\Gamma(h_{j+1}^{\perp})\setminus \Gamma(h_{j+2}^{\perp})$ as assumed, we have
$\<Y^j_z,\bar{Q}_{j+1}\>\nt0$ by (1) of Lemma \ref{lemma-hj-sheaf}. Hence $\<\tilde F^{j}_{0},\bar{Q}_{j+1}\>\neq0$ on an open denses subset $\tilde{M}_j$ of $M_{j}$
and  there exists some unique $\mu_j$ on $\tilde{M}_j$  such that
\eqref{eq-Emj-F} can be rewritten as
\[Q_{j+1}=(\cdots)(\tilde F^{j}_{0}+\frac{\mu_j}{2}Y^{j})+(\cdots)F^{j}_{1}+\cdots+(\cdots)F^{j}_{m}=(\cdots)F^{j}_{0}+(\cdots)F^{j}_{1}+\cdots+(\cdots)F^{j}_{m}.\]
with $F^{j}_{0}=\tilde F^{j}_{0}+\frac{\mu_j}{2}Y^{j}$ and $\<Q_{j+1},\overline{F^{j}_{0}}\>\not\equiv0$.
Moreover, there exists a unique light-like section $Y^{j+1}\in\Gamma(f^j_0)$ with  $\<Y^{j+1},F^{j}_0\>=0$ and $\<Y^{j},Y^{j+1}\>=-1$.
This is the second gauge to get $Y^{j+1}$.\vspace{2mm}

Since $Y^{j+1}\perp F^{j}_{0}$ and $\<\tilde F^{j}_{0},\bar{Q}_{j+1}\>\neq0$, we have $Y^{j+1}\perp Q_{j+1}$ and  hence $\{Y^{j}, Y^{j+1}\}=f^{j}_{0}\cap \hp_{j+1}$.


Finally, consider $Q_{j+2}$. Since $Y^j$ has a Frenet-bundle structure $\{f^j_0,\ZZ^j,h_0\}$ as proved in (1), \eqref{eq-moving-2} holds for $Y^j$. By \eqref{eq-Qjz}, \eqref{eq-Emj-F} and \eqref{eq-moving-2} for $Y^j$, we get
\[\begin{split}
    Q_{j+2}&=(Q_{j+1})_{\bar z}-(\cdots )Q_{j+1}=(\cdots) F^{j}_{0}+(\cdots)F^{j}_{1}+\cdots+(\cdots)F^{j}_{m}+(\cdots)Y^{j+1}+ \vartheta_{j}Y^j.\\
\end{split}\]
Since $\<Q_{j+2},Y^j\>\nt0=\<Q_{j+2},Q_{j+2}\>$ by $Y^j\in\Gamma(h_{j+1}^{\perp})\setminus \Gamma(h_{j+2}^{\perp})$, $\vartheta_{j}=0$ and we get \eqref{eq-Ql+2-F}.
\vspace{2mm}

(3)
Since $Y^j$ is strictly $j$-isotropic with $j\geq1$, we have  $\partial f^j_0\subset \ZZ^j$ by (1) of Lemma \ref{lemma-hj-sheaf}. Since $Y^{j+1}\in f^j_0\cap h_{j+2}^{\perp}$, we have $Y^{j+1}_z\in \Gamma((f^j_0\otimes\C)\oplus\ZZ^j)$ with $\<Y^{j+1}_z,Y^{j+1}\>=0$, that is, \[Y^{j+1}_z= (\cdots)Y^{j+1}+(\cdots) F^{j}_{0}+\theta_j\overline{ F^{j}_{0}}+(\cdots)F^{j}_{1}+\cdots+(\cdots)F^{j}_{m}.\]
Since $\<Q_{j+1},\overline{F^{j}_{0}}\>\not\equiv0$ by \eqref{eq-Ql1-Fb}, and also since  $\<Y^{j+1},(Q_{j+1})_z\>=-\<Y^{j+1}_z,Q_{j+1}\>=\<Q_{j+1},\overline{F^{j}_{0}}\>\theta_j=0$, we have $\theta_j=0$, that is,
\begin{equation}\label{eq-Yj+1z}
Y^{j+1}_{z}=(\cdots) Y^{j+1}+(\cdots)F^{j}_{0}+(\cdots)F^{j}_{1}+\cdots+(\cdots)F^{j}_{m}.
\end{equation}
So $Y^{j+1}$ is an adjoint transform of $Y^j$ by Definition \ref{def-adj}.

Now we assume $j\leq m-1$.   Since $\<Y^j, Q_{j+2}\>\neq 0$ by $Y^j\in\Gamma(h_{j+1}^{\perp})\setminus \Gamma(h_{j+2}^{\perp})$, we obtain
 \[(Q_{j+2})_{\zb}=(\cdots) F^j_0+\cdots+(\cdots) F^j_m+(\cdots)Y^{j+1}+(\cdots)Y^{j}+(\cdots)Y^{j+1}_{\zb},\]
with $Y^{j+1}_{\zb}=(\cdots)Y^{j+1}+(\cdots) \bar{F}^j_0+\cdots+(\cdots) \bar{F}^j_m$.
So either  $(Q_{j+2})_{\zb}\nt 0\mod\{Q_{j+2}\}$, or $Y^{j+1}_{\zb}=(\cdots)Y^{j+1}$. We prove $Y^{j+1}_{\zb}\nt(\cdots)Y^{j+1}$
by contradiction. Assume now $Y^{j+1}_{\zb}=(\cdots)Y^{j+1}$. So $[Y^{j+1}]=const$. By Proposition \ref{prop-minimal}, $Y^j$ is M\"obius congruent to a minimal surface in $\R^n$ and hence S-Willmore. It is also strictly $j$-isotropic by inductive assumption (3).
Set $\tf^j_0=(\cdots)F^j_0$ and  set inductively
\[\tf^j_{l+1}=(\cdots)(F^{j}_{0})^{(l+1)}_{z}-(\cdots)\tf^j_0+\cdots+(\cdots)\tf^j_l,~ 0\leq l\leq j-1,\]
so that $\{\tf^j_{l},0\leq l\leq j\}$ is a unitary frame of $Span_{\C}\{F^j_0,(F^j_0)_z,\cdots,(F^j_0)^{(j)}_z\}$.
Then it is straightforward to
 check that
\[\left\{
\begin{split}
    (Y^j)_{z}&=(\cdots)Y^j+\tilde{F}^j_0,\\
    (Y^{j+1})_{z}&=(\cdots)Y^j,\\
    (\tilde{F}^j_l)_{z}&=(\cdots)\tilde{F}^j_l+(\cdots)\tilde{F}^j_{l+1},\ 0\leq l\leq j-1,\\
    (\tilde{F}^j_j)_{z}&=(\cdots)\tilde{F}^j_j+(\cdots)\wx,\\
    (\tilde{F}^j_0)_{\zb}&=(\cdots)\tilde{F}^j_0+(\cdots)Y^j+(\cdots)Y^{j+1},\\    (\tilde{F}^j_l)_{\zb}&=(\cdots)\tilde{F}^j_{l-1}+(\cdots)\tilde{F}^j_{l},\ 1\leq l\leq j,\\
    \xi_{\alpha z}&=(\cdots)\xi_1+\cdots+(\cdots)\xi_{r}-\<\xi_{\alpha},\wx\>\tilde{F}^j_j, \ 1\leq \alpha\leq r,
\end{split}\right.\]
 by \eqref{mov-eq}, since $Y^j$ is strictly $j$-isotropic with a Frenet-bundle structure.
So, $Y^j$ is contained in the space spanned by $\{Y^j,Y^{j+1},Re(\tf^j_{l}),Im(\tf^j_{l}), \xi_{\alpha} ,0\leq l\leq j, 1\leq \alpha\leq r\}|_{z=z_0}$ for some $z_0\in M_j$. This space has dimension $2j+r+4<2m+4+r=n+2$. So $[Y^j]$ is contained in some $S^{2j+r+2}\subsetneq S^{n}$, that is, $[Y^j]$ is not full. By Proposition \ref{prop-full} with recursion, we obtain $[Y^k]$ is not full, a contradiction.

As a result, we have $Q_{j+3}=Q_{j+2,\zb}-(\cdots)Q_{j+2}\nt0$ is well-defined and $Y^{j+1}$ is adjoint to $Y^j$ with $[Y^{j+1}]$ not being constant. So $Y^j$ is also adjoint to $Y^{j+1}$ and therefore is contained in $f^{j+1}_{0}$. Assume  {on the contrary} that $Y^{j+1}\in \Gamma(\hp_{j+3})$, then $f^{j+1}_{0}\subset \hp_{j+2}$ by (1) of Lemma \ref{lemma-hj-sheaf} and hence $Y^j\in\Gamma(\hp_{j+2})$, a contradiction to $Y^{j}\in\Gamma(\hp_{j+1})\setminus\Gamma(\hp_{j+2})$. So we have $Y^{j+1}\in\Gamma(\hp_{j+2})\setminus\Gamma(\hp_{j+3})$.  Moreover, by \eqref{eq-moving-2} for $Y^j$, \eqref{eq-Yj+1z} and Lemma \ref{lemma-hj-sheaf-2}, $Y^{j+1}$ is strictly $(j+1)-$isotropic.
\end{proof}\vspace{2mm}

\subsection{Proof of Theorem \ref{thm-inductive} and Theorem \ref{thm-1}}

\begin{proof}[Proof of Theorem \ref{thm-inductive}]
    By Lemma \ref{lemma-yk-initial}, the assumptions of Lemma \ref{lemma-key-inductive} for $j=k$ holds.  Hence we obtain inductively the Willmore sequence $\{Y^j, k\leq j\leq m\}$, with $Y^j$ being strictly $j$-isotropic.

    The left part is to prove that $Y^m$ is M\"obius congruent to a minimal surface in $\R^n$. Note that by Lemma \ref{lemma-key-inductive} and  \eqref{eq-Ql+2-F}, we already have $Y^{m}\in\Gamma(\hp_{m+1})\setminus\Gamma(\hp_{m+2})$ and $Span\{Y^m, Y^{m+1}\}=f^m_0\cap\hp_{m+1}$. We also have
    $Y^m_z\in \Gamma(\hp_m\otimes\C)$.   By \eqref{eq-Ql1-Fb}, we have
    \[Q_{m+1}=(\cdots)F^m_0+(\cdots)F^m_1+(\cdots)+(\cdots)F^m_m.\]
In particular,    $\<Y^m_z,
    Q_{m+1}\>=0$.  Since $\{Q_{m+1},\bar{Q}_{m+1},Y^m,Y^{m+1}\}$ is a basis of $\hp_m\otimes\C$ with  $\<Y^m_z,Y^m\>=0$, we get $Y^m_z=(\cdots) Q_{m+1}+(\cdots)Y^m$ and hence
\begin{equation}\label{eq-Qm+1}
    Q_{m+1}=(\cdots)Y^m_z+(\cdots)Y^m=(\cdots)F^m_0+(\cdots)Y^m=(\cdots)F^m_0,
\end{equation}
from which we obtain
$\hp_{m+1}=Span\{Y^m, Y^{m+1}\}, \hbox{ and hence } Q_{m+2}=(\cdots) Y^{m+1}+(\cdots)Y^m.$
Since $Q_{m+2}$ is isotropic and $\<Y^m,Q_{m+2}\>\nt0$ due to $Y^m\in\Gamma(\hp_{m+1})\setminus\Gamma(\hp_{m+2})$,  we know
\[Q_{m+2}=(\cdots) Y^{m+1}, \hbox{ and hence }
\hp_{m+2}=Span\{Y^{m+1}\}.\]
Since $h_0$ is totally isotropic by Theorem \ref{thm-h0}, $(Q_{m+2})_{\zb}$ is isotropic with $(Q_{m+2})_{\zb}\in\Gamma(\hp_{m+2}\otimes\C)$. Hence
$(Q_{m+2})_{\zb}=(\cdots) Y^{m+1}$. So $Y^{m+1}_z=(\cdots)Y^{m+1}$. By Proposition \ref{prop-minimal}, $[Y^m]$ is M\"obius congruent to a minimal surface in $\R^n$.
\end{proof}
 \begin{proof}[Proof of Theorem \ref{thm-1}]\

 Theorem \ref{thm-1} is a corollary of Theorem \ref{thm-line}, Theorem \ref{thm-h0}, Theorem \ref{thm-initial} and Theorem \ref{thm-inductive}.
\end{proof}

\begin{remark}
Let  $\hat{\kappa}$ denote the Hopf differential of $Y^m$ and $\hat{D}$ its normal connection.  By equation \eqref{eq-Qm+1} and \eqref{eq-moving-2} for $Y^m$,  together with \eqref{eq-Qjz}, we obtain
 \[Q_m=(\cdots)(Q_{m+1})_{z}-(\cdots)Q_{m+1}=(\cdots)\hat{\kappa}\]
 on an open dense subset where $\hat{\kappa}\neq0$.
 By \eqref{eq-moving-2} for $Y^m$, we see that for any $\psi\in\Gamma(\ZZ^m)$
 \[\psi_z=\hat{D}_z\psi_z=(\cdots)\wx\mod\{\ZZ^m\}.\]
 Applying this inductively, together with \eqref{eq-Qjz}, we obtain
\begin{equation}\label{eq-Qm-j}
    Q_{m-j}=(\cdots)\hat{D}_z^{(j)}\hat{\kappa} \mod\{Q_{m-j+1},\cdots, Q_m\}, ~ 1\leq j\leq m-1,
\end{equation}
 on open dense subsets where $\hat{D}_z^{(j)}\hat{\kappa}\neq0 \mod\{Q_{m-j+1},\cdots, Q_m\}$, $1\leq j\leq m-1$.
 So essentially, what we did in the above proofs is to adjust the Frenet-bundle structure of $Y^j$ inductively, with respect to the fixed, canonical, isotropic, Frenet frame associated to the strictly $m$-isotropic minimal surface $Y^m$. In a sum, we have the following theorem.
\end{remark}
\begin{theorem}\label{thm-minimal}
    Let $[Y^m]$ denote the minimal surface in Theorem \ref{thm-1} with $f^m_0$ its conformal Gauss map. Set inductively
$f^m_{j+1}=\DD f^m_{j}, ~ 1\leq j\leq m-1,$
 on open dense subsets of $M_m$ where $\DD f^m_{j}$ is well-defined.  Then
$f^m_j=\hp_{m-j}, ~~ 1\leq j\leq m.$
\end{theorem}
\begin{proof}
It follows from the definition of $h_{m-j}$ and $f^m_j$, together with \eqref{eq-Qm-j}.
\end{proof}

\section{Construction of examples} \label{example}

It is an interesting and highly non-trivial problem to construct concrete Willmore $2$-spheres.
In this section, we will first briefly discuss examples in each class in Theorem \ref{thm-uniqueness}.  We find complete, $k$-isotropic genus $0$ minimal surfaces in $\R^{6k+3}$ with $(6k+4)$ embedded planar ends. Such examples indicate that when $k$ is small and $n$ is large, there might be Willmore sequence $\{y^j, 0\leq j\leq k\}$ such that each $y^j$ is immersed on $S^2$.

\subsection{ A remark on the construction of Willmore $2$-spheres}

For Willmore two-spheres of Case (1) in Theorem  \ref{thm-uniqueness}, one can derive examples by the DPW methods for Willmore surfaces \cite{DoWa1,Wang-3} or by adjoint transforms of some totally isotropic minimal surfaces in $\R^{2m}$ with given special ends \cite{MWW}. In this way, one can construct totally isotropic Willmore two-spheres which are not S-Willmore in $S^{2m}$ when $m\geq3$.

For Case (2) we refer to Bryant \cite{Bryant1984,Bryant1988}, Kusner \cite{kusner}, Montiel \cite{Montiel}, Peng-Xiao \cite{PX} and Section \ref{ss-ex} for the construction of examples and non-existence results on minimal surfaces with embedded planar ends.

For Willmore two-spheres of Case (3), the construction of examples in \cite{MWW} is helpful. It turns out to be very difficult to construct a concrete example  of strictly $j$-isotropic immersed Willmore two-spheres in $S^n$, which should come from $(m-j)$-step adjoint transforms of a strictly $m$-isotropic branched minimal surface in $\R^n$ for any triple $(n,m,j)$ satisfying $n>2m+2$ and $m>j\geq0$. We refer to \cite{Bryant1984,Bryant1988,MWW,PX} for more works on the construction of minimal surfaces with given properties.

It is better and easier to start with branched strictly $m$-isotropic minimal surfaces with special planar ends. In fact, given such a surface $x^m$ in $\R^n$, since its dual surface is the point at infinity, the Riccati equation associated to $[Y^m]=[(\frac{1+|x^m|^2}{2},\frac{1-|x^m|^2}{2},x^m)]$ has a concrete solution \cite{MWW}. Therefore, one can solve the the Riccati equation for $Y^m$ explicitly. In this way, we obtain all possible adjoint surfaces of $Y^m$. Let $Y^{m-1}$ be one of them. Similarly, since $Y^m$ is an adjoint transform of $Y^{m-1}$, we can solve the Riccati equation for $Y_{m-1}$ explicitly and then we obtained all possible adjoint surfaces of $Y^{m-1}$. Inductively we get an adjoint Willmore sequence
$\{Y^j, 0\leq j\leq m\}$, with $Y^j$ possibly branched.
This provides an algorithm to construct examples of $Y^j$ associated to the triple $(n,m,j)$. Finding an immersed $Y^j$ would be a challenging problem associated to algebraic geometry and complex analysis as in \cite{Bryant1988, MWW}, which we will leave for future study. The examples in the next subsection can be a starting point.

\subsection{$k$-isotropic minimal surfaces in $\R^{6k+3}$ with embedded planar ends}\label{ss-ex}

Building on the construction of minimal surfaces by Bryant \cite{Bryant1984}, Peng-Xiao \cite{PX}, and Ma-Wang-Wang \cite{MWW}, we present new examples of $k$-isotropic minimal surfaces with embedded planar ends.

\begin{theorem}\label{thm-exam}
There exists a $k$-isotropic, genus $0$ minimal surface $x$ in $\R^{6k+3}$ with $(2m+2)$ embedded planar ends, where $m\geq 3k+1$.
\end{theorem}
\begin{remark}
    Note that when $k\geq1$, by the same method as in \cite{MWW}, we can show there exist an adjoint surface $Y^{k-1}$ of the minimal surface in Theorem \ref{thm-exam} which is immersed on the whole $S^2$. This means that  there exists a sequence of immersed Willmore $2$-spheres $\{y^0,y^1\}$ in $S^9$.
\end{remark}

To prove Theorem \ref{thm-exam}, we need the following lemmas. Here we identify $S^2$ with  $\overline{\C}$.
 \begin{lemma} \label{lemma-flat}
 Set $M=\overline{\C}\setminus\{0,\epsilon_j=e^{\frac{2\pi i j}{n}}, 1\leq j\leq n\}.$
 Let $x:M \rightarrow \R^{\tilde{n}}$ be a minimal surface with
    \begin{equation}\label{eq-xz}
        x_z=\frac{\sum_{0\leq j\leq 2n}v_jz^j}{z^2(z^{n}-1)^2},\hspace{6mm} v_j:~\text{coefficient vectors to be determined.}
    \end{equation}
\begin{enumerate}
\item
Set $\lambda_{j,l}:=\<v_j,v_l\>$, $0\leq j, l\leq 2n$.  The conformal condition $\<x_z,x_z\>=0$ reads
\begin{equation}\label{eq-conf}
\sum_{0\leq j\leq l}\lambda_{j,l-j}=0, ~\hbox{if } 0\leq l\leq 2n, ~~\sum_{l-2n\leq j\leq l}\lambda_{j,l-j}=0, ~\hbox{if } 2n<l\leq 4n.
\end{equation}
\item
It has only embedded planar ends, if the following equations hold:
\begin{equation}
    \label{eq-res}
    \begin{split}
        & v_1=0,\ \
         v_{n+j}=\frac{n-j+1}{j-1}v_j, ~2\leq j\leq n-1,\ \
         v_{2n}=\frac{1}{n-1}v_n+\frac{n+1}{n-1}v_0.
    \end{split}
\end{equation}

\end{enumerate}
 \end{lemma}
\begin{proof}
The proof is straightforward.
\end{proof}

 \begin{lemma} Assume $n=2m+1$ and $l=m+1$. Set
 \[
 \tau_{j}=\lambda_{l-j,l+j},~~~0\leq j\leq ,
 ~~~\tau_{s+1}=\lambda_{0,2l},
 ~~~\tau_{s+2}=\lambda_{n,2l},
 ~~\lambda_{i,j}=0~~~ \text{for all other~} (i,j).
 \]
 Then the solution to \eqref{eq-conf} under the conditions \eqref{eq-res} is
  \begin{equation}
        \label{eq-ex-odd-3}
        \begin{split}
&    \tau_{s+1}=-\frac{1}{2}\tau_0-\sum_{1\leq j\leq s}\tau_{j},\ \
   \tau_{s+2}=-\frac{m+1}{m}\tau_0-\sum_{1\leq j\leq s}\left(\frac{m-j+1}{m+j}+\frac{m+j+1}{m-j}\right)\tau_j.\\
\end{split}
    \end{equation}
\end{lemma}

\begin{proof}
We first rewrite \eqref{eq-conf} as follows:
    \begin{equation}
        \label{eq-ex-odd}
        \left\{
        \begin{split}
&    2\lambda_{0,2l}+\lambda_{l,l}+2\sum_{1\leq j\leq s}\lambda_{l-j,l+j}=0,\\
&    \lambda_{n,2l}+\lambda_{l,n+l}+\sum_{1\leq j\leq s}(\lambda_{l-j,n+l+j}+\lambda_{l+j,n+l-j})=0,\\
&    2\lambda_{2l,2n}+\lambda_{n+l,n+l}+2\sum_{1\leq j\leq s}\lambda_{n+l-j,n+l+j}=0.\\
\end{split}\right.
    \end{equation}
By \eqref{eq-res}, we get $v_{2n}=\frac{m+1}{m}v_0+\frac{1}{2m}v_n,~~v_{n+l+j}=\frac{m-j+1}{m+j}v_{l+j}, ~~ 0\leq j\leq s.$
So we have
\[\lambda_{2l,2n}=\frac{m+1}{m}\tau_{s+1}+\frac{1}{2m}\tau_{s+2},~
\lambda_{n+l-j,n+l+j}=\frac{(m-j+1)(m+j+1)}{(m+j)(m-j)}\tau_j,~0\leq j\leq s.\]
\[\lambda_{l,n+l}=\frac{m+1}{m}\tau_0,~ \lambda_{l-j,n+l+j}=\frac{m-j+1}{m+j}\tau_j,~ \lambda_{l+j,n+l-j}=\frac{m+j+1}{m-j}\tau_j, 1\leq j\leq s.\]
Substituting the above equations into \eqref{eq-ex-odd}, we have
     \begin{equation}
        \label{eq-ex-odd-2}
        \left\{
        \begin{split}
&    2\tau_{s+1}+\tau_0+2\sum_{1\leq j\leq s}\tau_{j}=0,\\
&   \tau_{s+2}+\frac{m+1}{m}\tau_0+\sum_{1\leq j\leq s}\left(\frac{m-j+1}{m+j}+\frac{m+j+1}{m-j}\right)\tau_j=0,\\
&    \frac{2(m+1)}{m}\tau_{s+1}+\frac{1}{m}\tau_{s+2}+\frac{(m+1)^2}{m^2}\tau_0+2\sum_{1\leq j\leq s} \frac{(m-j+1)(m+j+1)}{(m+j)(m-j)}\tau_j=0.\\        \end{split}\right.
    \end{equation}
It is straightforward to check that \eqref{eq-ex-odd-3} is a solution to \eqref{eq-ex-odd-2}.
\end{proof}
\begin{lemma}
The minimal surface $x$ is $k$-isotropic, when $\{\tau_j, 0\leq j\leq s+2\}$ is a solution to   both \eqref{eq-ex-odd-3} and the following equations:
    \begin{equation}\label{eq-k-iso}
            \left\{
        \begin{split}
&   a_{i,0}\tau_0+2\sum_{1\leq j\leq s}a_{i,j}\tau_{j}=0,\\
&   b_{i,s+2}\tau_{s+2}+b_{i,0}\tau_0+\sum_{1\leq j\leq s}b_{i,j}\tau_j=0,\\
&    c_{i,s+2}\left(\frac{2(m+1)}{m}\tau_{s+1}+\frac{1}{m}\tau_{s+2}\right)+c_{i,0}\tau_0+2\sum_{1\leq j\leq s}c_{i,j}\tau_j=0,\\        \end{split}\right.
    \end{equation}
    where  $1\leq j\leq s$, $0\leq i\leq k-1$, and
\[
  \begin{split}
         a_{i,0}=&\left(\frac{l!}{(l-i-1)!}\right)^2,~ a_{i,j}=\frac{(l-j)!}{(l-j-i-1)!}\frac{(l+j)!}{(l+j-i-1)!},\\
               b_{i,s+2}=&\frac{n!}{(n-i-1)!}\cdot \frac{(2l)!}{(2l-i-1)!},~ b_{i,0}=\frac{l!}{(l-i-1)!}\cdot \frac{(n+l)!}{(n+l-i-1)!},\\
               b_{i,j}=&\frac{(l-j)!}{(l-j-i-1)!}\frac{(n+l+j)!}{(n+l+j-i-1)!}\frac{n-l-j+1}{l+j-1}  \\           &~~+\frac{(l+j)!}{(l+j-i-1)!}\frac{(n+l-j)!}{(n+l-j-i-1)!}\frac{n-l+j+1}{l-j-1},\\
               c_{i,0}=& \frac{(2l)!}{(2l-i-1)!}\cdot \frac{(2n)!}{(2n-i-1)!},\\
               c_{i,j}=&\frac{(n+l-j)!}{(n+l-j-i-1)!}\frac{(n+l+j)!}{(n+l+j-i-1)!}\frac{(n-l-j+1)(n-l+j+1)}{(l+j-1)(l-j-1)}.\\
            \end{split}\]

\end{lemma}
    \begin{proof}
The proof is straightforward.
    \end{proof}

\begin{proof}[Proof of Theorem \ref{thm-exam}]
For simplicity, we assume that
\[\sum_{0\leq j\leq 2n}v_jz^j=v_0+\sum_{-s\leq j\leq s}v_{l+j}z^{l+j}+v_{n}z^n+v_{2l}z^{2l}+\sum_{-s\leq j\leq s}v_{n+l+j}z^{n+l+j}+v_{2n}z^{2n},\] with $\lambda_{i,j}$ as a solution to both \eqref{eq-ex-odd-3} and \eqref{eq-k-iso}. Here $0\leq s\leq m-1$.  When $s=3k$ and $m\geq s+1$, we will have at least a $1$-parameter family of solutions to both \eqref{eq-ex-odd-3} and \eqref{eq-k-iso}, which are none zero.
Let $\{\ee_{i}| 1\leq i\leq 2s+3\}$ be the canonical basis of $\R^{2s+3}$.
We can set
\[v_{0}=(\cdots)(\ee_1+i\ee_2),~v_{n}=(\cdots)(\ee_1+i\ee_2),~v_{2n}=(\cdots)(\ee_1+i\ee_2),~v_{2l}=(\cdots)(\ee_1-i\ee_2),~ v_{l}=\ee_3,\]
\[ ~~v_{l+j}=(\cdots)(\ee_{2j+4}+i\ee_{2j+5}), v_{l-j}=(\cdots)(\ee_{2j+4}-i\ee_{2j+5}),~1\leq j\leq s,\]
so that both  \eqref{eq-ex-odd-3} and \eqref{eq-k-iso} are satisfied, with none of $\{v_0, v_{2l}, v_{2n}\}$ vanishing.

Moreover, we have $(z^{2m+1}-1)^2x_z\neq0$ when $z\neq0$ since $v_{2l}z^{2l}\neq0$ when $z\neq0$. So $x$ is immersed on $M$  and all of $\{\epsilon_j,1\leq j\leq 2m+1\}$ are embedded planar ends of $x$. Because $v_0\neq0$, $0$ is also an planar end of $x$. Since $v_{2n}\neq0$, $\infty$ is a regular point of $x$.
Thus $x$ is a complete, $k$-isotropic,  genus $0$ minimal surface in $\R^{6k+3}$ with $(2m+2)$ embedded planar ends.
\end{proof}
\begin{remark}\
    \begin{enumerate}
        \item It stays an open question to find the lowest dimension $\tilde n \leq 6k+3$ such that there exists a complete strictly $k$-istropic minimal surface with embedded planar ends in $\R^{\tilde n}$. 
The $k$-isotropy condition for a minimal surface $x$ is equivalent to the requirement that, when viewed as a curve, each of $\{[x_z], \cdots, [x^{(k+1)}_z]\}$ lies in the complex hyperquadric $Q_{n-2}\subset \C P^{n-1}$. This observation implies that the techniques developed in \cite{CXX} could be applied to construct $k$-isotropic minimal surfaces in $\R^n$.

        \item  Consider the simplest case $k=0$. A special solution is to choose $s=0$ and $\tau_0=2m$. So,
$\lambda_{0,2m+2}=-\frac{1}{2}\tau_0,~~\lambda_{n,2m+2}=-\frac{m+1}{m}\tau_0$ and $\lambda_{2n,2m+2}=-\frac{(m+1)^2}{2m^2}\tau_0$.  Set
\[x_z=\frac{1}{z^2(z^{2m+1}-1)^2}\left(
R(z)^2(\ee_1+i\ee_2)
-z^{2m+2}(\ee_1-i\ee_2)+2z^{m+1}R(z))\ee_3\right)
\]
with $R(z)=m+(m+1)z^{2m+1}.$  Then we re-obtain the examples of complete minimal surfaces constructed by Bryant \cite{Bryant1984}, Peng-Xiao \cite{PX}).
    \end{enumerate}
\end{remark}
By Bryant's results \cite{Bryant1988}, there exists no complete, genus $0$ minimal surfaces with $r$ embedded planar ends in $\R^3$ for $r=2,3,5,7$.
As shown in Peng-Xiao's paper \cite{PX}, there exist complete, genus $0$ minimal surfaces with $(2m+1)$ embedded planar ends in $\R^3$ when $m\geq4$.
Moreover, in $\R^4$, similar to Bryant's treatment \cite{Bryant1984,Bryant1988}, by use of the Plücker embedding, Montiel  \cite{Montiel} show that  such minimal surfaces exist and they must be totally isotropic when $r=2$ or $r=3$.
For other cases in $\R^{\tilde n}$, $\tilde n\geq 4$, we have the following  results.

  \begin{proposition}
      \
\begin{enumerate}
\item
Let $x$ be a complete, genus $0$,  minimal surface in $\R^{\tilde{n}}$ with  $2$ or $3$  embedded planar ends.  Then it is totally isotropic with $\tilde n\geq4$.
 \item There exist complete, non-isotropic, genus $0$ minimal surfaces in $\R^4$ with $(2m+1)$ embedded planar ends for $m\geq2$.
\end{enumerate}
  \end{proposition}

 \begin{proof} (1) The case $r=2$ is simple and well-known (see for example \cite{HO}).

 The case $r=3$ is similar. Here we include it for completeness. Now assume w.l.g. that $\{0,\pm 1\}$ are the $3$ ends of $x$. Then we have
     $x_z=\frac{1}{z^2(z^2-1)^2}(v_0+v_1z+v_2z^2+v_3z^3+v_4z^4).$ Since the ends are planar, the vanishing residue conditions \eqref{eq-res} read $v_1=0$ and $v_4=2v_0+v_2$. Set $\lambda_{i,j}=\<v_i,v_j\>$ as before.
     The conformal condition reads
     $\lambda_{0,0}=\lambda_{0,2}=\lambda_{0,3}=\lambda_{2,2}+2\lambda_{0,4}=\lambda_{2,3}=2\lambda_{2,4}+\lambda_{3,3}=\lambda_{3,4}=\lambda_{4,4}=0.$
     By the equation $v_4=2v_0+v_2$, we see $\lambda_{i,j}=0$ for all $i,j$, that is, $x$ is totally isotropic.

(2)  By Lemma \ref{lemma-flat}, we can construct a complete, non-isotropic, minimal surface $x$ with  $(2m+1)$ embedded planar ends as follows:
  \[x=Re\left(\frac{(\frac{m+1}{m-1}z^{2m}-1)(\ee_1+i\ee_2)}{2z(z^{2m}-1)}+\frac{(\ee_1-i\ee_2)}{2mz(z^{2m}-1)}+\frac{z^{m-1}(\ee_3+i\ee_4)}{(m-1)(z^{2m}-1)}+\frac{z^{m}(\ee_3-i\ee_4)}{2m(z^{2m}-1)}\right).\]
 \end{proof}

\vspace{2mm} {\bf Acknowledgements}:   The third author  is grateful to Professor Josef Dorfmeister and Professor Changping Wang for many valuable discussions and constant support. The third author is also grateful to Professor Yijun He and Professor Zizhou Tang for valuable discussions on twistor theory.   The first author is supported by NSFC grant 11831005. Part of this research was done while the second author was supported by the Yau Institute at Tsinghua University.The third author is supported by the NSFC Project 12371052.\\

\def\refname{Reference}


\begin{thebibliography}{9}

\bibitem{Al} Almgren. F.J. {\em  Some interior regularity theorems for minimal surfaces and an extension of Bernstein's theorem.}  Ann. of Math. (2), 84:277-292, 1966.

\bibitem{BB} Babich, M.,  Bobenko, A. {\em Willmore tori with umbilic lines and minimal surfaces in hyperbolic space.} Duke Math. J., 72:151-185, 1993.

\bibitem{Bar}
Barbosa, J. L. M.  {\em On minimal immersions of $S\sp{2}$ into $S\sp{2m}$}, Trans. Amer. Math. Soc. 210 (1975), 75-106.

\bibitem{BK} Bauer, M. and Kuwert, E., {\em Existence of minimizing Willmore surfaces of prescribed genus}.  Int. Math. Res. Not. (2003), 553--576.

\bibitem{Be-Ri}
Bernard, Y., Rivi\`{e}re, T. {\em Energy quantization for Willmore surfaces and applications,}  Ann. of Math. (2) 180 (2014), no. 1, 87-136.

\bibitem{blaschke}
Blaschke, W. {\em Vorlesungen {\"u}ber Differentialgeometrie III: Differentialgeometrie der Kreise und Kugeln,}
Springer Grundlehren XXIX, Berlin, 1929.
\bibitem{Bryant1982}
Bryant, R. {\em Conformal and minimal immersions of compact surfaces into the  $4$-sphere,} J. Differential Geometry 17 (1982), no. 3, 455-473.
\bibitem{Bryant1984}
Bryant, R. {\em A duality theorem for Willmore surfaces,} J. Diff. Geom. 20(1984), 23-53.

\bibitem{Bryant1988}
Bryant, R., {\em Surfaces in conformal geometry,} Proceedings of Symposia in Pure Mathematics 48(1988), 227-240.

\bibitem{Burst}Burstall, F., {\em Harmonic tori in spheres and complex projective spaces}, J. Reine Angew. Math., 1995, 149-178.
\bibitem{BFLPP}
Burstall, F., Ferus, D., Leschke, K., Pedit, F., Pinkall, U. {\em  Conformal geometry of surfaces in $S^{4}$ and quaternions,} Lecture Notes in Math., 1772, Springer-Verlag, 2002.

\bibitem{Bu-Gu}
Burstall, F., Guest, M., {\em Harmonic $2$-spheres in compact symmetric spaces, revisited}, Math. Ann. 309(1997), 541-572.

\bibitem{BPP}
Burstall, F., Pedit, F., Pinkall, U.
{\em Schwarzian derivatives and flows of surfaces,}
Contemporary Mathematics {308}, 2002, 39-61.

\bibitem{BR}
Burstall, F., Rawnsley, J. {\em Twistor theory for Riemannian symmetric spaces,}  Lecture Notes in Math., 1424, Springer-Verlag, 1990.
\bibitem{BW} Burstall, F.,  Wood, J. C. {\em The construction of harmonic maps into complex Grassmannians}, J. Diff. Geom. 23 (1986), 255-298.
\bibitem{Calabi}
Calabi, E., {\em Minimal immersions of surfaces in Euclidean spheres,}
J. Diff. Geom. 1(1967), 111-125.

\bibitem{Chern}
Chern, S. S, {\em On the minimal immersions of the $2$-sphere in
a space of constant curvature}, Problems in analysis, 27-40, Princeton
Univ. Press, Princeton, NJ, 1970.
\bibitem{Chern-W} Chern, S. S., Wolfson, J. G. {\em Harmonic maps of the $2$-sphere into a complex Grassmann manifold II},  Annals of Mathematics, 1987, 125(2): 301-335.
 \bibitem{CXX}
Chi, Q. S., Xie, Z. X., Xu, Y. {\em
Structure of minimal $2$-spheres of constant curvature in the complex hyperquadric,} Adv. in Math., Vol. 391, 2021, 107967.


\bibitem{DPW}
Dorfmeister, J., Pedit, F., Wu, H. {\em Weierstrass type representation of harmonic maps into symmetric spaces,} Comm. Anal. Geom. 6 (1998), no. 4, 633-668.

\bibitem{DoWa1} Dorfmeister, J.,  Wang, P., {\em Willmore surfaces in spheres: the DPW approach via the conformal Gauss map,} Abh. Math. Semin. Univ. Hambg. 89 (2019), no. 1, 77-103.
\bibitem{DoWa12} Dorfmeister, J.,  Wang, P., {\em Weierstrass-Kenmotsu representation of Willmore surfaces in spheres}, Nagoya Math. J. 244 (2021), 35-59.

\bibitem{DoWa13}  Dorfmeister, J.,  Wang, P. {\em A duality theorem for harmonic  maps into non-compact} symmetric spaces and compact symmetric spaces, arXiv:1903.00885, to appear in  Comm. Anal. Geom.

\bibitem{DoWa-fu} Dorfmeister, J.,  Wang, P. {\em Harmonic maps of finite uniton number into inner symmetric spaces via based normalized extended frames}, arXiv:2408.12899, submitted.

\bibitem{ES} Eells, J., Salamon, S. {\em Twistorial construction of harmonic maps of surfaces into four-manifolds,} Ann. Scuola Norm. Sup. Pisa Cl. Sci., (4)124, 589-640 (1985).
\bibitem{EW} Eells, J., Wood, J.C. {\em Harmonic maps from surfaces to complex projective spaces},
Adv. in Math. 49 (1983), no. 3, 217-263.

\bibitem{Ejiri} Ejiri, N., {\em  Willmore surfaces with a duality in $S^N(1)$,} Proc. London Math. Soc. (3) 57(1988), 383-416.

\bibitem{Ejiri-i} Ejiri, N., {\em  Isotropic harmonic maps of Riemann surfaces into the de Sitter space time,} Quart. J. Math. Oxford Ser. (2) 39 (1988), no. 155, 291-306.

\bibitem{Ejiri-D} Ejiri, N., {\em  A Darboux theorem for null curves in $C^{2m+1}$,} pp. 209-210 in Geometry and global analysis (Sendai, 1993), edited by T. Kotake et al., Tohoku Univ., Sendai, 1993.

\bibitem{Fer-03} Fern\'{a}ndez, L. {\em
On the moduli space of superminimal surfaces in spheres,}
Int. J. Math. Math. Sci. 2003, no. 44, 2803-2827.

\bibitem{Fer-12} Fern\'{a}ndez, L. {\em The dimension and structure of the space of harmonic $2$-spheres in the  m-sphere}, Ann. of Math. (2) 175 (2012), no. 3, 1093-1125.


\bibitem{FP}
Ferus, D., Pedit, F. {\em $S^1$-equivariant minmal tori in $S^4$ and $S^1$-equivariant Willmore tori in $S^3$,} Math. Z. 204 (1990), 269-282.

\bibitem{Gr}
Grothendieck, A. {\em Sur la classification des fibr\'es holomorphes sur la sphere de Riemann,} Amer. J. Math. 79.1 (1957), 121-138.


\bibitem{Gu2002}  Guest,  M.A. {\em An update on Harmonic maps of finite uniton number, via the Zero Curvature Equation,} Integrable Systems, Topology, and Physics: A Conference on Integrable Systems  in Differential Geometry (Contemp. Math., Vol. 309, M. Guest et al., eds.), Amer. Math. Soc., Providence, R. I. (2002), 85-113.

\bibitem{Helein} H\'{e}lein, F. {\em Willmore immersions and loop groups,}
 J. Diff. Geom., 50(1998), 331-385.
\bibitem{Helein-B} H\'{e}lein, F.  {\em Constant mean curvature surfaces, harmonic maps and integrable systems.} Lectures Math. ETH Zurich, Birkhauser Verlag, Basel, 2001. 122 pp.
\bibitem{Helein-2004} H\'{e}lein, F.   {\em Removability of singularities of harmonic maps into pseudo-Riemannian manifolds,} Ann. Fac. Sci. Toulouse Math. (6) 13 (2004), no. 1, 45-71.

\bibitem{HO}  Hoffman, D.,  Osserman, R. {\em The geometry of the generalized Gauss map,} Mem. Amer. Math. Soc., 236(1980).

\bibitem{Hopf} Hopf, H. {\em Differential geometry in the large,} Lecture Notes in Math., 1000 Springer-Verlag, Berlin, 1983.


\bibitem{KW} Kapouleas, N., Wiygul, D. {\em The Lawson surfaces are determined by their symmetries and topology,} J. Reine Angew. Math. 786 (2022), 155-173.

\bibitem{kusner} Kusner, R.  {\em Conformal geometry and complete minimal surfaces,} Bull. Amer. Math. Soc.  17(1987), no. 2, 291-295.

\bibitem{Kusner1989} Kusner, R. {\em Comparison surfaces for the Willmore problem.} Pacific J. Math. 138 (1989), 317-345.


\bibitem{KLW} Kusner, R., L\"{u}, Y., Wang, P. {\em
  The Willmore problem for surfaces with symmetry,} arXiv:2410.12582.

\bibitem{KW2} Kusner, R. and Wang, P. {\em Symmetric minimal surfaces in $S^3$ as conformally-constrained Willmore minimizers in $S^n$}. In preparation.

 \bibitem{Ku-Li}  Kuwert, E., Li, Y. {\em $W^{2,2}$-conformal immersions of a closed Riemann surface into  $R^n$}, Comm. Anal. Geom. 20 (2012), no. 2, 313-340.


\bibitem{Ku-Li-Sc} Kuwert, E., Li, Y., Sch\"{a}tzle, R. {\em The large genus limit of the infimum of the Willmore energy,} Amer. J. Math. 132 (2010), no. 1, 37-51.
 \bibitem{Ku-Sch} Kuwert, E.,  Sch\"{a}tzle, R. {\em Removability of point singularities of Willmore surfaces,} Ann. of Math. (2) 160 (2004), no. 1, 315-357.

\bibitem{Lawson}  Lawson, H.B. {\em Complete minimal surfaces in ${\mathbb S}^3$.}     Ann. of Math. (2)    92 (1970), 335-374.

\bibitem{Leschke} Leschke, K., Pedit, F. {\em Sequences of Willmore surfaces,} Math. Z. 259(2008), no. 1, 113-122.

\bibitem{Li-y} Li, P., Yau, S. T. {\em A new conformal invariant and its applications to the Willmore conjecture and the first eigenvalue of compact surfaces.} Invent. Math. 69 (1982), no. 2, 269-291.

\bibitem{ma0} Ma, X., {\em Willmore surfaces in $S^n$: transforms and vanishing theorems,} dissertation,
Technischen Universit\"at Berlin, 2005.

\bibitem{ma1} Ma, X., {\em Isothermic and S-Willmore Surfaces as Solutions to a Problem of Blaschke,}
Results Math. 48(2005), no. 3-4, 301-309.

\bibitem{ma2}
  Ma, X., {\em  Adjoint transforms of Willmore surfaces in $S^n$,}
Manuscripta Mathematica 120(2006), no.2, 163-179.

\bibitem{MPW} Ma, X., Pedit, F., Wang, P. {\em M\"obius homogeneous Willmore $2$-spheres}, Bull. Lond. Math. Soc. 50 (2018), no. 3, 509-512.

\bibitem{MWW} Ma, X., Wang, C. P., Wang, P. {\em
Classification of Willmore two-spheres in the 5-dimensional sphere}, J. Differential Geom. 106 (2017), no. 2, 245-281.

\bibitem{Marques}
Marques, F., Neves, A.,
{\em Min-Max theory and the Willmore conjecture,}
Ann. Math. 179(2014), no. 2, 683-782.

\bibitem{Mi-Ri}
Michelat, A., Rivi\`{e}re, T.
{\em The classification of branched Willmore spheres in the 3-sphere and the 4-sphere,} Ann. Sci. \'{E}c. Norm. Sup\'{e}r. (4) 55 (2022), no. 5, 1199-1288.
\bibitem{Montiel}
Montiel, S., {\em Willmore $2$-spheres in the four-sphere,}
Trans. Amer. Math. Soc. 352(2000), 4469-4486.

\bibitem{Musso}
Musso, E., {\em  Willmore surfaces in the four-sphere,} Ann.
Global Anal. Geom.  8(1990), no.1, 21-41.

\bibitem{PT} Peng, C. K., Tang, Z. Z., {\em Twistor bundle theory and its application,} Sci. China Ser. A 47 (2004), no. 4, 605-616.

\bibitem{PX} Peng, C. K., Xiao, L., {\em Willmore surfaces and minimal surfaces with flat ends,} Geometry and topology of submanifolds, X (Beijing/Berlin, 1999), 259-265, World Sci. Publ., River Edge, NJ, 2000.


\bibitem{Pinkall} Pinkall, U.{\em Hopf tori in $S^3$}, Invent. Math. 81 (1985), 379-386.

\bibitem{Ri} Rivi\`{e}re, T. {\em Analysis aspects of Willmore surfaces}, Invent. Math. 174 (2008), no. 1, 1-45.

\bibitem{Simon93} Simon, L. {\em Existence of surfaces minimizing the Willmore functional}. Comm. Anal. Geom. 1 (1993), no. 2, 281-326.

\bibitem{Uh} Uhlenbeck, K. {\em Harmonic maps into Lie groups (classical solutions of the chiral model),} J. Diff. Geom. 30 (1989), 1-50.

\bibitem{WangCP} Wang, C.P. {\em  M\"obius geometry of submanifolds in $S^{n}$,} manuscripta math., 96 (1998), No.4, 517-534.



 \bibitem{Wang-1} Wang, P., {\em Willmore surfaces in spheres via loop groups II: a coarse classification of Willmore $2$-spheres by potentials}, arXiv:1412.6737.

\bibitem{Wang-3}  Wang, P., {\em Willmore surfaces in spheres via loop groups IV: on totally isotropic Willmore $2$-spheres in $S^6$}, Chinese Ann. Math. Ser. B 42 (2021), no. 3, 383-408.
\bibitem{Wang} Wang, P. {\em Constructing Willmore $2$-spheres via harmonic maps into $SO^+(1,n+3)/(SO(1,1)\times SO(n+2))$,} Monogr. Res. Notes Math. CRC Press, Boca Raton, FL, 2018, 85–117.


\end{thebibliography}
\end{document}